\documentclass[11pt]{amsart}
\usepackage{verbatim, amscd,
  epsfig,amsmath,amsthm,amstext,amssymb,hyperref,float
            }
\usepackage[all]{xy} 
\usepackage{todonotes}

\theoremstyle{plain}
\newtheorem*{theorem*}{Theorem}
\newtheorem{theorem}{Theorem}[section]

\newtheorem{cor}[theorem]{Corollary}
\newtheorem{prop}[theorem]{Proposition}

\theoremstyle{definition}
\newtheorem{definition}[theorem]{Definition}
\newtheorem{example}[theorem]{Example}

\theoremstyle{remark}
\newtheorem{rem}[theorem]{Remark}

\numberwithin{equation}{section}

\renewcommand{\Re}{{\rm Re}\,}
\renewcommand{\Im}{{\rm Im}\,}

\newcommand{\R}{\mathbb{ R}}
\newcommand{\C}{\mathbb{ C}}

\renewcommand{\d}{\mathrm{d}}

\renewcommand{\H}{\mathbb{ H}}

\newcommand{\N}{\mathbb{ N}}
\renewcommand{\P}{\mathbb{ P}}
\newcommand{\HP}{\H\P}
\newcommand{\CP}{\C\P}

\newcommand{\trivial}[1]{\underline{\H}^{#1}}
\newcommand{\ttrivial}[1]{\underline{\widetilde{\H}}^{#1}}

\newcommand{\invers}{^{-1}}
\DeclareMathOperator{\End}{End}
\DeclareMathOperator{\Hom}{Hom}

\DeclareMathOperator{\Gl}{GL}

\DeclareMathOperator{\im}{im}

\DeclareMathOperator{\Span}{span}

\DeclareMathOperator{\Gr}{Gr}

\DeclareMathOperator{\Ad}{Ad}

\newcommand{\Hh}{\mathcal{H}}

\newcommand{\zo}{^{(0,1)}}
\newcommand{\oz}{^{(1,0)}}

\begin{document}
\title[CMC, isothermic and constrained Willmore surfaces]{Links
  between the integrable systems of CMC surfaces, isothermic surfaces
  and constrained Willmore surfaces.}

\author{K. Leschke}
\date{\today}
\address{K. Leschke, School of Computing and Mathematical Sciences,
  University of Leicester, University Road, Leicester LE1 7RH, United
  Kingdom}
 \email{k.leschke@leicester.ac.uk }
 \thanks{}

\maketitle

\begin{abstract}
  Since constant mean curvature surfaces in 3--space are special cases
  of isothermic and constrained Willmore surfaces, they give rise to
  three, apriori distinct, integrable systems. We provide a
  comprehensive and unified view of these integrable systems in terms
  of the associated families of flat connections and their parallel
  sections: in case of a CMC surface, parallel sections of all three
  associated families of flat connections are given algebraically by
  parallel sections of either one of the families. As a consequence,
  we provide a complete description of the links between the simple
  factor dressing given by the conformal Gauss map, the simple factor
  dressing given by isothermicity, the simple factor dressing given by
  the harmonic Gauss map, as well as the relationship to the
  classical, the $\mu$-- and the $\varrho$--Darboux transforms of a
  CMC surface. Moreover, we establish the associated family of the CMC
  surfaces as limits of the associated family of isothermic surfaces
  and constrained Willmore surfaces.
\end{abstract}

\section{Introduction}

Surfaces of non--zero constant mean curvature in Euclidean $3$--space,
so--called \emph{CMC surfaces}, are critical with respect to the
variation of area with prescribed fixed volume. CMC surfaces arise
widely in nature but also in man--made structures. For example, soap
bubbles, equilibrium capillary surfaces (in the absence of gravity),
Fermi surfaces in condensed matter physics (and the related constant
potential surfaces in crystals), block copolymers, as well as freeform
structures or pneumatic architectures such as inflatable domes and
enclosures, can be modeled on CMC surfaces. The geometry of CMC
surfaces has profound consequences for the physical properties of the
materials in which they occur, and thus have a large relevance in
applications in shape optimisation. Due to these applications, finding
computational models of CMC surfaces and understanding their
properties is a very active research area in surface theory.

The study of CMC surfaces has been approached from a variety of
directions.  For example the Euler--Lagrange equation of a CMC surface
is an important example of a geometric non--linear elliptic partial
differential equation and variational principles and analytic gluing
techniques can been used to construct and investigate shapes of CMC
surfaces, e.g., \cite{Wente, kap2, mp, gks3}.  On the other hand, the
theory of integrable systems via the Gauss--Codazzi equation of a CMC
surface, the sinh--Gordon equation, has led to a classification of
constant mean curvature tori by their spectral data, \cite{PS,
  hitchin-harmonic, Bob, EKT, Jaggy}.

Both of these approaches have laid the foundation for finding discrete
approximations of CMC surfaces: variational and integrable system
methods have been successfully used to optimise the energy functional
on a triangulated surface or to find geometric discrete CMC nets by
using the structure of special coordinates and transformations, see
e.g. \cite{brakke, discretecmc-index, finiteelement-cmc,
  discretelawson, discretecmc}.

Here we follow the integrable system approach for various surface
classes: for the purposes of this paper, the core indicator of
integrability is that the underlying Gauss--Codazzi equations of
integrable surface classes give a Lax pair and allow to introduce a
spectral parameter to reformulate the corresponding Lax equation as
the flatness of the associated family of connections. This way, the
non--linear compatibility equations become a system of linear
differential equations.  In the case of a CMC surface, we have three
apriori distinct integrable structures associated to it: CMC surfaces
are special cases of both isothermic surfaces and constrained Willmore
surfaces.

Recall that isothermic surfaces are surfaces which allow (away from
umbilic points) a conformal curvature line parametrisation, the
so--called \emph{isothermal} coordinates.  The Calapso equation, the
fourth order integrability condition for isothermic surfaces
\cite{calapso}, is linked to the zoomeron equation in physics
\cite{boomeron}, and solitonic solutions have been studied in physics
in the context of membrane deformations.

On the other hand, the elastic bending energy of an immersion, the
so--called \emph{Willmore functional}, is an important invariant in
surface theory. Its applications range from the biophysics of
membranes to string theory. \emph{Constrained Willmore surfaces} are
the critical points of the Willmore functional when restricting to
conformal immersions of a fixed Riemann surface. In case of tori,
isothermic surfaces which are constrained Willmore are exactly the CMC
surfaces in space forms \cite{richter}.

The purpose of this paper is to provide a unified view on the
integrable systems of CMC surfaces, isothermic surfaces and
constrained Willmore surfaces.  To this end, we first recall the
Ruh--Vilms theorem \cite{ruh_vilms}: given a conformal immersion
$f: M \to\R^3$ from a Riemann surface $M$ into 3--space with
non--vanishing constant mean curvature the Gauss map of $f$ is a
harmonic map $N: M \to S^2$ into the 2--sphere.  Theoretical
physicists observed that a harmonic map equation is an integrable
system, \cite{pohlmeyer}, \cite{zakharov_mikhailov},
\cite{zakharov_shabat}: by introducing the spectral parameter
$\lambda\in\C_*=\C\setminus\{0\}$, one can associate a $\C_*$--family
of flat connections $d^N_\lambda$ with simple poles at $0,\infty$ to a
CMC surface $f: M \to\R^3$, which is given by the harmonic Gauss map
$N$.  The \emph{associated family of flat connections}
$d_\lambda^N, \lambda\in\C_*,$ on the trivial $\C^2$--bundle uniquely
determines the CMC surface.  This fact can be used to obtain new CMC
surfaces in various ways, and we will now recall those transformations
which are relevant to this paper. We will assume in the following
without loss of generality that the mean curvature is $H=1$.

For spectral parameter $\lambda$ on the unit circle, the connections
$d_\lambda^N$ are quaternionic on $\C^2=\C\oplus j \C = \H$ and give
via the Sym--Bobenko formula \cite{sym_bob} the \emph{associated
  family of CMC surfaces} on the universal cover: gauging the family
of flat connections
$d_{\lambda}^{N^\alpha} = \alpha\invers\cdot d_{\lambda\mu}^N$ by a
$d_\mu^N$--parallel section $\alpha$, $\mu\in S^1$, gives locally a
new associated family of flat connections with harmonic Gauss map
$N^\alpha=\alpha\invers N\alpha$ and thus CMC surfaces $f^\alpha$, see
e.g. \cite{simple_factor_dressing} for details on this in our
setting. Another way to obtain new harmonic maps, and thus CMC
surfaces, was introduced by Uhlenbeck and Terng, \cite{uhlenbeck},
\cite{terng_uhlenbeck}: \emph{dressing} in the case of a harmonic map
$N: M \to S^2$ is given by a gauge
$\hat d_\lambda^N = r_\lambda\cdot d_\lambda^N$ of $d_\lambda^N$ by a
$\lambda$--dependent dressing matrix $r_\lambda$, see
e.g. \cite{simple_factor_dressing}. The \emph{dressing} of $N$ is then
the harmonic map $\hat N$ that has $\hat d_\lambda$ as its associated
family of flat connections. In general, it is difficult to find
explicit dressing matrices preserving the shape of the associated
family of flat connections. However, if $r_\lambda$ has a simple pole
at the \emph{(CMC) spectral parameter} $\mu\in\C_*$ and is given by a
$d_\mu^N$--parallel bundle (possibly on the universal cover $\tilde M$
of the Riemann surface $M$), then a CMC surface $\check f$, the
so--called \emph{simple factor dressing} of $f$, can be computed
explicitly, e.g., \cite{terng_uhlenbeck}, \cite{dorfmeister_kilian},
\cite{simple_factor_dressing}.

Finally, a $d_\mu^N$--parallel bundle can also be used to discuss the
Darboux transformation on CMC surfaces.  Recall first that in the
general case of a conformal immersions the classical Darboux
transformation \cite{darboux} can be generalised
\cite{conformal_tori}: generalised Darboux transforms are given by
prolongations of holomorphic sections of an associated quaternionic
holomorphic line bundle of the conformal immersion.  In the case of a
CMC surface, a $d_\mu^N$--parallel section $\alpha$ gives such a
holomorphic section, and its prolongation, a so--called
\emph{$\mu$--Darboux transform}, is CMC in 3--space for
$\mu\in\C\setminus\{0,1\}$, up to translation in 4--space,
\cite{cmc}. The classical Darboux transforms of a CMC surface which
are CMC too are a special case of the $\mu$--Darboux transformation
and are exactly given by real or unit length spectral parameters
$\mu\in S^1\cup \R\setminus\{0,1\}$. In general, $\mu$--Darboux
transforms are defined on the universal cover and closing condition
can be expressed via the monodromy matrix of the connection $d_\mu^N$.

The relevance of the generalised Darboux transform arises from its
link to the spectral data of conformal torus. Recall that in the case
of a harmonic torus the monodromy representation of the family
$d_\lambda^N$ with respect to a chosen base point on the torus is
abelian and hence has simultaneous eigenlines,
\cite{hitchin-harmonic}. From the corresponding eigenvalues one can
define the \emph{spectral curve} $\Sigma$, a hyperelliptic curve over
$\CP^1$ together with the holomorphic eigenline bundle over $\Sigma$.
Geometrically, in case of a CMC torus, the eigenvalues correspond to
so--called multipliers of $d_\mu^N$--parallel sections and the
spectral curve is essentially the set of all closed $\mu$--Darboux
transforms, \cite{cmc}.  Conversely, the spectral data can be used to
construct all harmonic tori in terms of theta functions on the
spectral curve $\Sigma$.  This approach has been used to classify CMC
tori by spectral data, \cite{PS,Bob}.

The classical Darboux transform was defined for \emph{isothermic
  surfaces}, that is, surfaces which allow a conformal curvature line
parametrisation: a Darboux pair $(f, \hat f)$ is a pair of isothermic
surfaces conformally enveloping a sphere congruence. Darboux
transforms are obtained as solutions to a Riccati equation
\cite{darboux_isothermic} with real parameter
$r\in\R_*=\R\setminus\{0\}$, the \emph{(isothermic) spectral
  parameter}.  This indicates that the class of isothermic surfaces
constitutes an integrable system, as shown in
\cite{cieslinski1995isothermic}.  In particular, again a family of
flat connections $d_r$, $r\in\R$, can be introduced,
\cite{ferus_curved_1996, KamPedPin, burstall_conformal_2010}, which
completely determines the isothermic surface.  Note that in contrast
to the associated family of a harmonic Gauss map the family of
connections of an isothermic surface in our setup is defined on a
trivial $\H^2$--bundle with real spectral parameter $r$. In this
scenario, one can view classical Darboux transforms as the parallel
sections of the connection $d_r$ in the associated family of flat
connections \cite{hertrich-jeromin_mobius_2001, udo_habil}. The
resulting classical Darboux transforms of a CMC surface are only CMC
surfaces if an additional initial condition is satisfied.

Since CMC surfaces are isothermic we investigate the link between the
Darboux transforms given by parallel sections of $d_\mu^N$ and those
given by parallel sections of $d_r$.  To this end we expand in Section
\ref{sec:isothermic} the family $d_r, r\in\R$, to a family $d_\varrho$
with complex parameter $\varrho\in\C$ on the trivial $\C^4$--bundle,
see also \cite{bohle-diss}.  We call the isothermic surface arising
from a $d_\varrho$--parallel section, $\varrho\in\C_*$, a
\emph{$\varrho$--Darboux transform}. A $\varrho$--Darboux transform is
a generalised Darboux transform, and a classical Darboux transform if
and only if $\varrho\in\R_*$. As an example we show that all closed
$\varrho$--Darboux transforms, including all closed classical Darboux
transforms, of a surface of revolution are rotation surfaces in
4--space for $\varrho\not\in \C\setminus\{0, 1\}$ or isothermic
bubbletons, which occur for special spectral parameter, the so--called
\emph{resonance points}.

Since $\mu$--Darboux transforms are CMC but only classical Darboux
transforms for $\mu\in\R\cup S^1\setminus\{0,1\}$ and, on the other
hand, $\varrho$--Darboux transforms of isothermic surfaces are CMC for
special initial conditions, we expect to see a link between the CMC
spectral parameter $\mu$ and the isothermic spectral parameter
$\varrho$. Indeed, we show in Section \ref{sec:cmc} that each
$\mu$--Darboux transform is a $\varrho$--Darboux transform with
isothermic spectral parameter
$\varrho=-\frac{(\mu-1)^2}{4\mu}\in\C_*$. Moreover, this indicates a
2:1 relationship between the corresponding spaces of parallel sections
because the family $d_\lambda^N$ gives flat connections on the trivial
$\C^2$--bundle, whereas $d_\lambda$ are defined on the trivial
$\C^4$--bundle. More precisely, we show that every
$d_\varrho$--parallel section is obtained algebraically from
$d_{\mu_\pm}^N$--parallel sections $\alpha_\pm$ (and the Gauss map $N$
of the CMC surface), where
$\mu_\pm = 1-2\varrho \pm 2i \sqrt{\varrho(1-\varrho)}$. In
particular, the $\varrho$--Darboux transforms which are CMC are
exactly the $\mu_+$-- and $\mu_-$--Darboux transforms. Conversely, all
$d_\mu^N$--parallel sections are algebraically determined from
$d_\varrho$--parallel sections where $\varrho$ is the isothermic
spectral parameter given by $\mu$.

Additionally, the explicit form of the parallel sections of the
associated family $d_\lambda$ in terms of sections of $d_\lambda^N$
allows in Section \ref{sec:cmc} to give a geometric interpretation of
the Sym–Bobenko formula for CMC surfaces as a limit of surfaces in the
isothermic associated family. Here, the associated family in the
isothermic integrable system is analogously to the CMC case given,
\cite{hertrich-jeromin_mobius_2001}, \cite{fran_epos}, by the gauge
$d^\Phi_\lambda=\Phi\invers\cdot d_{\lambda+\varrho}$, where $\Phi$ is
a $d_\varrho$--parallel endomorphism, $\varrho\in\R_*$.

To investigate the simple factor dressing of a CMC surface in terms of
the integrable system of the isothermic surface, we first define in
Section \ref{sec:isothermic} a simple factor dressing by the
$\lambda$--dependent gauge matrix $r_\varrho^E(\lambda)$ for complex
spectral parameter $\varrho\in\C_*$ in terms of a
$d_\varrho$--parallel complex line subbundle
$E\subset \tilde M \times \C^4 = \underline{\tilde{\C}}^4$. If
$\varrho\in\R_*$ our gauge matrix $r_\varrho^E(\lambda)$ with simple
pole at $\lambda=\bar\varrho$ coincides with the dressing matrix given
in \cite{sym-darboux}. Moreover, the family of flat connections
$\hat d_\lambda= r_\varrho^E(\lambda)\cdot d_\lambda$ is the
associated family of the $\varrho$--Darboux transform given by the
twistor projection of $E$, that is, the quaternionic line bundle
$E\H$. In particular, simple factor dressing and $\varrho$--Darboux
transformation coincide in this case.

Furthermore, given two $d_\varrho$--stable complex line bundles
$E_1, E_2$ so that $W \oplus Wj = \tilde M\times\H^2$ for
$W= E_1\oplus E_2$, we define a dressing matrix $r_\lambda^W$ with
simple pole in $\lambda=\bar\varrho$ and a \emph{2--step simple factor
  dressing}. Note that for $\varrho\not\in\R$ the 2--step simple
factor dressing gives surfaces which are not the original isothermic
surface $f$.  Indeed, the 2--step factor dressing is the common
Darboux transform of the $\varrho$--Darboux transforms given by $E_1$
and $E_2$ via Bianchi permutability: in particular, the 2--step
Darboux transform is non--trivial for $\varrho\not\in\R_*$.

To compare to the simple factor dressing of a CMC surface, we recall
\cite{simple_factor_dressing} that a $\mu$--Darboux transform of a CMC
surface $f: M \to\R^3$ is the simple factor dressing of its parallel
surface $g=f+N$ with parameter $\mu$.  Here we show that this aligns
with the 2--step Darboux transformation given via Bianchi
permutability by a $\mu$--Darboux transform of $f$ and the parallel
CMC surface which is a $\mu$--Darboux transform for $\mu=-1$. In
particular, the simple factor dressing via the CMC integrable system
is a 2--step simple factor dressing via the isothermic integrable
system.  This duality in the dressing operation appears as well when
considering the third integrable system of a CMC surface.

A CMC surface is \emph{constrained Willmore}, that is, it is a
critical point of the Willmore energy under variations which preserve
the conformal structure.  As before, one can associate a family of
flat connections $d_\lambda^S$ on the trivial $\C^4$--bundle given by
the conformal Gauss map $S$ for $\lambda\in\C_*$, see
\cite{constrainedWillmore, bohle-diss}. In the case of a CMC surface,
we show in Section \ref{sec:cw} that the family can be expressed in
terms of the associated family $d^{-N}_\lambda$ of the parallel
surface $g= f+N$ and that hence parallel sections of $d^S_\lambda$ are
given exactly by prolongations of parallel sections of
$d^{-N}_\lambda$, up to constant sections. In particular, the
$\mu$--Darboux transforms of a CMC surface are exactly the Darboux
transforms given by parallel sections of $d_\lambda^S$ which are CMC.

Moreover, for suitable complex rank 2--bundles, the simple factor
dressing given by the constrained Willmore property in
\cite{burstall_quintino} gives surfaces with harmonic right normal.
In particular, amongst these simple factor dressings those with
constant real part are indeed the simple factor dressings given by the
CMC integrable system. This sheds a little light on the duality
appearing in the simple factor dressing: the associated family
$d_\lambda^S$ of flat connections in case of a CMC surface is defined
in terms of data of the dual surface. Finally, we show that a surface
in the CMC associated family is obtained as a limit of the associated
family of constrained Willmore surfaces as defined in
\cite{burstall_quintino} up to M\"obius transformation.

Since parallel sections of all three associated families of flat
connections can be algebraically recovered from parallel sections of
either one of the families, we thus have demonstrated a unified view
of on the three distinct integrable systems associated to a CMC
surface and their transformations.

The author would like to thank Fran Burstall, Joseph Cho, Yuta Ogata
and Mason Pember for many fruitful and enjoyable discussions on
isothermic surfaces. This work was partially supported by the Research
Institute for Mathematical Sciences, an International Joint
Usage/Research Center located at Kyoto University, and the author
would like to thank the members of the institute for the warm
reception and support. Additionally, the author gratefully
acknowledges that the work was finalised during a study leave granted
by the University of Leicester.

\section{Preliminaries}

We first give a short summary of results on conformal immersions
$f: M \to\R^4$ from a Riemann surface $M$ into 4--space which are
needed in the following, for details and more results see \cite{icm,
  coimbra, klassiker, nara}.

\subsection{Conformal immersions}
We model $\R^4$ by the quaternions $\H=\Span_\R\{1, i, j, k\}$ where
$i^2=j^2=k^2=-1$ and $ij=k$. The imaginary quaternions
$\Im\H =\Span_\R\{ i, j, k\}$ can be identified with $\R^3$ and for
imaginary quaternions $a, b\in\Im\H$ we have
\[
ab = -<a, b> + a \times b
\]
where we consider $\H = \R\oplus\R^3$, and $<,>$ denotes the standard
inner product and $\times$ the cross product in $\R^3$. In particular,
the 2--sphere in $\R^3$ is given by
$S^2 = \{ n\in\Im\H \mid n^2=-1\}$.

An immersion $f: M \to\R^4=\H$ is conformal if and only if there
exists a pair $(N, R): M \to S^2\times S^2$ with
\[
*df = Ndf = -df R
\]
where $*\omega(X) = \omega(JX)$ for $X\in TM$ and $\omega$ a
1--form. Here $J$ is the complex structure on the tangent space of the
Riemann surface $M$. Identifying the Grassmannian of 2--planes in
$\R^4$ with $\Gr_2(\R^4) = S^2\times S^2$ the pair $(N, R)$ is the
Gauss map of the immersion $f$. In particular, the tangent and normal
space of $f$ are uniquely determined by $(N, R)$ as
\begin{equation}
\label{eq:tangent space}
df_p(T_pM) = \{ v\in\H \mid Nv + vR=0\}, \quad \perp_f = \{ v\in\H
\mid Nv -vR=0\}\,.
\end{equation}
We call $N$ and $R$ the \emph{left} and \emph{right normal} of
$f$. Note that $f: M\to\R^3$ if and only if $N =R$.

The \emph{mean curvature vector} $\Hh$ of $f: M \to\R^4$ satisfies
\[
\overline{\Hh} df = \frac 12(*dR + RdR), \quad df \overline{\Hh} =-
\frac 12(*dN + NdN)
\]
Denoting $H = -\overline{\Hh} N = -R\overline{\Hh}$, the above
equations are equivalent to
\begin{equation}
\label{eq:Hdf}
-Hdf= (dR)', \quad -df H = (dN)'
\end{equation}
where $(dR)' = \frac 12(dR -R*dR), (dN)' = \frac 12(dN-N*dN)$ denote
the $(1,0)$ parts of the 1--forms $dR$ and $dN$ with respect to the
complex structures given by left multiplication by $R$ and $N$
respectively. If $f: M\to\R^3$ then $H: M\to\R$ is the mean curvature
of $f$.

The Willmore energy of a conformal immersion $f: M \to\R^4$ is given
by
\[
W(f) = \int_M(|H|^2 -K-K^\perp)|df|^2
\]
where $K$ is the Gauss curvature of $f$ and $K^\perp$ is the normal
curvature of $f$. If $f: M \to\R^3$ then $K^\perp=0$ and we obtain the
classical Willmore energy of a conformal immersion $f$. In terms of
the Gauss map of $f$, the integrand of the Willmore energy is given as
\[
 (|H|^2 -K-K^\perp)|df|^2= \frac 14|*dR-RdR|^2\,.
\]

Since the classes of isothermic surfaces and of constrained Willmore
surfaces are M\"obius invariant, we can also consider surfaces into
$S^4=\HP^1$. The orientation perserving M\"obius transformations are
given by $\Gl(2,\H)$: $v\H\in\HP^1 \mapsto (Av)\H, A\in \Gl(2,\H)$.
In the following we identify a map $f: M \to\HP^1$ into quaternionic
projective space $\HP^1$ with the line subbundle of the trivial $\H^2$
bundle $\trivial 2= M\times\H^2$ over $M$ which is given by
\[
f(p) = L_p\,.
\]
The derivative of $f$ is given by
$\delta= \pi_L d$ where $\pi_L: \trivial 2 \to \trivial 2/L$ denotes
the canonical projection. 

If $f: M\to\R^4$ then we will denote by $L=\psi\H$ the associated line
bundle of $f$ and by $\infty = e\H$ the point at infinity where
\[
\psi =\begin{pmatrix} f \\ 1
\end{pmatrix}, \quad e=\begin{pmatrix} 1\\ 0
\end{pmatrix}
\]
so that $\trivial 2 = L \oplus \infty$.  Moreover, the left normal
$N: M \to S^2$ induces a complex structure on the line bundle $L$ by
setting $J \psi = -\psi N$ and $f$ is conformal if and only if
\[
*\delta \psi = \delta J\psi
\]
since $\delta\psi = \begin{pmatrix} df \\ 0 
\end{pmatrix}$ when identifying $e\H =\trivial {2}/L$ via the
isomorphism $\pi_L|_{e\H}: e\H \to \trivial{2}/L$.

The conformal Gauss map $S$ of a conformal immersion $f: M \to\R^4$ is
\cite{coimbra} the unique complex structure $S\in\End_\H(\H^2)$ such
that
\[
SL = L, S|_L = J,  dSL\in \Omega^1(L), A\trivial 2 \subset L
\]
where $A$ is the \emph{Hopf field} of $S$, that is, 
\[
A = \frac 14(SdS +*dS)\,.
\]

Then $f$ is a Willmore surface, that is, a critical point of the
Willmore energy, if and only if $S$ is harmonic, that is, $d*A =0$,
see \cite{bryant, coimbra}. The conformal immersion $f$ is called
\emph{constrained Willmore} if there exists
$\eta\in\Omega^1(\Hom(\trivial 2/L, L))$ with $d(*A +\eta)=0$ and
$*\eta = S\eta = \eta S$, see \cite{constrainedWillmore}: a
constrained Willmore surface is a critical point of the Willmore
energy under variations that preserve the conformal structure.

\subsection{Isothermic surfaces}

We first give a short review of known results on isothermic surfaces
which are relevant for our paper. For more historical background and
details see for example \cite{udo_habil, fran_epos, bjpp,
  darboux_isothermic, darboux}.

For the purposes of the paper, it is useful to define isothermic
surfaces in terms of the Christoffel
transformation. 

\begin{definition}[\cite{christoffel, hertrich-jeromin_supplement_1997}]
  Let $f: M \to\R^4$ be a conformal immersion. Then $f$ is called
  \emph{isothermic} if there exists a conformal, branched immersion
  $f^d: \tilde M\to\R^4$ on the universal cover $\tilde M$ of $M$ such
  that
\begin{equation}
\label{eq:dual}
df \wedge df^d = df^d \wedge df =0\,.
\end{equation}
In this case, $f^d$ is called a \emph{dual surface} (or
\emph{Christoffel transform}) of $f$\,.
\end{definition}
Note that (\ref{eq:dual}) implies by type considerations that
$*df^d= -R df^d = df^d N$ where $(N, R)$ is the Gauss map of
$f$. Moreover, since (\ref{eq:dual}) is symmetric in $f$ and $f^d$, we
see that if $f$ is isothermic so is its dual surface $f^d$. The dual
surface is unique up to homothety and translation, except in the case
when the surface is part of a round sphere, see
e.g. \cite{darboux_isothermic}. We will exclude these case in the
following discussions.

Let $z=x+i y$ be a conformal curvature line parametrisation of an
isothermic surface $f$, that is, $z$ is a conformal coordinate
\[
f_y = Nf_x = - f_x R
\]
with $f_{xy}$ tangential, that is
\[
Nf_{xy} + f_{xy} R =0\,.
\]
Then 
\begin{equation}
\label{eq:dual in isothermic coordinates}
df^d = f_x\invers dx - f_y\invers dy
\end{equation}
satisfies
\[
df^d \wedge df = (f_x\invers f_y + f_y\invers f_x) dx\wedge dy = 
(f_x\invers f_y + f_y\invers f_x) dx\wedge dy =0\,.
\]
Here we used that $f_x = f_yR$ and thus $f_x\invers =
-Rf_y\invers$. Similarly, one can show $df\wedge df^d=0$ and, using
the condition that $f_{xy}$ is tangential that $df^d$ is a closed
1--form.  Therefore, we see that there exists locally a dual surface
$f^d$ of $f$: surfaces which admit a (local) conformal curvature line
parametrisation are isothermic.

\begin{example}
\label{ex:isothermic}
We will give some examples of isothermic surfaces which will be
relevant in the following.

\begin{enumerate}
\item 
If $f:  M\to\R^3$ is a  conformal surface of revolution then we can
parametrise 
\[
f(x,y)
=i p(x) + jq(x) e^{-iy}
\] where $(p, q)$ is a unit speed curve in the hyperbolic plane, that
is, $(p')^2+(q')^2= q^2$.  It is a straight--forward computation, see
e.g. \cite{udo_habil}, to show that the 1--form
\[
\omega = - \frac{ip'(x) + jq'(x)e^{-iy}}{q^2(x)}dx - j\frac i{q(x)} e^{-iy}dy
\]
is closed and thus there exists locally a surface of revolution $f^d$
given by $df^d=\omega$.  In particular, $f$ and $f^d$ are isothermic
since $df^d \wedge df = df \wedge df^d=0$.

\begin{figure}[H]
\includegraphics[height=.4\linewidth ]{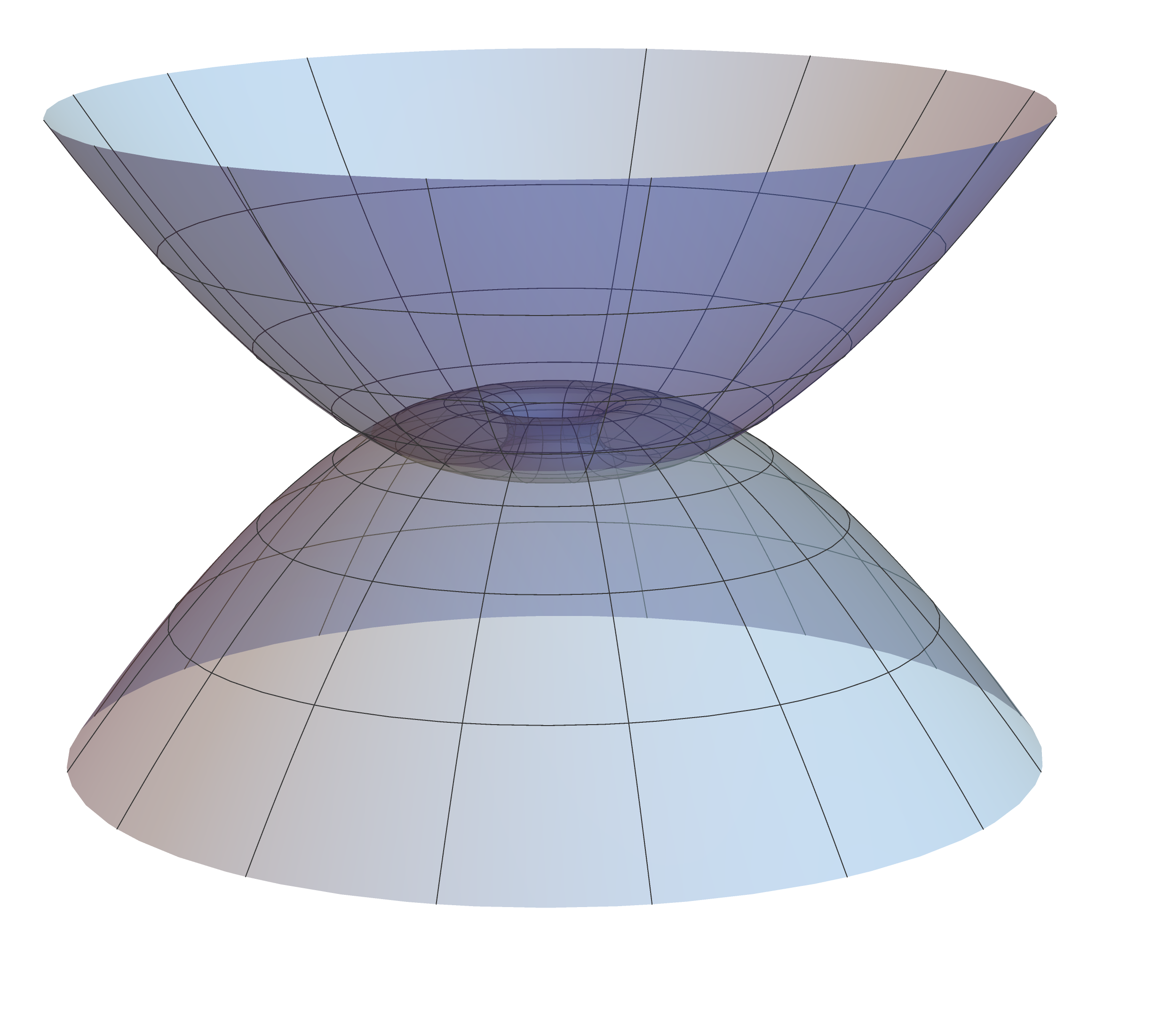}
\includegraphics[height=.4\linewidth]{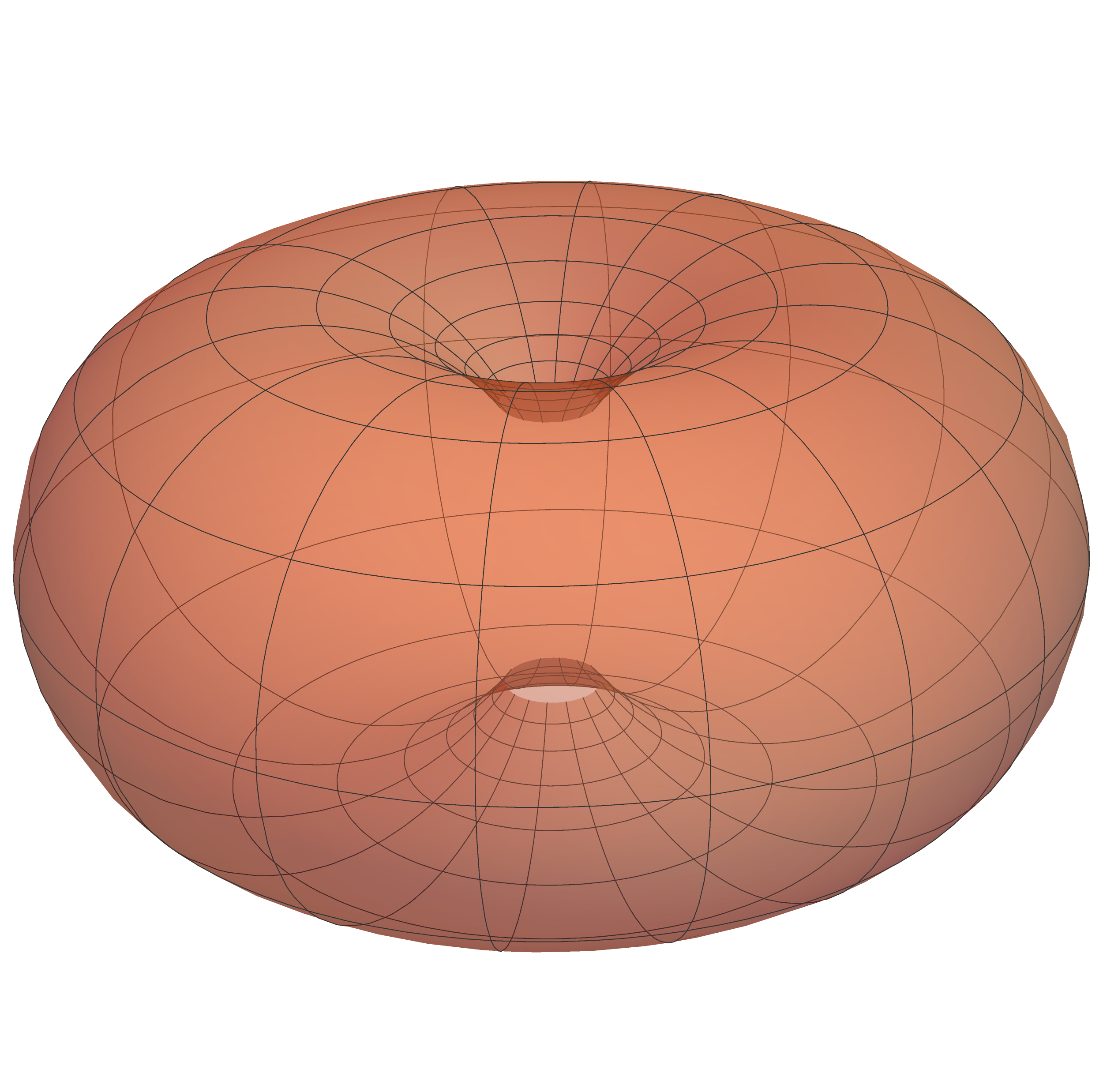}
\caption{Surface of revolution $f(x,y)= i( -x +
  \frac{x^3}3) + j(1+x^2)e^{-iy}$  and its dual surface $f^d(x,y)  =
  \frac {ix}{1+x^2} + j \frac 1{1+x^2} e^{-iy}$.}
\label{fig: surface of revolution}
\end{figure}

\item Another example of isothermic surfaces are CMC surfaces. If
  $f: M \to\R^3$ has non--vanishing constant mean curvature $H$, we
  will always assume without loss of generality that $H=1$.  Then
  (\ref{eq:Hdf}) shows that the parallel CMC surface $g= f+ N$ is a
  dual surface of $f$ (if $f$ is not the round sphere): we have
  $dg = df + dN = \frac 12(dN + N*dN)$ so that $*dg = -Ndg = dg N$.
  Again, this shows that $dg \wedge df = df \wedge dg=0$ and both $f$
  and its parallel surface $g$ are isothermic. Note that in general in
  isothermic coordinates
  $dg = \varrho_0(f_x\invers dx - f_y\invers dy)$ with
  $\varrho_0\in\R_*=\R\setminus\{0\}$ and $g$ is only up to scale (and
  translation) the isothermic surface defined by (\ref{eq:dual in
    isothermic coordinates}). \\

  \noindent In the following we will use the parallel CMC surface
  $g=f+N$ as dual surface, and hence implicitly exclude the round
  sphere by assuming that $g$ is a branched conformal immersion.
\end{enumerate}
\end{example}

\subsection{The classical Darboux transformation}

An alternative way to characterise iso\-thermic surfaces is by the
classical Darboux transformation given by its geometric
properties. For us however it is convenient to describe Darboux
transforms by solutions to a Riccati equation in terms of the dual
surface:
\begin{definition}[\cite{darboux, darboux_isothermic}]
Let $f: M \to\R^4$ be isothermic with dual surface $f^d: \tilde M
\to\R^4$.  If $T: \tilde M \to\R^4\setminus\{0\}$ satisfies the Riccati equation
\begin{equation}
\label{eq:riccati}
dT = -df + T df^d r T 
\end{equation}
for some $r\in\R_*$ then $\hat f = f + T$ is called a
\emph{classical Darboux transform} of $f$.   

\end{definition}
\begin{rem}
If $T$ is not vanishing, but $T(p)=0$ for isolated $p\in M$, we call
$\hat f = f+ T$ a \emph{singular Darboux transform}, see
\cite{conformal_tori}. In this case, $\hat f$ is branched at the
zeros of $T$ since $d\hat f = Tdf^dr T$.

\end{rem}
\begin{figure}[H]

 \includegraphics[height=.3\linewidth]{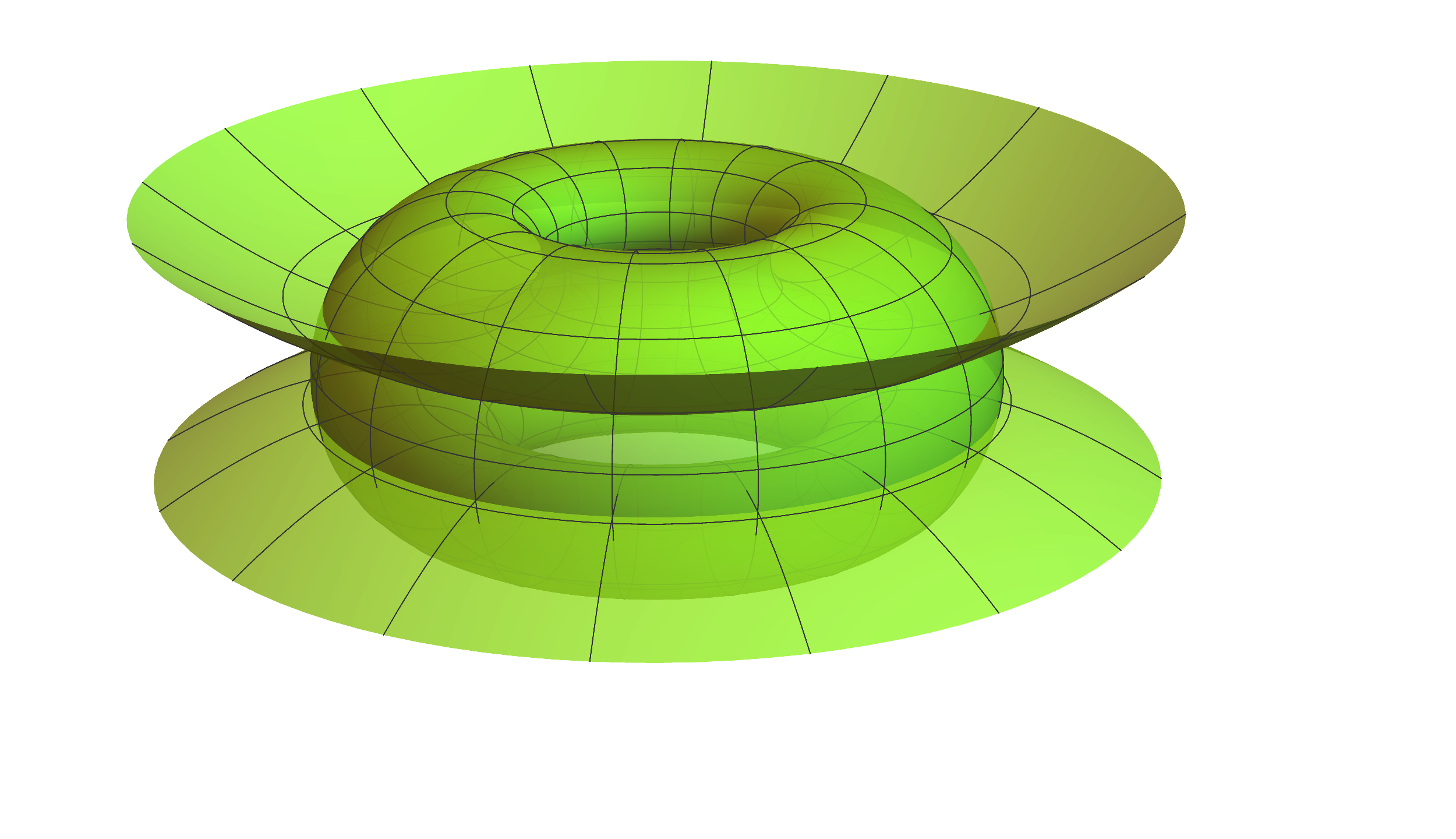}
 \includegraphics[height=.3\linewidth]{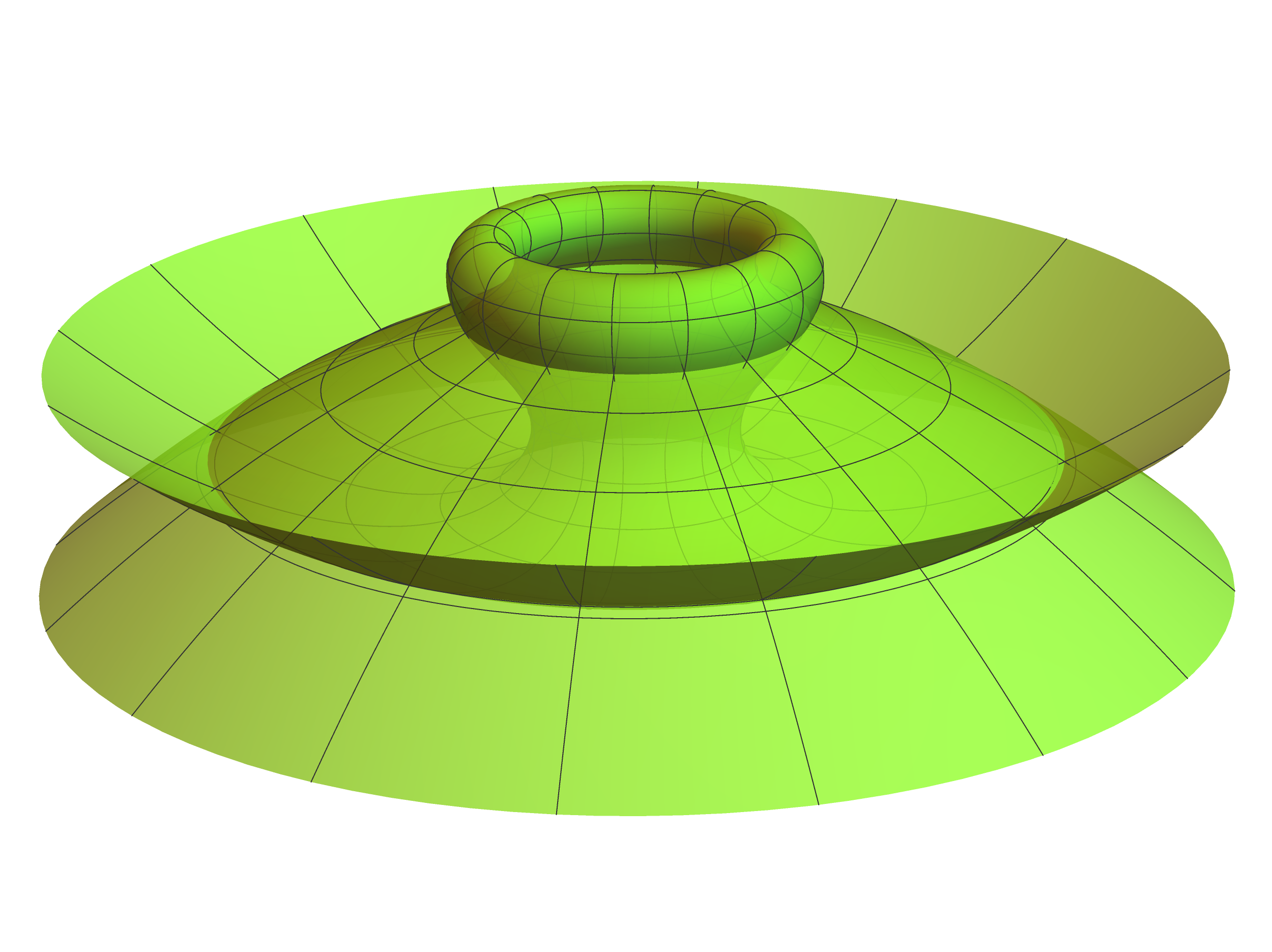}

\caption{Classical Darboux
  transforms of the surface of revolution $f(x,y)= i( -x +
  \frac{x^3}3) + j(1+x^2)e^{-iy}$ at spectral parameter $r=\frac 34$, with different
  initial conditions. We give a complete description of all Darboux
 transforms of a surface of revolution in Example  \ref{ex:rhodt of sor}.} 
\label{fig:cdt}
\end{figure}

Note that a classical Darboux transform $\hat f$ is conformal with
left normal $\hat N = -TRT\invers$ and right normal
$\hat R = - TNT\invers$ since $d\hat f = Tdf ^dr T$ and
$*df^d = -Rdf^d = df^d N$.

Since $dT\invers = -T\invers dT T\invers$ we see that if $T$ satisfies
(\ref{eq:riccati}) then $T^d =\frac 1r T\invers$ satisfies the Riccati
equation $dT^d= -df^d + T^d df r T^d$:

\begin{theorem}[\cite{darboux, darboux_isothermic}]
\label{thm: dual CDT}
If $f^d$ is a fixed dual surface of the isothermic $f$, and
$\hat f =f +T$ is a classical Darboux transform of $f$ with parameter
$r\in\R_*$ then $\hat f^d = f^d + \frac 1r T\invers$ is a classical
Darboux transform of $f^d$ with parameter $r$.  In particular, a
classical Darboux transform $\hat f$ of an isothermic surface is
isothermic with dual surface $\hat f^d$.
\end{theorem}
\begin{proof}
  Since $d\hat f^d = \frac1r T\invers df T\invers$ we see that
  $\hat f^d$ has left normal $\hat N^d =-\hat R$ and right normal
  $\hat R^d =-\hat N$ so that
  $d\hat f^d \wedge d\hat f = d\hat f \wedge d\hat f^d =0$ by type
  arguments.
\end{proof}

Geometrically, a classical Darboux pair $(f, \hat f)$ in $\R^3$ is
determined by a sphere congruence which envelopes both surfaces,
\cite{darboux, udo_habil}: the radius and centres of the spheres are
given by $r(T)=\frac{|T|^2}{2<T,N>}$ and $m=f+ r(T) N$ respectively
where $T = \hat f-f$.  In the case when the Darboux transform $\hat f$
is singular, the enveloping sphere congruence degenerates to a point
at $p\in M$ with $T(p)=0$.

\subsection{The associated family of flat connections}

It is known that the classical Darboux transformation is given by parallel sections of an
associated $\R$--family of flat connections, see
e.g. \cite{udo_habil, burstall_conformal_2010,
burstall_isothermic_2011}: this indicates that isothermic surfaces form  an integrable system as shown in \cite{cieslinski1995isothermic}.  

 For $t\in\R$ the family
\[
d_t = d+ \omega_t, \quad \omega_t = \begin{pmatrix} 0 & df \\ df^d t &0
\end{pmatrix}, 
\]
of connections on $M\times\H^2$ is flat: since
$df\wedge df^d =df^d\wedge df =0$ we see that $d\omega_t=0$ and
$\omega_t \wedge \omega_t=0$.  If for fixed $r\in\R_*$ the section
$\phi=\begin{pmatrix}\alpha \\ \beta
\end{pmatrix}\in\Gamma(\ttrivial{2})$  is a non--trivial $d_r$--parallel
section of the trivial $\H^2$ bundle
$\ttrivial{2}=\tilde M\times \H^2$ over the universal cover $\tilde M$
of $M$, that is,
\[
d\alpha =-df\beta, d\beta =-df^d\alpha r,
\]
then $T=\alpha\beta\invers:\tilde M\to \HP^1$ is a solution to the
Riccati equation (\ref{eq:riccati}) with parameter $r$ and
$\hat f = f+ T$ is a (possibly singular) classical Darboux transform
of $f$. Note that a Darboux transform is closed if and only if
$\varphi$ is a \emph{section with multiplier}, \cite{conformal_tori},
that is, $\gamma^*\varphi = \varphi h_\gamma$ for $\gamma\in\pi_1(M)$,
which is equivalent to $\gamma^*\alpha = \alpha h_\gamma$ since
$d\alpha =-df\beta$ and $\varphi =e\alpha + \psi\beta$.

Conversely, given a solution $T$ to the Riccati equation with
parameter $r$ the connection $d^T = d+T\invers d\hat f$,
$\hat f= f+ T$, is flat. Then any $d^T$--parallel section $\beta$ on
$\tilde M$ gives a $d_r$--parallel section by putting
$\alpha = T\beta$ and $\phi=\begin{pmatrix}\alpha \\ \beta
\end{pmatrix}\in\Gamma(\ttrivial{2})$:
since $T\invers d\hat f \beta = df^dr \alpha$ we have $d\beta
=-df^dr\alpha$ and 
\[
d\alpha = dT \beta + Td\beta
=-df\beta\,.
\]

If the isothermic surface $f: M \to\R^3$ is in 3--space then the
classical Darboux transform is in 3--space with the correct choice of
initial data for the Riccati equation:

\begin{theorem}[see e.g. \cite{udo_habil}]
\label{thm: reality}
Let $f: M \to \R^3$ be isothermic and $r\in\R$.  Then there exists a
3--parameter family of classical Darboux transforms with parameter $r$
so that $\hat f\in\R^3$.
\end{theorem}

Conversely, every family of flat connections of the correct form
determines an isothermic surface. To see this, we will consider a
second family of flat connections, given by the gauge
\[
\d_t= F\cdot d_t = d+ FdF\invers + F\omega_t F\invers, \quad
F = \begin{pmatrix} 1& f\\ 0&1
\end{pmatrix}\,, \quad t\in\R\,,
\] of $d_t$. We see  that $\d_t = d+ t\eta$ with \emph{retraction form}
\[
 \eta= \begin{pmatrix} fdf^d & -fdf^df
  \\ df^d & -df^d f
\end{pmatrix}\,.
\]
  
Note that $\im \eta \subset L \subset \ker\eta$ where $L=\psi\H
= \begin{pmatrix} f\\ 1
\end{pmatrix}
\H$ is the line bundle of $f$, and thus the flatness of $\d_t$
  implies that $d\eta=0$ since $\eta\wedge\eta=0$.

\begin{theorem}[\cite{udo_habil, burstall_conformal_2010,
    burstall_isothermic_2011, sym-darboux}]
\label{thm:retraction-real}
Let $\eta$ be a non--trivial retraction form, that is, 
$\eta\in\Omega^1(\End(\trivial 2))$ such that $\eta^2=0$ and $d\eta=0$. Then
\[
\d_t = d + t \eta\,, t\in\R,
\]
is a $\R$--family of flat connections. Moreover, if the associated line
bundle $L$ of an immersion
$f: M \to \H=\R^4$ satisfies $\im \eta \subset L
\subset \ker\eta$  then $f$ is  isothermic. Conversely, every isothermic
surface $f: M \to\R^4$ arises this way.   
\end{theorem}

\begin{rem} The above theorem holds more generally for surfaces in the
  4--sphere: an isothermic surface in the 4--sphere is the line bundle
  given by the kernel of a non--trivial retraction form. This way we
  obtain a conformal theory. However, since we want to compare the
  integrable systems of a CMC surface, which has an Euclidean theory,
  we stated the theorem here for surfaces in 4--space.
\end{rem}

The $\R$--family of flat connections $d_t$ allows to define a variety
of transforms including the associated family of isothermic
surfaces:

\begin{theorem}[see \cite{hertrich_musso_nicolodi,
  fran_epos}]
\label{thm:iso_ass_fam}
Let $f: M \to\R^4$ be isothermic and $\d_t=d+t\eta$, $t\in\R$, be its
associated family of flat connections. Let $r\in\R_*$ and
$\Phi\in\Gamma(\End_\H(\ttrivial 2))$ with
$(d\Phi)\Phi\invers = -r\eta$. Then
\[
\eta^\Phi =  \Phi\invers\eta\Phi
\]
 is a retraction form.  In particular, if 
\[
\Phi\invers L = \begin{pmatrix} f^\Phi \\ 1
 \end{pmatrix} \H
\] 
then the isothermic surface $A_{\Phi,r}(f) = f^\Phi: \tilde M \to\H$
with associated family $d^\Phi_t = d+ t \eta^\Phi$ is called a
\emph{Calapso transform} of $f$.

Moreover, if $f: M \to\R^3$ then $f^\Phi$ is, up to M\"obius
transformation,  in the 3--sphere.
\end{theorem}

Note that the associated family of $f^\Phi$ is given by
 $\d^\Phi_t = \Phi\invers\cdot \d_{t + r}$ since
$\Phi\invers \cdot \d_r = d$ and  $\d_{t+ r} =
\d_r + t \eta$.  

\begin{figure}[H]
  \includegraphics[height=.3\linewidth]{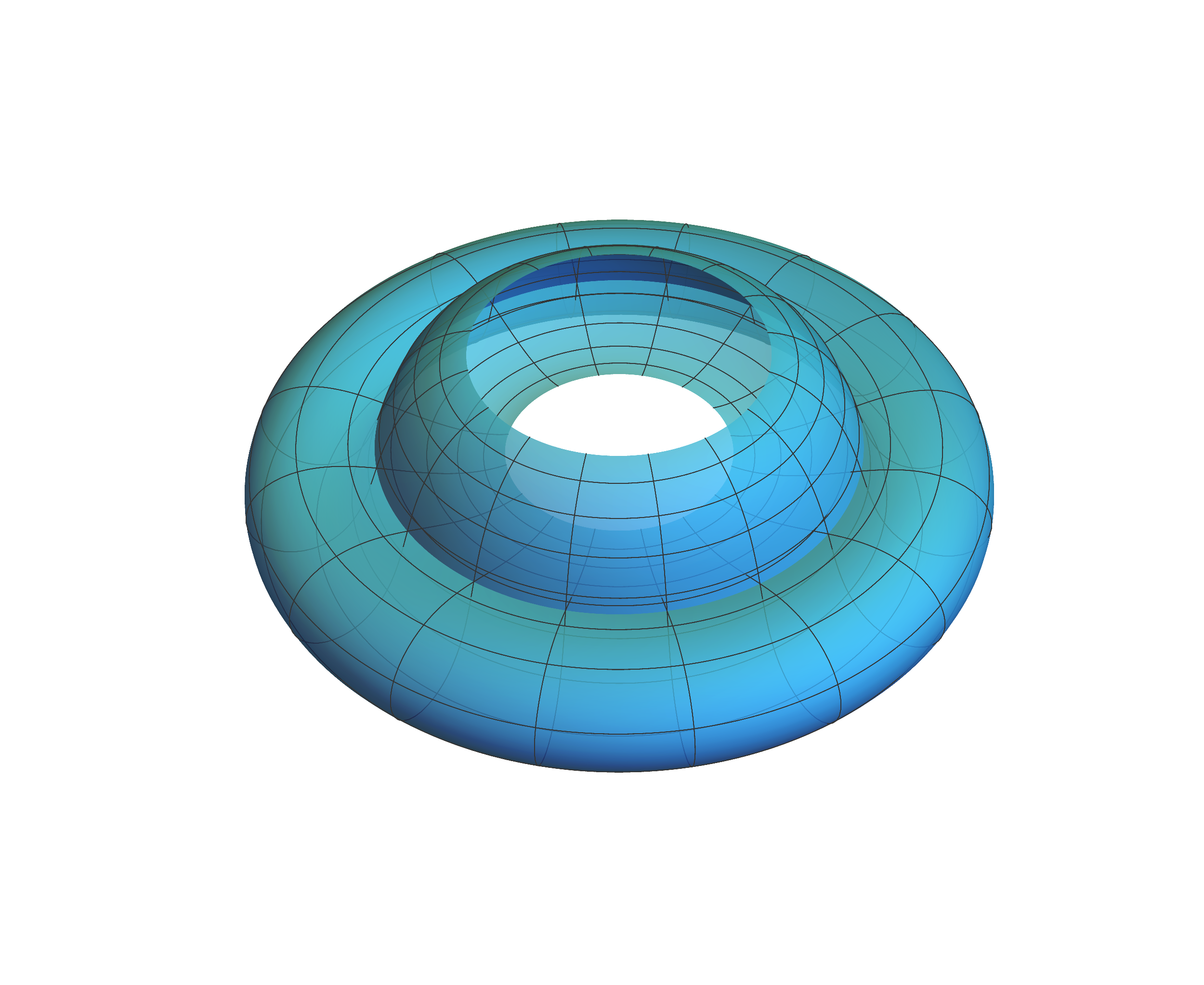}
  \includegraphics[height=.3\linewidth]{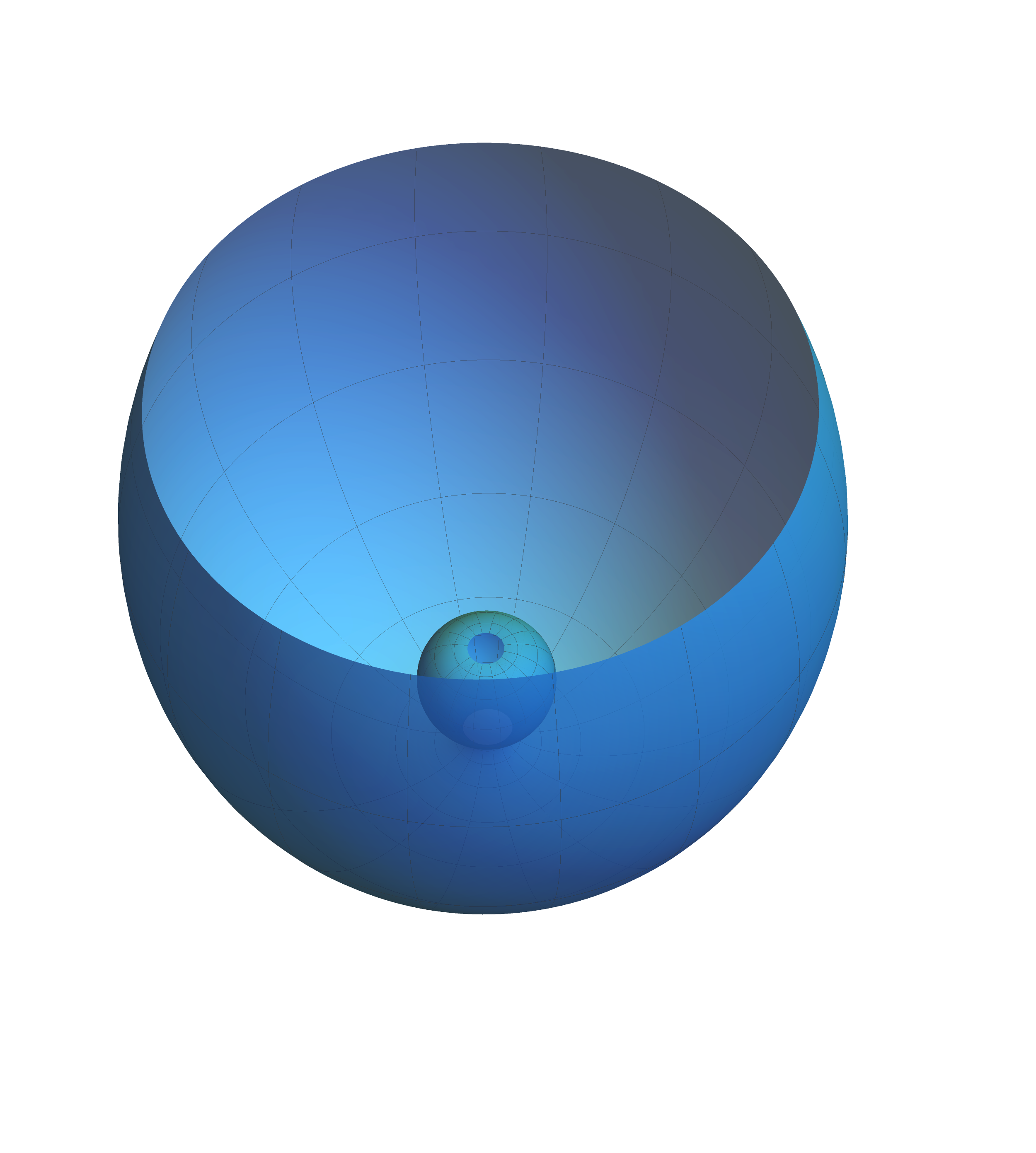}
  \includegraphics[height=.3\linewidth]{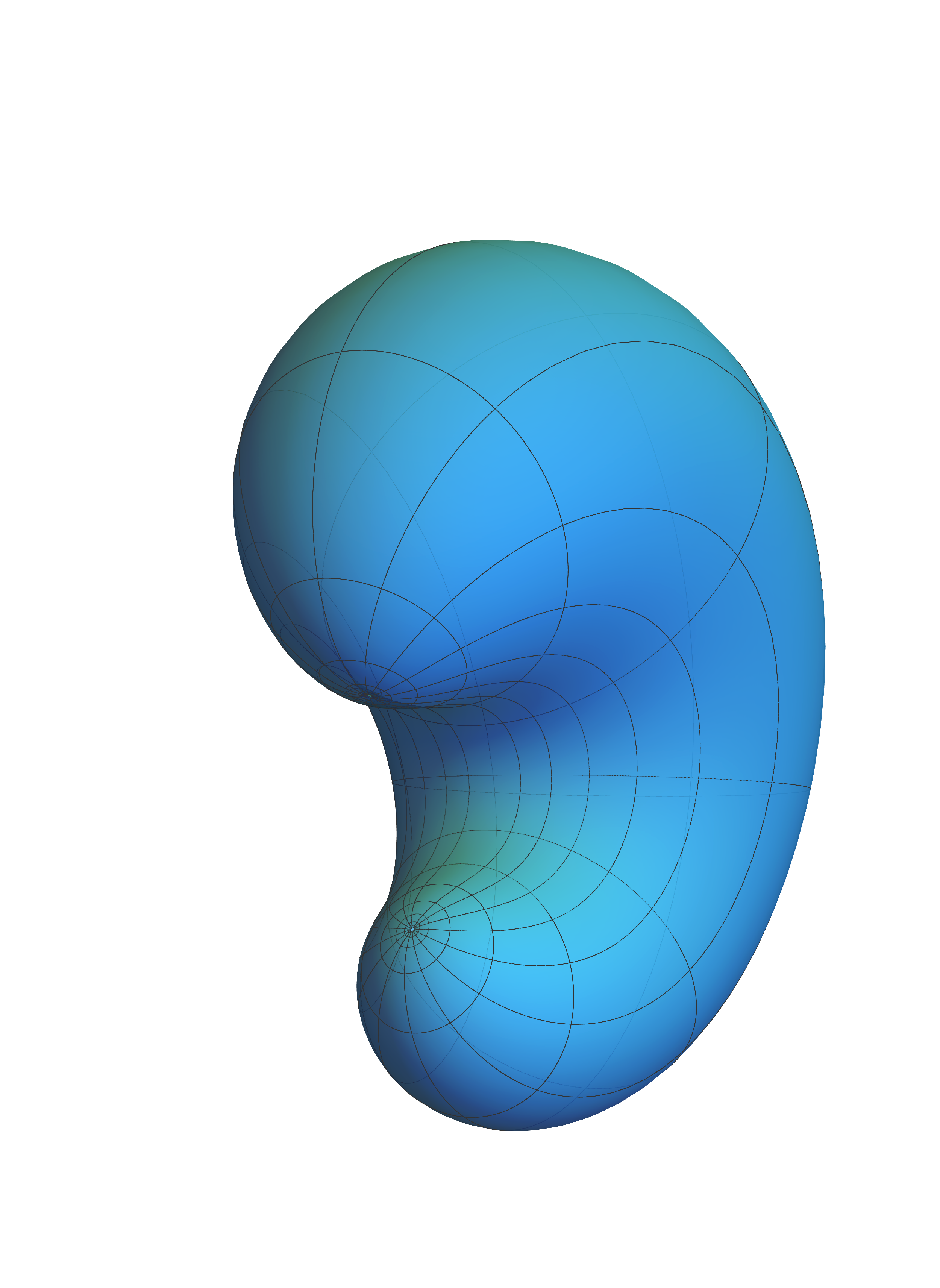}
\caption{Elements of the associated family of the surface of
  revolution $f(x,y)= i( -x +
  \frac{x^3}3) + j(1+x^2)e^{-iy}$, for $r=-\frac 15$, $r=\frac 14$ and
  $r=\frac 34$.}
\end{figure}

\subsection{Holomorphic sections and (generalised) Darboux transformation}

In \cite{conformal_tori} the notion of a Darboux transformation was
extended to give a geometric interpretation of the spectral curve of a
conformal torus as, essentially, the set of its closed Darboux
transforms. This generalised Darboux transformation can geometrically
be understood as a weakening of the enveloping condition but for our
purposes we will use the following definition:

\begin{definition}[\cite{conformal_tori}] Let $f: M \to\H$ be a
  conformal immersion with associated line bundle $L$.  We denote by
  $\tilde L\subset \ttrivial 2$ the induced line bundle on the
  universal cover $\tilde M$ of $M$.

  If $\varphi\in\Gamma(\ttrivial 2)$ satisfies
  $d\varphi\in\Omega^1(\tilde L)$ and
  $\H^2=\varphi_p\H\oplus\tilde L_p$ for all $p\in \tilde M$, then
  $\hat L =\varphi\H$ is called a \emph{Darboux transform} of $f$.
\end{definition}
\begin{rem}
  As before, if $\varphi(p)\in \tilde L_p$ for some $p\in \tilde M$
  then $\hat f$ is called a \emph{singular Darboux transform}.  In the
  case when $f: \tilde M \to\R^3$ this means that the enveloping
  sphere congruence degenerates to a point, and $\hat f$ is branched.
\end{rem}

Darboux transforms are obtained by prolongation of holomorphic
sections: the left normal of an immersion $f: M \to\H$ defines a
quaternionic holomorphic structure $D$, see \cite{klassiker}, on
$\trivial 2/L$ by setting
\[
D(e\alpha) = \frac 12 e(d\alpha+ N*d\alpha)\,.
\]
Here we identify $\infty=\trivial 2/L$ via the isomorphism
$\pi|_{e\H}$ where $\pi: \trivial 2 \to \trivial 2/L$ is the canonical
projection. For any $e\alpha\in H^0(\tilde M\times e\H)=\ker D$ there
exists a unique \emph{prolongation} $\varphi\in\Gamma(\ttrivial 2)$,
that is, $\varphi$ is a lift of $e \alpha$ such that
$\pi d\varphi = 0$. Then $\hat L =\varphi\H$ is a (possibly singular)
Darboux transform of $f$, and every Darboux transform of $f$ arises
this way, \cite{conformal_tori}.

As before, closed Darboux transforms are given by holomorphic sections
with multipliers, i.e., $\gamma^*\varphi = \varphi h_\gamma$, for all
$\gamma\in\pi_1(M)$.

If $f: M \to\R^4$ is isothermic with dual surface $f^d$ then any
$\d_r$--parallel section
$\varphi=e\alpha + \psi\beta \in\Gamma(\ttrivial 2)$, $r\in\R_*$,
defines a holomorphic section: Since $d\alpha =-df\beta$ we see that
$*d\alpha =Nd\alpha$ and $e\alpha$ is a holomorphic section. Moreover,
$\varphi $ is a prolongation of $e\alpha$ since $\pi \varphi= e\alpha$
and $d\varphi = -r\eta\varphi \in\Omega^1(\tilde L)$. But then the
line bundle $\hat L =\varphi\H$ has affine coordinate
$\hat f = f+ \alpha\beta\invers$, and the classical Darboux
transformation is indeed a Darboux transformation in the more general
sense.

The classical Bianchi permutability theorem
\cite{bianchi_ricerche_1905} extends to the more general
transformation:
\begin{theorem}[Bianchi permutability \cite{conformal_tori}]
\label{thm:bianchi}
Let $f: M \to\R^4$ be conformal and $e\alpha_1, e\alpha_2$ be
holomorphic sections with $\alpha_1(p)\not=\alpha_2(p)$ for all
$p\in \tilde M$ and prolongations $\varphi_1$ and $\varphi_2$
respectively. If the Darboux transforms
$L_1=\varphi_1\H, L_2=\varphi_2\H$ are immersed, then there is a
common (possibly singular) Darboux transform $\hat L =\hat\varphi\H$
of $L_1$ and $L_2$ given by
\[
\varphi = \varphi_1 - \varphi_2\chi
\]
where $\chi$ is given by $d\varphi_1 = (d\varphi_2)\chi$. 
\end{theorem}
\begin{proof}
  Since $d\varphi_1, d\varphi_2\in\Omega^1(\tilde L)$ are 1--forms in
  the line bundle $\tilde L$ of $f$ there is $\chi: \tilde M \to\H$
  with $d\varphi_1 =(d\varphi_2)\chi$. Then
  $d\varphi= -\varphi_2d\chi\in\Omega^1(L_2)$ and $\hat L$ is a
  Darboux transform of $L_2$. Similarly, $\hat L$ is a Darboux
  transform of $L_1$.
\end{proof} 
\begin{figure}[H]
\includegraphics[height=.24\linewidth]{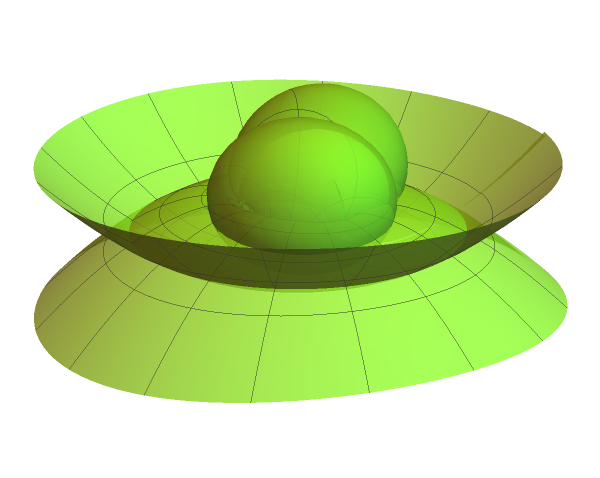}
\includegraphics[height=.24\linewidth]{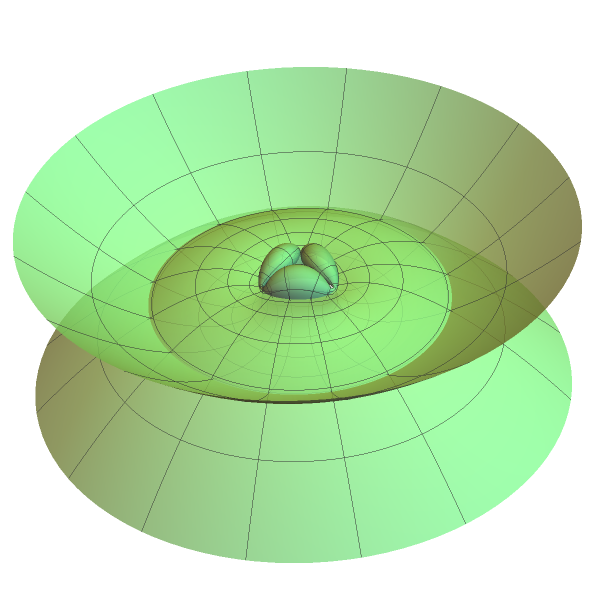}
\includegraphics[height=.24\linewidth]{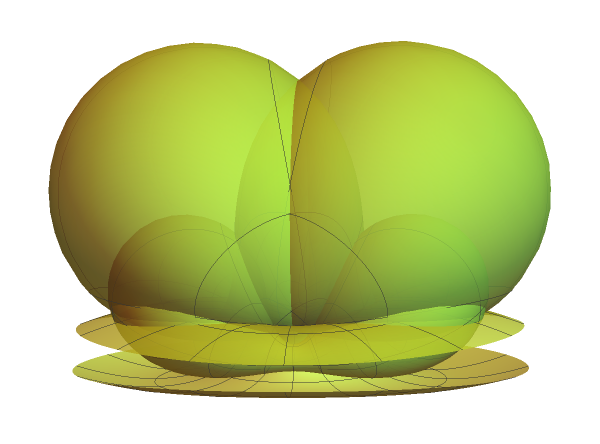}
\caption{The surface on the right is the common Darboux transform of
  the two Darboux transforms on
  the left (with spectral parameter $\frac 34$ and $2$ respectively) of the surface of
  revolution $f(x,y)= i( -x +
  \frac{x^3}3) + j(1+x^2)e^{-iy}$.}
\end{figure}

\section{Isothermic surfaces}
\label{sec:isothermic}

In this paper, we discuss the link between the various integrable
systems of a CMC surface. Since the Gauss map of a CMC surface is
harmonic, there exists a $\C_* $--family of flat connections on the
trivial $\C^2$--bundle, \cite{hitchin-harmonic}, where we denote
$\C_*=\C\setminus\{0\}$. On the other hand a CMC surface is isothermic
and we have seen that every isothermic surfaces gives rise to an
associated family of flat connections $d_t$ with real parameter
$t\in\R$ on the trivial $\H^2$--bundle. To compare these families, we
first extend the family of quaternionic connections $d_t$ to allow for
a complex parameter. We then discuss the link between generalisations
of the classical Darboux transformation and the simple factor dressing
for complex parameters.

\subsection{A $\C$--family of flat connections}

Choosing the complex structure $I$ which operates by right
multiplication by the quaternion $i$ on $\H^2$ we can identify
$(\H^2, I) =\C^4$. In particular, the associated family $d_t$,
$t\in\R$, of the isothermic surface $f: M \to\R^4$ can be seen as a
family of flat connections on the trivial $\C^4$--bundle
$(M\times\H, I)=M\times \C^4$.

This allows to use a complex parameter $\lambda$ instead of the real
parameter $t$ to obtain a $\C$--family of connections on
$\underline{\C}^4=M\times\C^4$ by setting
\[
d_\lambda = d+ \omega_\lambda
\]
where
\[
\omega_\lambda = \begin{pmatrix} 0 & df \\ 0 &0
\end{pmatrix} + \lambda\cdot\begin{pmatrix} 0 & 0 \\ df^d &0
\end{pmatrix}
\]
and $\lambda $ acts by right multiplication by
$\lambda\in\C=\Span_\R\{1, i\}$. Then $d_\lambda$ is flat since
$d\omega_\lambda=\omega_\lambda\wedge \omega_\lambda=0$ where we used
that $df\wedge df^d = df^d\wedge df=0$.  Note that $d_\lambda$ is a
complex connection which is only quaternionic when $\lambda\in\R$.  A
section $\phi =\begin{pmatrix} \alpha\\ \beta
\end{pmatrix}
$ is $d_\varrho$--parallel, $\varrho\in\C_*$, if and only

\begin{align}
\label{eq:iso diff}
\begin{split}
d\alpha &= -df \beta\\
d\beta &= -df^d \alpha\varrho\,.
\end{split}
\end{align}

As before, we consider the gauge by $F$ of $d_\lambda$ to obtain the
family of flat connections as given in \cite{bohle-diss}
\[
\d_\lambda= F\cdot d_\lambda = d+ FdF\invers + F\omega_\lambda F\invers, \quad
F = \begin{pmatrix} 1& f\\ 0&1
\end{pmatrix}\,,
\] so  that $\d_\lambda = d+ \lambda\eta$ with retraction form 
\[
 \eta= \begin{pmatrix} fdf^d & -fdf^df
  \\ df^d & -df^d f
\end{pmatrix}\,.
\]
Note that we again use that $\lambda$ operates by right
multiplication.  As before $\im \eta \subset L \subset \ker\eta$ and
thus the flatness of $\d_\lambda$ implies that $d\eta=0$ since
$\eta\wedge\eta=0$.  Thus, the result in Theorem
\ref{thm:retraction-real} extends from the case $t\in\R$ to the case
$\lambda\in\C$:
\begin{theorem}
\label{thm:retraction}
Let $\eta$ be a non--trivial retraction form, that is, 
$\eta\in\Omega^1(\End(\trivial 2))$ such that $\eta^2=0$ and $d\eta=0$. Then
\[
\d_\lambda = d + \lambda \eta
\]
is a $\C$--family of flat connections. Moreover, if the associated
line bundle $L$ of an immersion $f: M \to \H=\R^4$ satisfies
$\im \eta \subset L \subset \ker\eta$ then $f$ is
isothermic. Conversely, every isothermic surface $f: M \to\R^4$ arises
this way.
\end{theorem}

We conclude this section by giving the affine coordinate of a surface
in the associated family of an isothermic surface.

\begin{theorem}
\label{thm:Ttransform}
Let $f: M \to\R^3$ be an isothermic surface.
For $r\in\R_*$ every Calapso transform $A_{\Phi, r}(f) = f^\Phi:
\tilde M \to\R^3$ of $f$  is given by 
\[
  f^\Phi= -\alpha_1\invers \alpha_2
\]
where $\Phi=(\varphi_1, \varphi_2)$ and
$\varphi_i = e\alpha_i + \psi \beta_i$ are $\H$--linearly independent
$\d_r$--parallel sections.
\end{theorem}

\begin{proof}
  Let $f^\Phi: \tilde M \to \R^4$ be a Calapso transform of $f$ given
  by an invertible $\Phi$ satisfying $(d\Phi)\Phi\invers=-r\eta$, that
  is, $f^\Phi: \tilde M \to \R^4$ is the affine coordinate of
  $\Phi\invers L$. Since $\varphi_i = e\alpha_i+ \psi\beta_i$ are
  $\d_r$--parallel and the sections
  $e\alpha_i\in H^0(\widetilde{\trivial 2/L})$ are holomorphic. Thus,
  see \cite{klassiker}, the holomorphic sections $e\alpha_i$, and thus
  also $\alpha_i$, have isolated zeros.

We write $\Phi =FM$ where
\[
M =(\phi_1, \phi_2) =\begin{pmatrix} \alpha_1 & \alpha_2\\ \beta_1
&\beta_2\end{pmatrix}\,, \quad F = \begin{pmatrix} 1 & f \\ 0&1
\end{pmatrix}\,.
\] 
Then $M$ is invertible and $\phi_1, \phi_2$ are $\H$--independent
$d_r$--parallel sections. Thus, $f^\Phi$ is given by the affine
coordinate of
\[
\Phi\invers L = M\invers \begin{pmatrix} 0 \\ 1
\end{pmatrix}\H\,.
\] 
We now consider $m=\beta_1 - \beta_2\alpha_2\invers \alpha_1$ away
from the isolated zeros of $\alpha_2$. If $m=0$ then
$\beta_1=\beta_2\alpha_2\invers \alpha_1$ shows
$\phi_1 = \phi_2(\alpha_2\invers\alpha_1)$. Since both sections
$\phi_1$ and $\phi_2$ are $d_r$--parallel, this implies that
$\alpha_2\invers\alpha_1$ is constant, contradicting the fact that $M$
is invertible. Thus, $m\not=0$ and since
$d\alpha_i =-df\beta_i, d\beta_i =-df^d \alpha_i r$ we have
$dm = (\beta_2\alpha_2\invers) df m$ so that $m$ is (branched)
conformal with isolated zeros.  Away from the isolated zeros of
$\alpha_2$ and $m$ we have
\[
M\invers = \begin{pmatrix} -m\invers \beta_2\alpha_2\invers & m\invers
  \\ \alpha_2\invers \alpha_1m\invers\beta_1\alpha_1\invers &
  -\alpha_2\invers\alpha_1 m\invers
\end{pmatrix}
\]
and thus, away from the isolated zeros of $\alpha_1, \alpha_2$ and
$m$,
\[
f^\Phi = -\alpha_1\invers\alpha_2\,.
\]
Since by assumption $f^\Phi: \tilde M \to\R^4$ this extends into the
isolated zeros of $\alpha_1, \alpha_2$ and $m$.
\end{proof}
\begin{rem}
  Note that the closing conditions for $f^\Phi$ are now given by left
  multipliers of the holomorphic sections: $f^\Phi$ is closed if and
  only if $\alpha_1, \alpha_2$ have the same left multiplier $h$, that
  is, $\gamma^*\alpha_i = h_\gamma\alpha_i$ for all
  $\gamma\in\pi_1(M)$, $i=1,2$.
\end{rem}

\subsection{The $\varrho$--Darboux transformation}

We will now consider Darboux transforms which arise from parallel
sections of the extended associated family $d_\lambda$,
$\lambda\in\C$. These transforms are again isothermic, see
\cite{bohle-diss}, but still extend the notion of the classical
Darboux transformation:

\begin{definition}
  Let $f: M \to\R^3$ be an isothermic surface with dual surface $f^d$,
  and $\varrho\in \C_*$ fixed. Let
  $\phi=\begin{pmatrix} \alpha \\ \beta
\end{pmatrix}
\in\Gamma(\ttrivial 2)$ be a $d_\varrho$--parallel section such that
\[T=\alpha\beta\invers \colon \tilde M \to\R^4\setminus\{0\}\,.
\]
Then $D_{\phi,\varrho}(f) = f+ T$ is called a \emph{
  $\varrho$--Darboux transform} of $f$.
\end{definition}
\begin{rem}
  We also note that as before, we obtain a \emph{singular Darboux
    transform} if $f(p) =\hat f(p)$ for isolated $p\in \tilde M$.
\end{rem}
We give an independent proof that $\varrho$--Darboux transforms are
isothermic by determining the dual surface of the Darboux transform as
in Theorem \ref{thm: dual CDT}:

\begin{theorem}
\label{thm: rho dt}
Let $f: M \to\R^3$ be an isothermic surface with dual surface $f^d$,
and $\varrho\in \C_*$ fixed. Let $\phi=\begin{pmatrix} \alpha \\ \beta
\end{pmatrix}
\in\Gamma(\ttrivial 2)$
be a $d_\varrho$--parallel section such that  
\[T=\alpha\beta\invers
\colon \tilde M \to\R^4\setminus\{0\}\,.
\]
Then the $ \varrho$--Darboux transform $\hat f =f+ T$ is
isothermic and its dual surface $(\hat f)^d = f^d+ T^d$ is a
$\varrho$--Darboux transform of $f^d$ where 
$T^d = \beta \varrho\invers \alpha\invers$, that is
\[ (D_{\phi, \varrho}(f))^d = D_{\phi^d,\varrho}(f^d)
\]
where $\phi^d = \begin{pmatrix} \beta\varrho\invers \\ \alpha
\end{pmatrix}
$. Moreover, $T$ satisfies
\begin{equation}
\label{eq:gen riccati}
dT = -df+ T df^d\hat \varrho T
\end{equation}
where $\hat \varrho = \alpha\varrho\alpha\invers$.

\end{theorem}

\begin{rem}
  The equation (\ref{eq:gen riccati}) generalises the Riccati type
  equation (\ref{eq:riccati}). Note however, that $\hat \varrho$ is
  not constant and depends on the parallel section $\phi$, that is, we
  cannot find $\varrho$--Darboux transforms by solving this Riccati
  type equation. However, its structure allows to carry over results
  from the theory of classical Darboux transforms to the
  $\varrho$--Darboux transformation.

\end{rem}
\begin{proof}
Note that
a section $\phi =\begin{pmatrix} \alpha \\ \beta
\end{pmatrix}
$ is $d_\varrho$--parallel if and only if
$\varphi = e\alpha + \psi \beta$ is $\d_\varrho$--parallel.  In this
case, $\varphi$ is the prolongation of the holomorphic section
$e\alpha$.  In particular, $\hat L=\varphi\H$ is a Darboux transform
of $f$.

  Next, since $d_\varrho\phi=0$  we have 
\[
d\alpha =-df\beta, \qquad d\beta = -df^d \alpha \varrho
\]
so that
\[
dT = -df+ T df^d\hat \varrho T
\]
where $\hat \varrho = \alpha\varrho\alpha\invers$, that is,
\begin{equation}
\label{eq:dhatf}
d\hat f = Tdf^d\hat\varrho T\,.
\end{equation}

 Putting $\alpha^d = \beta\varrho\invers$ and $\beta^d=\alpha$ 
we obtain a solution of  the equations (\ref{eq:iso diff}) for the
dual surface $f^d$, that is,  
\[
  d\alpha^d = -df^d\beta^d, \quad d\beta^d =-df \alpha^d \varrho\,.
\]
Therefore, $\phi^d = \begin{pmatrix} \alpha^d\\ \beta^d
\end{pmatrix}
$ is parallel with respect to the flat connection $d^d_\varrho$ in the
family of flat connections of the dual surface $f^d$. In particular,
$\widehat{ f^d} = f^d+ T^d$ is a $\varrho$--Darboux transform of $f^d$
with $T^d = \beta \varrho\invers \alpha\invers$ satisfying the Riccati
type equation
\begin{equation}
\label{eq:riccatidual}
dT^d =
-df^d + T^d df \beta \varrho \beta\invers T^d\,,
\end{equation}
and thus
\begin{equation}
\label{eq:dhatfd}
d\widehat {f^d} = T^ddf \beta\varrho\beta\invers T^d\,.
\end{equation}

Now
$ d\hat f\wedge d\widehat{f^d} = T df^d\hat \varrho T \wedge
T^ddf\beta\varrho\beta\invers T^d =0 $ since
$\hat\varrho T T^d =\alpha\varrho\alpha\invers \alpha\beta\invers
\beta \varrho\invers\alpha\invers = 1$ and $df^d\wedge df
=0$. Similarly, $\beta\varrho\beta\invers T^d T=1$ and
$df\wedge df^d=0$ give $ d\widehat{ f^d} \wedge d\hat f = 0.  $ This
shows that $\widehat{ f^d}$ is a dual surface of $\hat f$ and thus
both $\hat f$ and $\widehat{f^d}$ are isothermic.
\end{proof}
\begin{example}
\label{ex:rhodt of sor}
At this point it might be instructive to compute the
$\varrho$--Darboux transforms of a surface of revolution
$f(x,y) = i p(x) + j q(x) e^{-iy}$, $(p')^2+(q')^2=q^2$, for
$\varrho\in\C_*$. We assume that the dual surface is given by
$df^d = f_x\invers dx- f_y \invers dy$.  We observe that
$d\alpha=-df\beta, d\beta =-df^d\alpha \varrho$ give the differential
equation
\[
\alpha_{yy} -i\alpha_y + \alpha \varrho =0
\]
which has solutions  
\[
\alpha^\pm =  e^{\frac {iy}2}c^\pm e^{\pm\frac{isy}2} 
\]
where $s=\sqrt{1+4\varrho}$ and $c^\pm$ are $\H$--valued functions,
independent of $y$.  We observe that $\alpha^\pm$ have multiplier
$h^\pm = -e^{\pm i\pi s}$.

Note that all solutions for $\varrho\not=-\frac 14$ arise as
$\alpha = \alpha_++\alpha_-$. For $\varrho=-\frac 14$ the general
solution is of the form $\alpha = e^{\frac{iy}2}(c_1 + y c_2)$.  Since
we are interested in closed Darboux transforms, that is,
$2\pi$--periodic $\alpha$ in the parameter $y$, it is therefore enough
to consider again $\alpha=\alpha^\pm$ since in this case $s=0$ and
thus $\alpha^+=\alpha^- = e^{\frac{iy}2}c$ for some $\H$--valued,
$y$--independent function $c$.

Since $d\alpha =-df \beta$ we have $*d\alpha = Nd\alpha$ where the
Gauss map $N$ of $f$ is given by
 \[
N(x,y) = \frac 1{q(x)}(iq'(x) - jp'(x)e^{-iy})\,.
\]
To find explicit solutions for $\alpha$, one has to find the complex
functions $c_0^\pm, c_1^\pm$, $c^\pm = c_0^\pm + j c_1^\pm$, by
solving the differential equation $*d\alpha=Nd\alpha$ which is given
for the complex--valued functions $c_0^\pm, c_1^\pm$ as
\[
\begin{pmatrix}c_0^\pm \\ c_1^\pm
\end{pmatrix} ' =\frac 1{2q} \begin{pmatrix} (1\pm s)q'  &
  i(1\mp s) p' \\\ i (1\pm s)p' &(1\mp s)q'
\end{pmatrix} \begin{pmatrix}c_0^\pm \\ c_1^\pm
\end{pmatrix}\,. 
\]
In particular, we obtain a $\C^2$--worth of parallel sections with
multiplier $h^\pm$.

In general, it might not be possible to solve the above differential
equation explicitly but nevertheless we can conclude that the
corresponding closed Darboux transforms are rotation surfaces: Since
$d\alpha=-df\beta$ we compute
\[
\beta ^\pm= \frac 1{2q} e^{\frac{iy}2} \left( c_1^\pm(\pm s -1) + j
  c_0^\pm(1\pm s)\right)e^{\pm\frac{isy}2}\,.  
\]
and hence
\begin{align*}
T^\pm &=\alpha^\pm (\beta^\pm)\invers \\
&= \frac{2 q}{|c_1^\pm|^2|\pm
  s-1|^2+ |c_0^\pm|^2|1\pm s|^2}
\left( \pm 2 c_0^\pm \bar{c_1}^\pm \Re(s) + j \big(|c_1^\pm|^2(\pm \bar s-1)-
|c_0^\pm|^2(1\pm s)\big)e^{-iy}\right)\,.
\end{align*}
Therefore the $\varrho$--Darboux transforms 
\[
f^\pm = f+ T^\pm =
p^\pm + j q^\pm e^{i(\theta_\pm -y)}\,,
\]
where  we define $\theta_\pm$ and $R_\pm$ by  $(\pm s-1)(\pm \bar s+1)=R_\pm
e^{i\theta_\pm}$, 
are rotation surfaces in 4--space since $p^\pm$ is a
complex--valued function in $x$, and 
\[
q^\pm = \frac{R_\pm q(|c_1|^2+|c_0|^2)}{|c_1^\pm|^2|\pm
  s-1|^2+ |c_0^\pm|^2|1\pm s|^2} 
\]
is a real--valued function, independent of $y$.

Since for $\varrho\in\C\setminus\R$ the multipliers $h^\pm$ are not
complex conjugates of each other, we obtain two $\CP^1$--families of
closed Darboux transforms $f^\pm$ by $d_\varrho$--parallel sections
$\varphi^+$ and $\varphi^-$ respectively, where
$\varphi^\pm = e\alpha^\pm + \psi \beta^\pm$. In particular, the only
closed $\varrho$--Darboux transforms for $\varrho\not\in\R$ are
rotation surfaces in 4--space.
\begin{figure}[H]
\includegraphics[height=4.5cm]{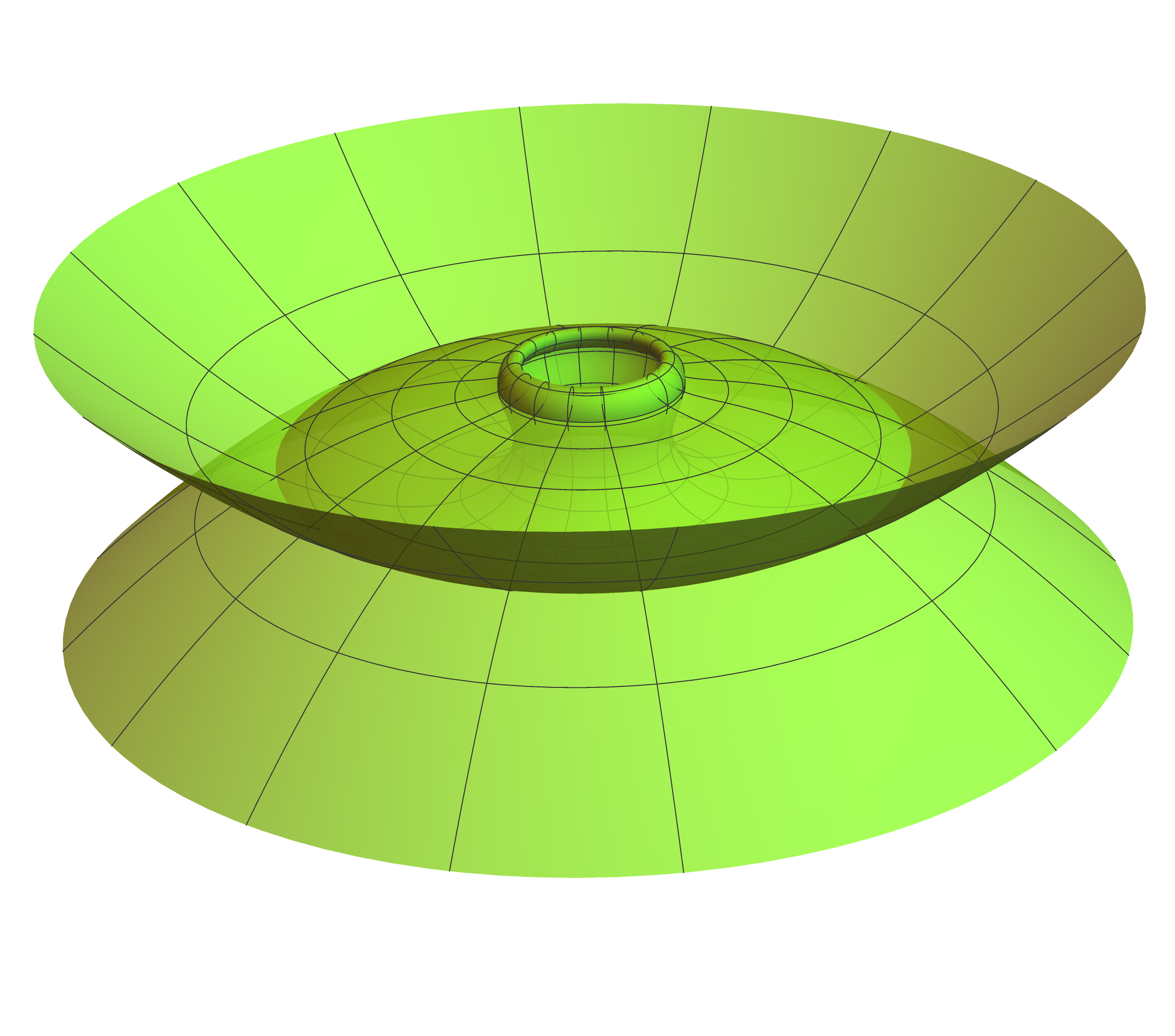} \quad
\includegraphics[height=4.5cm]{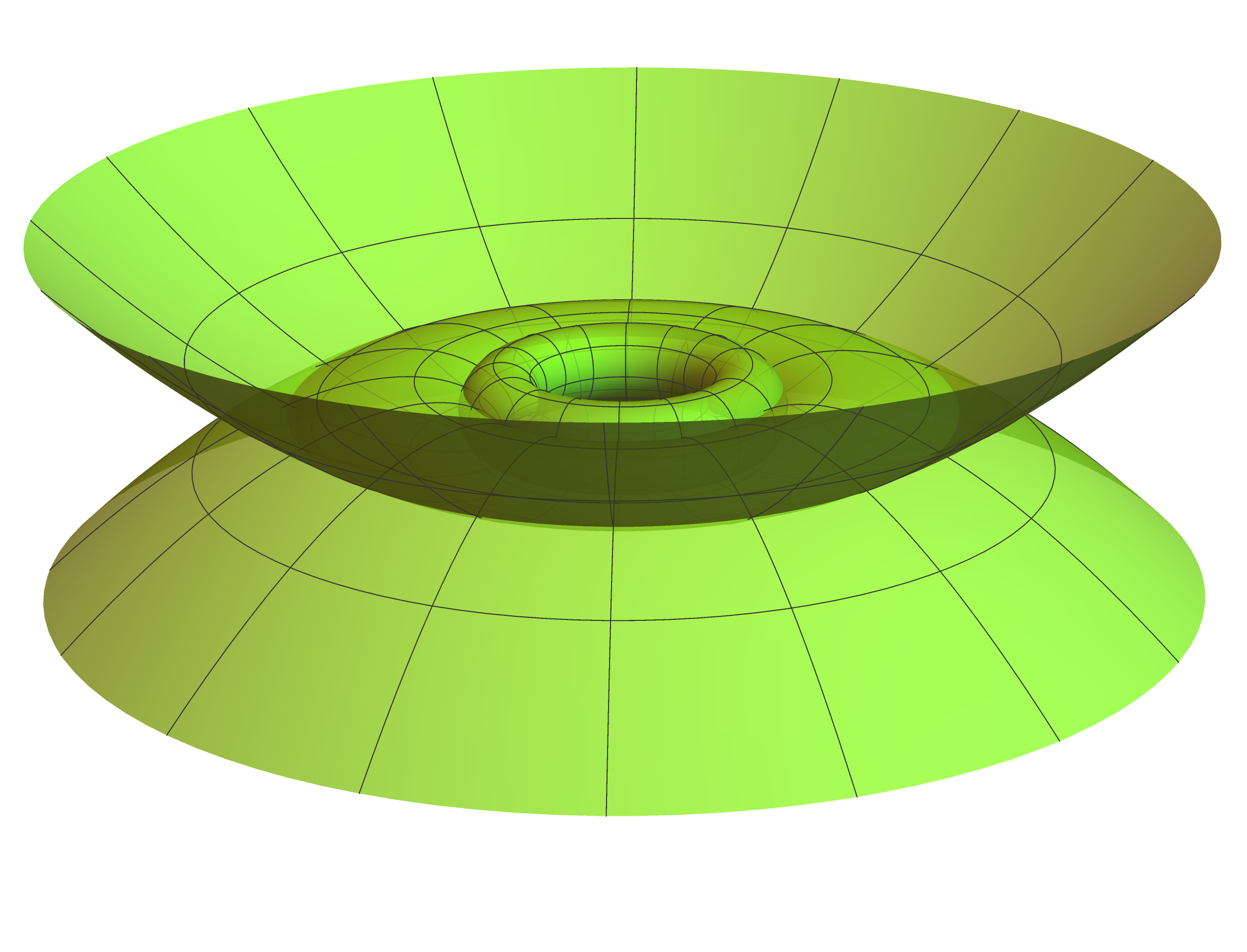}
\caption{$\varrho$--Darboux transforms $f_\pm$ of the surface of revolution
  $f(x,y)= i( -x +
  \frac{x^3}3) + j(1+x^2)e^{-iy}$ for $\varrho=1+i$, both
  orthogonal projections to 3--space of rotation surfaces in 4--space.}
\label{fig: generalised dt}
 \end{figure}

 In the case when $r=\varrho\in\R_*$ we have a similar situation as in
 the case of a cylinder, see e.g. \cite{sym-darboux,cmc}: If
 $r < -\frac 14$ we have $s\in i\R$ so that
 $h^+ = -e^{i\pi s}= \frac 1{h^-}\in\R$, $h^+\not=h^-$, and thus
 $\gamma^*(\alpha^\pm j) =( \alpha^\pm j) h^\pm$. In particular, we
 have exactly two $\H$--independent sections with multiplier and two
 closed Darboux transforms
\[
  f^\pm = f+ \alpha^\pm (\beta^\pm)\invers = ip + jq\frac{\pm
    s-1}{1\pm s} e^{-iy}
\]
which are rotations of $f$ by the angles $\pm \theta$ where
$e^{i\theta} = \frac{s-1}{1+ s}\in S^1$.  In the case when
$r=-\frac 14$ we have exactly one closed Darboux transform which is
$f^+=f^- =ip -jq e^{-iy}$.

For $r>-\frac 14, r\not=\frac{k^2-1}4, k\ge 1,$ the multipliers
satisfy $h^+=\overline{h^-}\in S^1\setminus\{\pm 1\}$. Since
$\varphi^+ j$ has multiplier $h^-$ and $\varphi^+\H = \varphi^+j\H$ we
obtain only one $\CP^1$--family of closed Darboux transforms in this
case which is given by
\[
f^+ = f +  p^+ + j q^+ e^{-iy} 
\]
where $p^+$ and $q^+$ are, as before, complex and real valued
functions in $x$ respectively. Thus, the corresponding Darboux
transforms are again rotation surfaces, see Figure \ref{fig:cdt}.

The final case $r_k=\frac{k^2-1}4, k>1$, gives real multiplier
$h^+=h^-= (-1)^{k+1}$. This means that every parallel section is a
section with multiplier, that is, $r$ is a \emph{resonance point}. The
corresponding Darboux transforms are rotation surfaces or isothermic
bubbletons.

\begin{figure}[H]
 
\includegraphics[height=4cm]{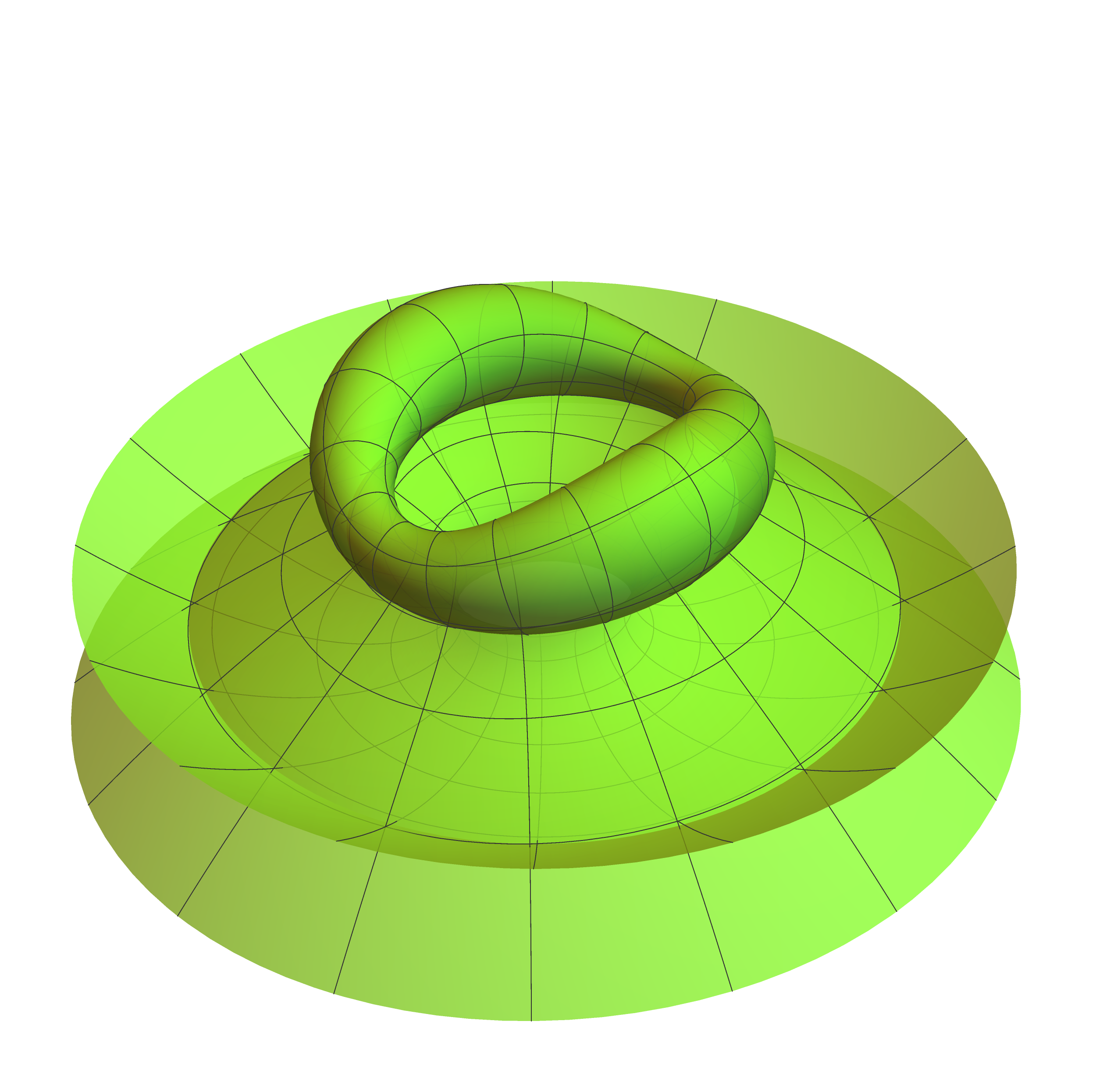}
\includegraphics[height=3.5cm]{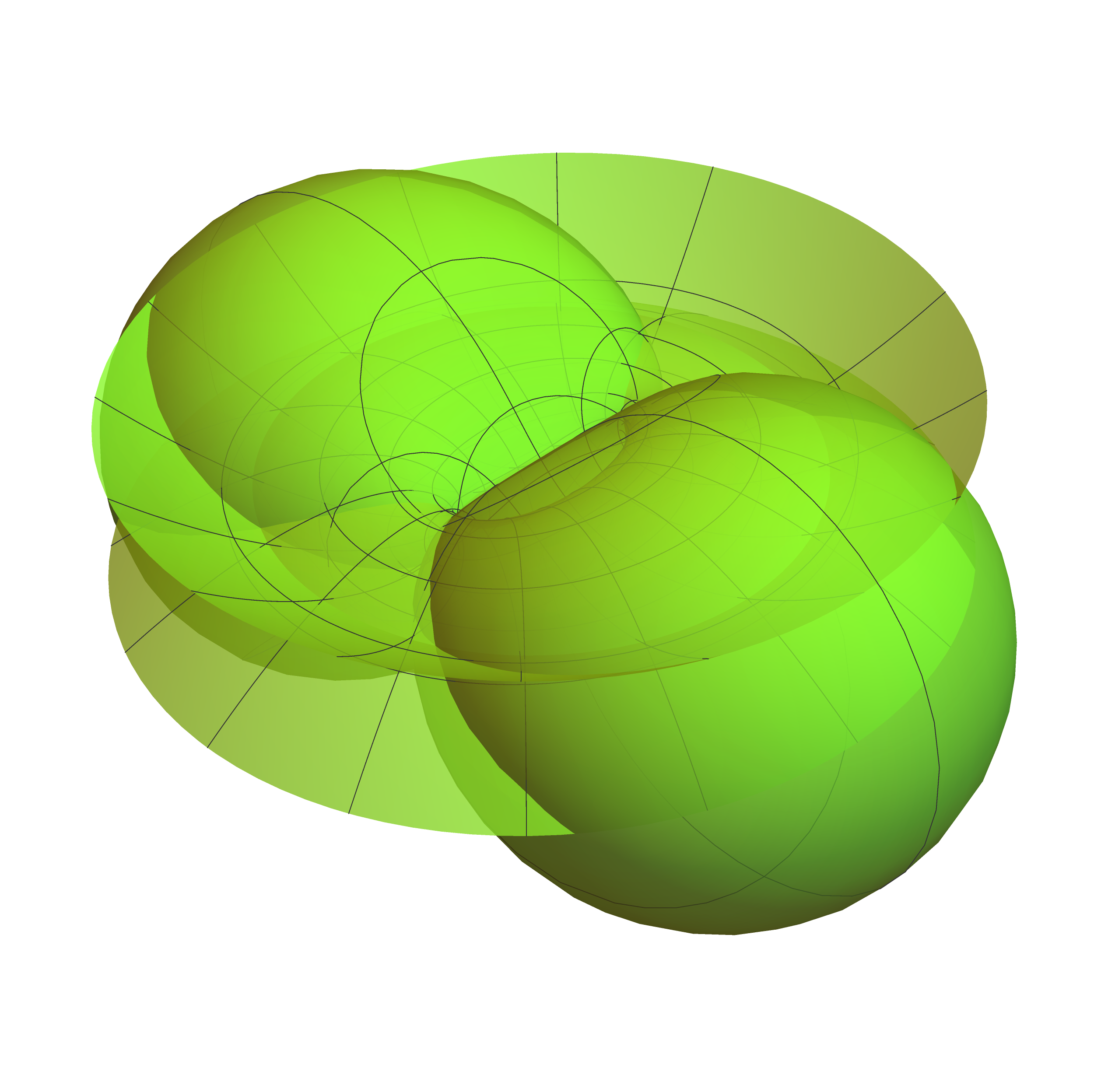}
\includegraphics[height=3.5cm]{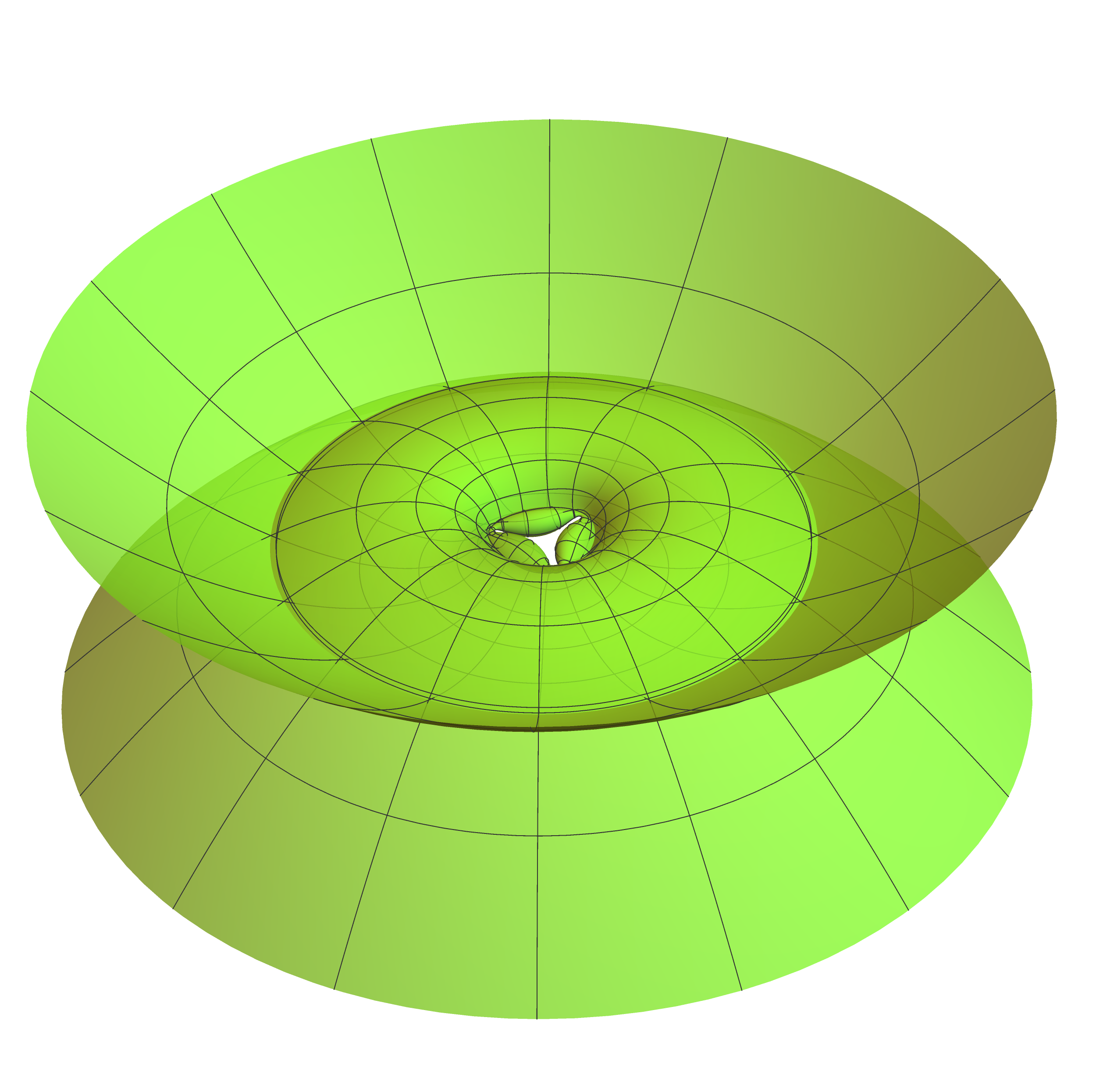}
\includegraphics[height=3.5cm]{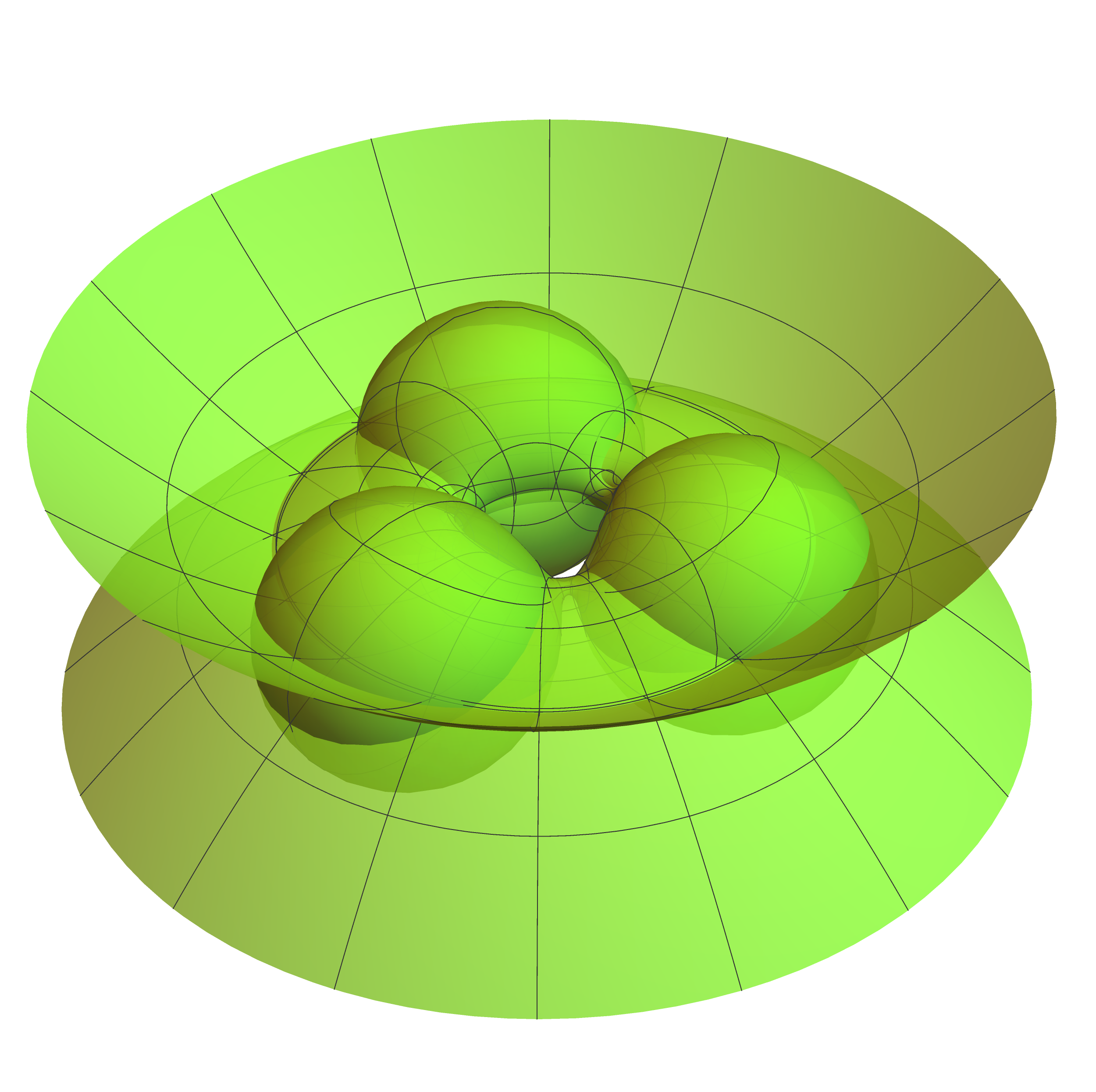}
\caption{Isothermic Bubbletons: classical Darboux transforms of the
  surface of revolution
  $f(x,y)= i( -x + \frac{x^3}3) + j(1+x^2)e^{-iy}$ at the resonance
  points $r_k=\frac{k^2-1}4$, $k=2,3$ with various initial
  conditions.}
\end{figure}

\end{example}

We can now conclude that the $\varrho$--Darboux transformation is a
genuine generalisation of the classical Darboux transformation. In
particular, the generalised Darboux transformation is a generalisation
of the classical Darboux transformation even when applied to
isothermic surfaces in 3--space.
\begin{cor}
  Let $f: M \to\R^3$ be an isothermic surface which is not contained
  in a sphere.  Let $D_{\phi,\varrho}(f) = f+ T$ be a
  $\varrho$--Darboux transform given by a $d_\varrho$--parallel
  section $\phi=\begin{pmatrix} \alpha\\ \beta
\end{pmatrix}$ with $\varrho\in\C_*$, that is,
$T=\alpha\beta\invers$. Then $\hat f$ is a classical Darboux transform
of $f$ if and only if $\varrho\in\R_*$.
\end{cor} 
\begin{proof}  
  Since $\hat f =f+T$ is a $\varrho$--Darboux transform we have by the
  generalised Riccati equation (\ref{eq:gen riccati})
\[
  dT = -df + Tdf^d \hat\varrho T
\]
with $\hat \varrho = \alpha\varrho\alpha\invers$, $\varrho\in\C$. Then
$\hat f$ is a classical Darboux transform of $f$ if and only if
$df^d\hat \varrho$ is the differential of a dual surface of $f$. Under
our assumptions, the uniqueness of the dual surface up to scaling and
translation, see e.g., \cite{darboux_isothermic}, shows that
$df^d\hat \varrho =df^d r$ for some $r\in\R$ but then
$\varrho=r\in\R$.
\end{proof}

For $\d_{\varrho_i}$--parallel sections $\varphi_i$ with
$d\varphi_2 = d\varphi_1\chi$ we denote as before by $f_i$ the Darboux
transforms given by $\varphi_i$.  Since $\varphi_i$ are the
prolongations of the holomorphic sections $e\alpha_i$ we have with
Theorem \ref{thm:bianchi} that
\[
\varphi = \varphi_2-\varphi_1\chi
\] gives a common Darboux transform $\hat L = \varphi\H$ of the two
Darboux transforms $f_i$ if $\varphi_1, \varphi_2$ are $\H$--linearly
independent. Indeed, we show that the common Darboux transform
$\hat L$ is isothermic:

\begin{theorem}[Bianchi permutability for the $\varrho$--Darboux
  transformation]
  \label{thm:
  gen bianchi}
Let $f_i= D_{\varphi_i, \varrho_i}(f): \tilde M \to \H$ be Darboux
transforms given by $\d_{\varrho_i}$--parallel sections $\varphi_i$,
$\varrho_i\in\C_*$. Let $\chi$ be defined by
$d\varphi_2 = d\varphi_1\chi$. Then the common Darboux transform
$\hat f$ given by $\hat L =\varphi\H$ where
\[
  \varphi = \varphi_2-\varphi_1\chi
\]
is a ${\varrho_1}$--Darboux transform of $L_2$ and a
${\varrho_2}$--Darboux transform of $L_1$ if $\varphi_1$ and
$\varphi_2$ are $\H$--linearly independent:
\[\hat f = D_{\varphi, \varrho_2}
 ( D_{\varphi_1, \varrho_1}(f)) = D_{-\varphi \chi\invers,
   \varrho_1}(D_{\varphi_2,\varrho_2}(f))\,.
\]
In
particular, $\hat f$ is isothermic.
\end{theorem}

\begin{proof}
  For $\d_{\varrho_i}$--parallel sections
  $\varphi_i = e\alpha_i + \psi \beta_i$ we see that
  $d\varphi_2 = d\varphi_1\chi$ with
\[
\chi = \varrho_1\invers\alpha_1\invers \alpha_2\varrho_2\,.
\]
Therefore, we write $\varphi = e\alpha + \psi_1 \beta$ with
$\psi_1 = e\alpha_1\beta_1\invers+ \psi$ and
\[
\alpha =
\alpha_2-\alpha_1\beta_1\invers \beta_2, \quad \beta=
\beta_2-\beta_1\varrho_1\invers\alpha_1\invers\alpha_2\varrho_2\,.
\]

Note that by assumption $\alpha_i(p), \beta_i(p)\not=0$ for all
$p\in M$.  We want to show that $\varphi$ is a
$\d^1_{\varrho_2}$--parallel section, or equivalently,
$\phi =\begin{pmatrix} \alpha\\ \beta
\end{pmatrix}
$ is a $d^1_{\varrho_2}$--parallel section of the associated family of
flat connections $d^1_\lambda$ of $f_1$ for
$\lambda=\varrho_2$. Therefore, we need to show that
$d\alpha=-df_1\beta$ and $d\beta=-df_1^d\alpha\varrho_2$. To do so we
recall (\ref{eq:dhatf}) that $f_1=f+ \alpha_1\beta_1\invers$ has
differential
\[
df_1 = T_1df^d\alpha_1\varrho_1\beta_1\invers
\]
and $f_1^d= f^d+ \beta_1\varrho_1\invers \alpha_1\invers$ has
(\ref{eq:dhatfd}) differential
\[
df_1^d = T_1^ddf\beta_1\alpha_1\invers\,.
\]
We also note that
\[
\alpha = \alpha_2 - T_1\beta_2, \quad \beta=\beta_2-T_1^d
\alpha_2\varrho_2\,.
\]
Thus, using the Riccati type equation (\ref{eq:gen riccati}) for
$T_1=\alpha_1\beta_1\invers$ and the fact that $\phi_i$ are
$d_{\varrho_i}$--parallel:
\begin{align*}
d\alpha
& = -T_1df^d\alpha_1\varrho_1\beta_1\invers(\beta_2
  -\beta_1\varrho_1\invers\alpha_1\invers\alpha_2\varrho_2) =-df_1\beta
\end{align*}
and similarly, using $\beta=\beta_2-T_1^d \alpha_2\varrho_2$ and the
Riccati type equation (\ref{eq:riccatidual}) for $T^d$, one obtains
$d\beta=-df_1^d\alpha \varrho_2$.

Since $\chi\invers = \varrho_2\invers\alpha_2\invers\alpha_1\varrho_1$
and $\hat L = \varphi\H =-\varphi\chi\invers\H$, we can apply the same
arguments to show that $\hat f$ is a $\varrho_1$--Darboux tranform of
$f_2$ by swapping the indices.
\end{proof}

Note that we did not exclude the case $\varrho_1 =\varrho_2$. Under
our assumption $f_1\not=f_2$ we thus obtain even in this case a common
Darboux transform of $f_i$.  In the classical situation, i.e., when
$\varrho\in\R_*$, this common transform is the original isothermic
surface $f$ since in this case
$\beta= \beta_2-\beta_1\alpha_1\invers\alpha_2$ and thus
$\alpha\beta\invers =-\alpha_1\beta_1\invers$. However, in the case
when $\varrho\not\in\R$ Bianchi permutability in general gives new
isothermic surfaces.

Even in the case when $\varrho\in\R$, there are double Darboux
transforms which are not $f$ but they do not arise from Bianchi
permutability.  They can be computed by a Sym--type formula,
\cite{sym-darboux}.

\begin{figure}[H]
  \includegraphics[width=.6\linewidth]{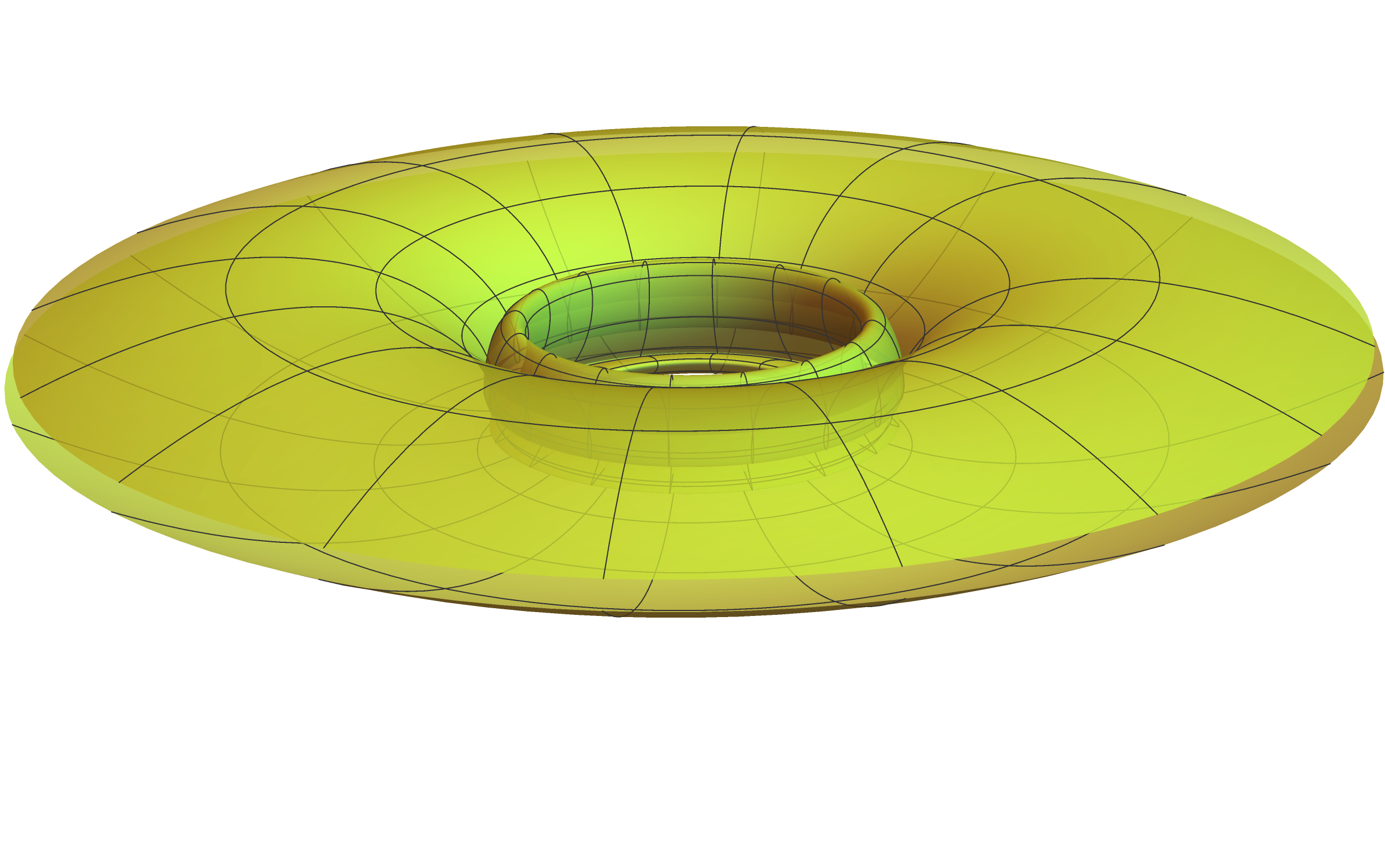}
\caption{Common $\rho$--Darboux transforms of the two Darboux
  transforms given in Figure \ref{fig: generalised dt}, with spectral paramater
  $\varrho=\varrho_1=\varrho_2 = 1+i$, orthogonally projected to
  $\R^3$.}
\end{figure} 

\subsection{Simple factor dressing of isothermic surfaces}

We will now discuss the dressing operation introduced by Terng and
Uhlenbeck for harmonic maps \cite{terng_uhlenbeck} in the case of
isothermic surfaces.  As we have seen, isothermic surfaces are
determined by their associated family of flat connections. To
construct new isothermic surfaces, one can gauge the associated family
with a suitable gauge matrix, see \cite{fran_epos}, which depends on
the spectral parameter, to create a new associated family of flat
connections of the correct form as in Theorem \ref{thm:retraction}.  A
special case of dressing is the simple factor dressing which is, as we
will show, the $\varrho$--Darboux transformation: we assume that the
map $\lambda \to r(\lambda)$ has only a simple pole at some fixed
complex value $\lambda=\bar\varrho\in\C$.

\begin{theorem}[Simple factor dressing]
  \label{thm:sfd} Let $f: M \to\R^4$ be isothermic with associated
  family $\d_\lambda = d+ \lambda\eta$, $\lambda\in\C$. Let
  $\varrho\in\C_*$ and $\varphi$ a $\d_\varrho$--parallel
  section. Assume that
  $\tilde{\underline{\C}}^4= E\oplus E j\oplus \tilde L$, where
  $E= \varphi\C$ and $\tilde L$ is the line bundle of $f$ over the
  universal cover $\tilde M$, and denote
  by $\pi_{E}, \pi_{E j}$ and $\pi_L$ the projections onto
  $E= \varphi\C$, $Ej$ and $\tilde L$ respectively along the splitting.  Let
\[
r_\varrho^E(\lambda)= \pi_E \gamma_\varrho(\lambda) + \pi_{Ej}+ \pi_L \sigma_\varrho(\lambda)
\]
where
\[
\gamma_\varrho(\lambda)= \frac{\bar \varrho
  (\varrho-\lambda)}{\varrho(\bar\varrho -\lambda)}, \quad \sigma_\varrho(\lambda)=\frac{\bar
  \varrho}{\bar\varrho - \lambda
}
\]
act on $\C^4$ by right multiplication.

Then the Darboux transform $\hat f = D_{\varphi, \varrho}(f)$ is given
by a \emph{simple factor dressing} with dressing matrix $r_\lambda$,
that is, the line bundle of $\hat f$ is given by
$r_\varrho^E(\infty)\tilde{\C}^4$ and its associated family is given
by $\hat d_\lambda = r_\lambda\cdot d_\lambda$. We denote the simple
factor dressing with parameter $\varrho$ and complex line bundle $E$
by $\hat f = r_{E,\varrho}(f)$.

 Moreover,  $\hat \d_\lambda= d+\lambda\hat \eta$
where
\[
\hat \eta = -(\pi_E\circ d \circ \pi_L \frac 1 \varrho + \pi_{Ej} \circ d \circ \pi_L
\frac 1{\bar \varrho})\,.
\]

\end{theorem}

In the following, we will abbreviate $r_\lambda=r_\varrho^E(\lambda)$,
$\gamma_\lambda=\gamma_\varrho(\lambda),$ and
$ \sigma_\lambda=\sigma_\varrho(\lambda)$ to simplify notations if the
dependence of the parameters is clear from the context.

\begin{rem} If $\varrho\in\R$ this simplifies to the simple factor
  dressing as discussed in \cite{sym-darboux}.
\end{rem}
\begin{proof}
We first observe that
\[
\gamma_\lambda\invers = \bar \gamma_{\bar\lambda}\,, \quad \sigma_{\bar \lambda} = \overline{\sigma_\lambda}\gamma_{\bar\lambda}
\]
and
\[
\pi_E(\xi j) = (\pi_{Ej}\xi)j
\]
 for $\xi\in\Gamma(\ttrivial 2)$  so that
\[
r_{\bar\lambda}(\xi j) =
r_\lambda(\xi)j\gamma_{\bar\lambda}
\]
and hence
\[
r_\lambda\invers(\xi j) = r_{\bar\lambda}\invers(\xi)\gamma_{\bar\lambda} j\,.
\]
Since $r_\lambda\cdot d_\lambda= r_\lambda \cdot d + \varrho
\Ad(r_\lambda) \eta$ we consider
\[
r_\lambda \cdot d(\xi j) = (r_{\bar\lambda} \cdot d(\xi))j\,, \quad 
\Ad(r_\lambda)\eta(\xi j) = (\Ad(r_{\bar\lambda}) \eta \xi) j
\] 
which gives
\begin{equation}
\label{eq:dhatreality}
\hat \d_\lambda(\xi j) = (\hat \d_{\bar\lambda} \xi)j\,.
\end{equation}

In particular, $\hat \d_\lambda$ is a quaternionic connection for
$\lambda\in\R$. 

We now claim that the map
$\lambda \mapsto \hat \d_\lambda, \lambda\in\C$, is holomorphic. Since
$r_\lambda$ and $r_\lambda\invers$ are holomorphic away from
$\lambda=\varrho, \bar\varrho$ and $\lambda\mapsto \d_\lambda$ is
holomorphic on $\C$ we only have to investigate holomorphicity at
$\lambda=\varrho, \bar\varrho$. By the reality condition
(\ref{eq:dhatreality}) it is enough to show holomorphicity at
$\lambda=\varrho$.

Writing 
\[
\d_\lambda = \d_\varrho+ (\lambda-\varrho)\eta
\]
we have
\[
\hat \d_\lambda = r_\lambda\cdot \d_\varrho +
(\lambda-\varrho)\Ad(r_\lambda) \eta\,.
\]

Since
$r_\lambda\invers= \pi_{E} \gamma_\lambda\invers+ \pi_{Ej} +
\pi_L\sigma_\lambda\invers$ and
$L \subset\ker \eta, \im \eta \subset L$ we see that
\begin{equation}
\label{eq:Adreta}
\Ad(r_\lambda)\eta = \pi_L \eta(\pi_E\frac{\varrho}{\varrho-\lambda} +
\pi_{Ej} \frac{\bar \varrho}{\bar\varrho-\lambda})
\end{equation}
and thus $(\lambda-\varrho) \Ad(r_\lambda)\eta$ is holomorphic at
$\lambda=\varrho$. We now investigate $r_\lambda\cdot
d_\varrho$. Since $r_\lambda$ and
$(\pi_{Ej} + \pi_L \sigma_\lambda\invers)$ are holomorphic at
$\lambda =\varrho$ and
\[
  r_\lambda\cdot d_\varrho = r_\lambda \circ d_\varrho \circ (\pi_E
  \gamma_\lambda\invers + \pi_{Ej} + \pi_L \sigma_\lambda\invers)
\]
we only have to consider the term
\[
r_\lambda \circ d_\varrho\circ (\pi_E\gamma_\lambda\invers)\,.
\]
Since $\varphi$ is $\d_\varrho$--parallel, the bundle $E=\varphi\C$ is
$\d_\varrho$--stable. Therefore,
\[
r_\lambda \circ \d_\varrho \circ(\pi_E\gamma_\lambda\invers)= 
\pi_E \circ \d_\varrho \circ \pi_E
\]
which is independent of $\lambda$. This shows that the map
$\lambda\mapsto \hat \d_\lambda =r_\lambda\cdot \d_\lambda$ is
holomorphic at $\lambda=\varrho$, and by the reality condition it is
holomorphic on $\C$.  We now investigate the behaviour of
$\hat \d_\lambda$ at $\infty$.

We write again
$\hat \d_\lambda = r_\lambda \cdot d + \lambda \Ad(r_\lambda) \eta $.
By (\ref{eq:Adreta}) we have
\[
\lim_{\lambda\to \infty} \lambda \Ad(r_\lambda)\eta
= -\pi_L \eta(\pi_E \varrho + \pi_{Ej} \bar\varrho)\,.
\]
On the other hand,
\begin{align*}
r_\lambda\cdot d & = (\pi_E\gamma_\lambda + \pi_{Ej} + \pi_L
                   \sigma_\lambda) \circ d \circ  (\pi_E\gamma_\lambda\invers + \pi_{Ej} + \pi_L
                   \sigma_\lambda\invers) \\
&= \pi_E\circ d \circ \pi_E+\pi_{E j} \circ d \circ \pi_{Ej} +
  \pi_L\circ d \circ \pi_L  + 
\pi_E\circ d \circ \pi_{Ej} \gamma_\lambda + \pi_E\circ d \circ \pi_L
  \frac{\varrho-\lambda}{\varrho}\\
&\quad +
\pi_{Ej}\circ d \circ \pi_E \gamma_\lambda\invers + \pi_{Ej}\circ d
  \circ \pi_L \frac{\bar\varrho-\lambda}{\bar\varrho} 
+ \pi_L\circ d \circ \pi_{E} \frac \varrho
  {\varrho-\lambda} + \pi_L\circ d \circ \pi_{Ej}
 \frac{\bar\varrho}{\bar\varrho-\lambda}
\end{align*}
gives with
$\lim_{\lambda\to\infty} \gamma_\lambda= \frac{\bar\varrho}\varrho$
and
$ \lim_{\lambda\to\infty} \frac \varrho {\varrho-\lambda}
=\lim_{\lambda\to\infty} \frac {\bar\varrho} {\bar\varrho-\lambda}=0$
that $r_\lambda \cdot d$ has a simple pole at infinity. Since
$\hat \d_{\lambda=0} = d$ we thus can write
\[
\hat \d_\lambda = d + \lambda \hat \eta_\lambda
\]
where $\hat \eta_\lambda$ is holomorphic on the compact
$\C\cup\{\infty\}$ and thus constant in $\lambda$.  Moreover,
combining the previous computations we have
\begin{align*}
\hat \eta = \hat \eta_\lambda &= \lim_{\lambda\to \infty} \frac 1 \lambda \hat d_\lambda = 
 \lim_{\lambda\to \infty}\pi_E\circ d \circ \pi_L
  \frac{\varrho-\lambda}{\lambda\varrho} +  \pi_{Ej}\circ d
  \circ \pi_L \frac{\bar\varrho-\lambda}{\lambda\bar\varrho} \\
&= 
-(\pi_E\circ d \circ \pi_L \frac 1 \varrho + \pi_{Ej} \circ d \circ \pi_L
\frac 1{\bar \varrho})\,.
\end{align*}

This shows that $\hat \d_\lambda$ is a family of flat connections with
$\hat\d_\lambda = d + \lambda\hat\eta$ and $\hat\eta^2=0$.  Therefore,
$\im \hat \eta =\ker \hat \eta = \varphi\H$ defines an isothermic
surface $\hat L$ with associated family $\hat \d_\lambda$ by Theorem
\ref{thm:retraction}. Indeed, $\hat L =\varphi\H$ is the
$\varrho$--Darboux transform $D_{\varphi, \varrho}(f)$ given by the
$\d_\varrho$--parallel section $\varphi$. Finally, since
$r_\varrho^E(\infty) = \pi_E\frac{\bar \rho}{\rho} + \pi_{E j}$ we see
that $\hat L = r_\varrho^E(\infty) \tilde{\C}^4$.
\end{proof}

We can now investigate the common $\varrho$--Darboux transform of two
$\varrho$--Darboux transforms, which is given by a similar dressing
matrix. We omit the proof of the following theorem as it is analogous
to the proof of Theorem \ref{thm:sfd}.

\begin{theorem} Let $f: M \to\R^4$ be isothermic with associated
  family $\d_\lambda = d + \lambda\eta$, $\lambda\in\C$.  Let
  $\varrho\in\C_*$ and $W$ a $\d_\varrho$--parallel, 2 dimensional
  complex bundle such that $W \oplus W j =
  \underline{\tilde\C}^4$. Define
\[
r_\varrho^W(\lambda) = \pi_W \gamma_\varrho(\lambda) + \pi_{Wj}
\]
where $\pi_W$ and $\pi_{Wj}$ are the canonical projections onto $W$
and $W j$ along the splitting $\underline{\tilde\C}^4= W\oplus
Wj$. Moreover,
\[
\gamma_\varrho(\lambda) = \frac{\bar\varrho(\varrho-\lambda)}{\varrho(\bar \varrho-\lambda)}
\]
acts by right multiplication on $\C^4$.  Then the surface
$\hat f = r_{W,\varrho}(f)$ given by the line bundle
$\hat L = r_\varrho^W(\infty) \tilde L$ is an isothermic surface, the
\emph{(2--step) simple factor dressing} of $f$. Its associated family
$\d_\lambda = r_\varrho^W(\lambda)\cdot \d_\lambda$ is given by the
gauge of $\d_\lambda$ by the gauge matrix $r_\varrho^W(\lambda)$.

\end{theorem}

\begin{rem}
  We see that $r_\varrho^W$ is the identity if $\varrho\in\R$, that
  is, an isothermic surface $f$ is a 2--step simple factor dressing
  $f=r_{W,\varrho}(f)$ of itself if $\varrho\in\R$.
\end{rem}

The 2--step simple factor dressing of an isothermic surface $f$ with
pole $\bar\varrho$ is linked to the $\varrho$--Darboux transformation
of $f$ via Bianchi permutability:

\begin{theorem}
  Let $f: M \to \R^4$ be isothermic and $\varphi_1, \varphi_2$ be two
  $\d_\varrho$--parallel sections such that $W \oplus W j=\ttrivial 2$
  where $W = \Span_\C\{\varphi_1, \varphi_2\}$.  Let
  $f_1=r_{E_1,\varrho}(f)$ and $f_2=r_{E_2,\varrho}(f)$ be the
  $\varrho$--Darboux transforms given by $E_1=\varphi_1\C$ and
  $E_2=\varphi_2\C$ respectively.

  Then the 2--step simple factor dressing $\hat f =r_{W,\varrho}(f)$
  of $f$ is a common $\varrho$--Darboux transform of $f_1$ and $f_2$.
\end{theorem}

\begin{proof}
Writing $\varphi_i = e\alpha_i + \psi \beta_i$
 we consider
\[
\tilde \psi=\psi(\beta_2-\beta_1\alpha_1\invers \alpha_2) =\varphi_2 -\varphi_1
\alpha_1\invers\alpha_2\in\Gamma(\tilde L)\,.
\]
Since the 2--step simple factor dressing is given by
$r_\varrho^W(\infty) \tilde L$ we will now compute
$r_\varrho^W(\infty) (\tilde\psi \frac 1{ \bar \varrho})$. Recall that
$r_\varrho^W (\infty)= \pi_{W}\frac{\bar \varrho}{\varrho} + \pi_{Wj}$
so that
$r_\varrho^W(\infty)(\varphi m \frac 1{\bar \varrho}) = \varphi \frac
1\varrho m$ for $\varphi\in\Gamma(W)$ and $m: M \to \H$ and
\begin{align*}
r_\varrho^W(\infty) (\tilde\psi \frac 1{
\bar \varrho}) &= r_\varrho^W(\infty)((\varphi_2-\varphi_1\alpha_1\invers\alpha_2)\frac
  1{\bar \varrho}) = (\varphi_2 - \varphi_1\varrho\invers
  \alpha_1\invers \alpha_2 \varrho)\varrho\invers =(\varphi_2-\varphi_1\chi)\varrho\invers
\end{align*}
with $\chi= \varrho\invers\alpha_1\invers\alpha_2\varrho$.  Thus, the
2--step simple factor dressing $\hat f$ is given by the line bundle
$r_\varrho^W(\infty)\tilde L= (\varphi_2-\varphi_1\chi)\H$ which by
Theorem \ref{thm: gen bianchi} is the line bundle of the common
$\varrho$--Darboux transform of $f_1$ and $f_2$.
\end{proof}

\begin{rem} Note that by Theorem \ref{thm: gen bianchi} the complex
  line bundles $E_{12} =\varphi\C$ and $E_{21}=-\varphi\chi\invers \C$
  respectively are given by the parallel section
  $\varphi = \varphi_2-\varphi_1\chi$ with
  $d\varphi_2 = d\varphi_1\chi$. In general, $E_{12}\not=E_{21}$ since
  $\chi$ is quaternionic.
\end{rem}

\section{CMC surfaces}
\label{sec:cmc}

In this section we will now compare results for general isothermic
surfaces to ones given by the integrable system structure of a CMC
surface which is given by the harmonic Gauss map: a CMC surface
$f: M\to\R^3$ from a Riemann surface into 3--space with constant mean
curvature $H=1$ is isothermic since the parallel CMC surface is a dual
surface of $f$ (if $f$ is not a round sphere). We will link the
associated family of flat connections of a CMC surface, which is given
by its harmonic Gauss map, to the isothermic family $\d_\rho$ of flat
connections. This will enable us to show that $\mu$--Darboux
transforms of a CMC surface $f$ as defined in \cite{cmc} are exactly
those $\rho$--Darboux transforms of $f$ which have constant mean
curvature.

\subsection{The associated family of flat connections of a harmonic map}

If $f: M \to\R^3$ is a conformal immersion then the Ruh--Vilms
theorem, \cite{ruh_vilms}, follows from \eqref{eq:Hdf}: $H$ is
constant if and only
\[ d*dN + 2|dN|^2N =0\,,
\]
that is, if $N: M\to S^2$ is harmonic. Here we identify a 2--form with
its quadratic form $\omega\wedge \eta = \omega *\eta -*\omega \eta$
for $\omega,\eta \in\Omega^1(M)$.

In this case, there is a $\C_*$--associated family of flat connections
and, in case of a torus, one can describe harmonic maps via their
spectral data, \cite{hitchin-harmonic}. If $f$ has constant mean
curvature $H\not=0$ then we consider the following $\C_*$--family of
flat connections on the trivial $\C^2$--bundle which is gauge
equivalent to the Hitchin family, see \cite{cmc}: identifying
$(\H, I) = \C^2$, where $I$ is the right multiplication by the
quaternion $i$, the family of connections
 \[
d^N_\lambda= d+ (\lambda-1) A\oz + (\lambda\invers-1)A\zo, \lambda\in\C_*\,,
\]
is flat if and only if $N$ is harmonic. Here $A =\frac 14(*dN + NdN)$
and $A\oz = \frac 12(A-I*A)$ and $A\zo=\frac 12(A+I*A)$ denote the
$(1,0)$ and $(0,1)$--parts of $A$ with respect to the complex
structure $I$.  Without loss of generality we will from now on assume
that our CMC surfaces have mean curvature $H=1$ and additionally
exclude the case of a round sphere. We recall from Example
\ref{ex:isothermic} that in this case the parallel surface $g=f+N$ is
a dual surface of $f$ where $N$ is the Gauss map of $f$.  Moreover, we
have $2*A = df$ by (\ref{eq:Hdf}) and for $\lambda\in\C_*$ we can
write for $\alpha\in\Gamma(\trivial {})$, see \cite{cmc}:
\begin{equation}
\label{eq: dalpha}
d^N_\lambda\alpha = d\alpha+ \frac 12df(N\alpha(a-1) + \alpha b)
\end{equation}
where
$a=\frac{\lambda+\lambda\invers}2, b=
i\frac{\lambda\invers-\lambda}2$.  In particular, $a^2 + b^2=1$ and
$\lambda =a+ i b$. Moreover, $d^N_\lambda$ is a quaternionic
connection on $\trivial{}$ if and only if $\lambda\in S^1$ and in this
case $a= \Re(\lambda), b=\Im(\lambda)$.

\begin{theorem}
  \label{thm: both parallel}
  Let $f: M\to\R^3$ be a CMC surfaces with $H=1$, and dual surface
  $g=f+N$, where $N$ is the Gauss map of $f$. Denote by $d_\mu^N $ the
  associated family of flat connections of $f$ and by $d_\mu^{N_g}$
  the associated family of flat connections of $g$.

  Let $\alpha\in\Gamma(\tilde\H)$ and for
  $\mu\in\C\setminus\{0,1\}\,, a=\frac{\mu+\mu\invers} 2$ and
  $b = i\frac{\mu\invers-\mu}2$ define
  \[
    \beta=\frac 12(N\alpha(a-1) + \alpha b)\,.
  \]

  Then $\alpha$ is $d_\mu^N$--parallel if and only if $\beta$ is
  $d_\mu^{N_g}$-- parallel.
\end{theorem}
\begin{proof}
  Recall that $N_g=-N$ is the Gauss map of $g$. Using $a^2+b^2=1$ we
  have
  \[
    \frac 12(-N\beta(a-1) + \beta b) = \alpha \frac{1-a}2
  \]
  so that the flat connection $d_\mu^{-N}$  of $g$ is given by
  \[
    d_\mu^{-N} \beta= d\beta + \frac 12 dg(N_g\beta (a-1) + \beta b)
    = d\beta + dg \alpha \frac{a-1}2\,.
  \]
  If $d_\mu^N\alpha=0$, that is $d\alpha=-df\beta$ then 
  \[
    d\beta = \frac 12((dg-df)\alpha (a-1) -Ndf\beta(a-1) - df\beta b)
    = dg \alpha\frac{a-1}2
  \]
  shows that $\beta$ is $d_\mu^{-N}$--parallel.

  Conversely, writing $\alpha_g= \beta$ we have $d_\mu^{N_g} \alpha_g
  = d\alpha_g - dg \alpha \frac{1-a}{2}$ so that
  \[
    \beta_g = \alpha \frac {1-a}2\,.
  \]
  Reversing now the roles of $f$ and $g$, we see that if
  $\alpha_g=\beta$ is $d^{N_g}_\mu $--parallel then
  $d^{N}_\mu \beta_g=0$ by symmetry. Since $\mu\not=1$, this shows
  that $\alpha$ is $d_\mu^N$--parallel.

\end{proof}

From this discussion we immediately see that a $d^N_\mu$--parallel
section gives, without further integration, a parallel section of the
associated family $d_\varrho$ where we now view $f$ as an isothermic
surface.

\begin{theorem}
\label{thm: drho parallel}
Let $f: M \to\R^3$ be a CMC surface with $H=1$ and 
\[
d_\varrho = d + \begin{pmatrix} 0 &df \\ 0 & 0
\end{pmatrix} + \varrho \begin{pmatrix} 0 & 0\\ dg &0
\end{pmatrix}\,,
\] 
where $g=f+N$ is the parallel CMC surface of $f$.  Then all
$d_\varrho$--parallel sections, $\varrho\in\C\setminus\{0,1\}$, are
given algebraically by
$d_\mu^N$--parallel sections. \\

More precisely, any $d_\varrho$--parallel section is given as
 \[
\phi = \phi_+ + \phi_-\,, \qquad\phi_\pm = \begin{pmatrix} \alpha_\pm \\ \beta_\pm
\end{pmatrix}\,,
\]
where  
\begin{align*}
d_{\mu_\pm}^N\alpha_\pm &=0\,,  \qquad\beta_\pm = \frac 12(N\alpha_\pm
(a-1) \pm \alpha_\pm b)
\end{align*}
with $a=1-2\varrho, b=2\sqrt{\varrho(1-\varrho)}$ and
$\mu_\pm = a\pm ib$. We call $\mu_\pm= a\pm ib$ the \emph{CMC spectral
  parameter} associated to the isothermic spectral parameter
$\varrho$.
\end{theorem}
\begin{rem} We note that there is an ambiguity of choice of
  $\sqrt{\varrho(1-\varrho)}$ however, either choice gives rise to the
  same pair $(\mu_+, \mu_-)$.
  \end{rem}
\begin{proof} 
  Since $a^2+b^2=1$ we see that $(a+ i b)\invers = a-ib$. Therefore,
  $\frac{\mu_\pm+\mu_\pm\invers} 2 = a$ and $
  i\frac{\mu_\pm\invers-\mu_\pm}2 = \pm b$
  and hence
  \[
\frac{2-\mu_\pm- \mu_\pm\invers}4 = \frac{1-a}2 =\varrho\,.
\]
  Since
  $d^N_{\mu_\pm}\alpha_\pm =0$ we see by Theorem \ref{thm: both
    parallel}
   that $d\alpha_\pm = -df
  \beta_\pm$ and
  $\beta_\pm = \frac 12(N \alpha_\pm(a-1) \pm \alpha_\pm b)$ satisfies
\[
d\beta_\pm  =
-dg\alpha_\pm \varrho\,,
\] 
 so that
$\phi_\pm=\begin{pmatrix}\alpha_\pm\\ \beta_\pm
\end{pmatrix}
$ is $d_\varrho$--parallel.   

Now, if $\varrho\not =0, 1$ then non--trivial  $\alpha_+, \alpha_-$ are complex
linearly independent: if $\alpha_+ = \alpha_- m$, $m\in\C$, then
$d_{\mu_+}^N \alpha_+ = d_{\mu_-}^N \alpha_+=0$ implies
\[
df(N\alpha_+(a-1)+\alpha_+b) = df(N\alpha_+(a-1) -\alpha_+b)\,.
\]
Since $f$ is an immersion this shows $b=0$ and $\mu_+=a+ib =\pm 1$
since $a^2+b^2=1$, contradicting our assumption
$\varrho=\frac{2-\mu_+-\mu_+\invers}4\not=0,1$.  Since the space of
$d_\mu^N$--parallel sections is complex 2--dimensional, and
$\alpha_+, \alpha_-$ are complex linearly independent if
$\varrho=\frac{1-a}2\not\in\{0,1\}$, we obtain a 4--dimensional space
of $d_\varrho$--parallel sections, that is, all $d_\varrho$--parallel
sections arise this way if $\varrho\in \C\setminus\{0,1\}$.
\end{proof}

\quad

\begin{rem}
\label{rem:quaternionicindependence}
If $\varrho\in (0,1)$ then $\mu_\pm\in S^1$. In this case, the flat
connections $d_{\mu_\pm}^N$ are quaternionic. In particular, the same
argument as above shows that in this case $d_{\mu_\pm}^N$--parallel
sections are $\H$--linearly independent.

Moreover, for $\varrho=0$, the $d_\varrho$--parallel sections are the
constant sections, whereas in the case $\varrho=1$ the two spectral
parameter $\mu_\pm=-1$ coincide and $d_\varrho$--parallel sections
occur which are not given by $d_{\mu=-1}^N$--parallel sections, see
the example of a cylinder below. Indeed for any
$d_{\mu=-1}^N$--parallel section $\alpha$ we obtain
$\beta = \frac 12(N\alpha(a-1)+\alpha b) = -N\alpha$.

\end{rem}

Theorem \ref{thm: drho parallel} shows that all $d_\rho$--parallel
sections are given algebraically by surface data and
$d_{\mu_\pm}^N$--parallel sections where $N$ is the Gauss map of $f$
and $\mu_\pm$ are the associated CMC spectral parameters of
$\rho$. Conversely, given all $d_\rho$--parallel sections we can
determine algebraically the $d_\mu^N$--parallel sections if $\rho$ is
the isothermic spectral parameter given by $\mu$.

\begin{theorem}
  Let $f: M\to\R^3$ be a CMC surface with $H=1$ and dual surface
  $g=f+N$. For spectral parameter $\mu\in \C\setminus\{0, \pm 1\}$ let
  $\varrho = -\frac{(\mu-1)^2}{4\mu}$ be the associated isothermic
  spectral parameter and
  $a= \frac{\mu+\mu\invers}2, b= i\frac{\mu\invers-\mu}2$. If a
  $\d_\varrho$--parallel section $\varphi = e\alpha+ \psi \beta$
  satisfies at some fixed point $p_0\in M$
  \[
    \beta(p_0) = \frac 12(N(p_0)
    \alpha(p_0)(a-1) + \alpha(p_0) b)
  \]
  then $\alpha$ is $d_\mu^N$--parallel. All $d_\mu^N$--parallel
  sections arise this way.

  Here $\d_\varrho$ is the isothermic family of flat connectionqs
  given by the choice of dual surface $g =f +N$, and $d_\mu^N$ is the
  CMC associated family given by the Gauss map $N$.
\end{theorem}
\begin{proof}
  Assume that $\varphi=e\alpha+\psi \beta$ is $d_\varrho$--parallel
  with $\varrho = -\frac{(\mu-1)^2}{4\mu} \in \C\setminus\{0, 1\}$ so
  that
  \[
    d\alpha=-df\beta, \quad d\beta = -dg\alpha\varrho\,.
  \]
  By assumption $X(p_0) =0$ for
  \[
    X = \beta - \frac 12(N\alpha(a-1) + \alpha b)\,.
  \]
  Now with $dN = dg-df$ and $a^2+b^2=1$ we have
  \[
    dX = d\beta - \frac 12(dN \alpha(a-1) + d\alpha(a-1)+ d \alpha b) =
    \frac 12 df(NX(1-a) + Xb)\,.
  \]
  Hence uniqueness of solutions gives $X(p) =0$ for all $p\in M$, that
  is, since $\varphi$ is $d_\varrho$--parallel
  \[
    d^N_\mu\alpha =d\alpha + \frac 12 df(N\alpha(a-1) + \alpha b) =
    d\alpha+ df\beta =0
  \]
  which shows the first claim. The second one follows from the fact
  that the space of $\d_\varrho$--parallel sections is complex
  4--dimensional, so the space of solution of $X(p_0)= 0$ is complex
  2--dimensional, that is, all $d_\mu^N$--parallel sections arise this
  way.

\end{proof}

\begin{example} 
  In the case of a cylinder $f(x,y) = \frac 12(ix + j e^{-iy})$ with
  Gauss map $N=-je^{-iy}$ we can use Example \ref{ex:rhodt of sor} to
  find all parallel sessions of $d_\varrho$, $\varrho\in\C_*$. Note
  that since we choose here the parallel surface $g=f+N$ as dual
  surface we have to scale the spectral parameter by $-\frac 14$
  because $dg = -\frac 14(f_x\invers dx - f_y\invers dy)$. Then all
  parallel sections $\phi$ are given for $\varrho\not=1$, see also
  \cite{sym-darboux, cmc}, by the (complex) basis
\[
(\phi^1_+, \phi^1_-, \phi^2_+, \phi^2_-)
\]
where $\phi^i_\pm =\begin{pmatrix} \alpha^i_\pm \\ \beta^i_\pm
\end{pmatrix}
$ with
\begin{align*}
\alpha^1_\pm &= e^{\frac {iy}2}(\sqrt \varrho +j (1\pm
         \sqrt{1-\varrho}))e^{\frac i2(\sqrt \varrho x\pm \sqrt{1-\varrho}y) } 
         ) \\
\alpha^2_\pm &= e^{\frac {iy}2}(-\sqrt \varrho +j (1\mp \sqrt{1-\varrho}))e^{-\frac{i}2(\sqrt \varrho x\pm \sqrt{1-\varrho}y) }\,,
\end{align*}
and $\beta^l_\pm$ given by $d\alpha^l_\pm =-df \beta^l_\pm$. It is now
a lengthy but straight forward computation to show that
$\alpha^l_\pm$, $l=1,2$, are $d_{\mu_\pm}^N$--parallel, where
$\mu_\pm =1-2\varrho \pm 2i\sqrt{\varrho(1-\varrho)}$. Here the square
roots are chosen so that
$\sqrt{\varrho(1-\varrho)} = \sqrt{\varrho}{\sqrt{1-\varrho}}$. In
particular,
\[\beta^l_\pm =
  \frac 12(N\alpha^l_\pm(a-1)\pm \alpha^l_\pm b)\,,
\]
with $a= 1-2\varrho, b= 2\sqrt{\varrho(1-\varrho)}$. We note that
$\phi^i_\pm$ are sections with multiplier
$h =- e^{\pm \pi i \sqrt{1-\varrho}}$ since
\begin{align*}
  \phi_+^1(x,y+2\pi) &= -\phi_+^1(x,y) e^{\pi i \sqrt{1-\varrho}}\,,
   \qquad & \phi_-^1(x,y+2\pi) = -\phi_-^1(x,y) e^{-\pi i
     \sqrt{1-\varrho}} \,, \\
    \phi_+^2(x,y+2\pi) &= -\phi_+^2(x,y) e^{-\pi i
      \sqrt{1-\varrho}}\,, 
    \qquad  & \phi_-^2(x,y+2\pi) = -\phi_-^2(x,y) e^{\pi i
       \sqrt{1-\varrho}}\,.
\end{align*}
Therefore the resonance points are at
$\rho_k = 1-k^2, k\in\N, k\ge 2$.

In the case when $\varrho=1$ we have $\mu_+ =\mu_-=-1$ and we only
obtain a quaternionic 1--dimensional space of parallel sections given
by $d_{\mu_\pm}^N$--parallel sections
\[
  \alpha=e^{\frac{iy}2} (1+j) e^{\frac{ ix}2} c, \qquad c\in\H\,,
\]
and $d\alpha =-df \beta$  since $\alpha_\pm^2 = \alpha_\pm^1j$
and $\alpha_+^l = \alpha_-^l, l=1,2$. The remaining quaternionic
1--dimensional space of $d_{\varrho=1}$--parallel sections is given by
\[
  \alpha = e^{\frac{iy}2}\left((1+j)e^{\frac{ i x}2} -
   iy (1-j) e^{-\frac{ix}2}\right)c\,, \qquad c \in\H\,,
\]
and $d\alpha =-df \beta$ as before. 
Note that these latter parallel sections are not sections with multiplier. \\
\end{example}

For $r\in\R_*$ and invertible $d_r$--parallel $\Phi$ the Calapso
transform $A_{\Phi, r}(f)$ is isothermic, see Theorem
\ref{thm:iso_ass_fam}, and it is well--known that this gives the
Lawson correspondence if the surface $f$ has constant mean
curvature. We recall:

\begin{theorem}[Lawson correspondence \cite{lawson}] 
  Let $f: M \to\R^3$ be a CMC surface. Then surfaces $A_{\Phi,r}(f)$
  in the associated family of $f$ for $r \in (0,1)$ are, up to
  M\"obius transformation, CMC surfaces in $S^3(\mathfrak r)$ for some
  $\mathfrak r\in\R$. In particular, if the initial condition is
  chosen so that $\mathfrak r= \frac 1{2\sqrt{r(1-r)}}$ then the mean
  curvature of $A_{\Phi,r}(f)$ in $S^3(\mathfrak r)$ is given by
\[
H^\Phi_{S^3(\mathfrak r)}  = 1- 2 r\,.
\]
\end{theorem} 
\begin{proof} For $r\in(0,1)$ let
  $\mu_\pm = 1-2r \pm 2i \sqrt{r(1-r)}$ be the associated CMC spectral
  parameter. Then we know that $\mu_\pm =a\pm i b\in S^1$ with real
  $a=1-2r, b= 2\sqrt{r(1-r)}.$ Choose two non--trivial parallel
  sections $d_{\mu_\pm}^N\alpha_\pm=0$, that is, see (\ref{eq:
    dalpha}),
\begin{equation}
\label{eq: parallel}
d\alpha_\pm =-\frac 12df(N \alpha_\pm(a-1) \pm
\alpha_\pm b) =-\frac 12df(N (a-1) \pm
b)\alpha_\pm\,. 
\end{equation}
Note that the sections $\alpha_\pm$ are nowhere vanishing since they
are $d_{\mu_\pm}^N$--parallel, and are quaternionic independent by
Remark \ref{rem:quaternionicindependence}.  Denote by
\[
f^\Phi= A_{\Phi,r} = -\alpha_+\invers\alpha_-
\]
the surface in the associated family given by, see Theorem
\ref{thm:Ttransform} and Theorem \ref{thm: drho parallel},
\[
\phi_\pm = \begin{pmatrix} 1 \\ \frac 12(N  (a-1)\pm
  b)
\end{pmatrix}\alpha_\pm 
\,.
\]

Since  $a, b\in\R$ we have
\[
df^\Phi= -\alpha_+\invers df b\alpha_-
\]
and the left and right normals are given by
$N^\Phi = \alpha_+\invers N\alpha_+$ and
$ R^\Phi= \alpha_-\invers N \alpha_-$. Thus
\[
N^\Phi f^\Phi= f^\Phi R^\Phi
\]
which shows by (\ref{eq:tangent space}) that
$f^\Phi \in \perp_{f^\Phi}$ is in the normal bundle of
$f^\Phi$. Therefore, $f^\Phi \in S^3(\mathfrak r)$ with
$\mathfrak r= \frac{|\alpha_-|}{|\alpha_+|}\in\R$.  Since $\alpha_\pm$
are $d_{\mu_\pm}^N$--parallel, see (\ref{eq: parallel}), we have
\[
dN^\Phi = \alpha_+\invers(dN + df(1-a + bN))\alpha_+
\]
and thus with (\ref{eq:Hdf})
\[
(dN^\Phi)'=\frac 12(dN^\Phi- N^\Phi *dN^\Phi)=  \alpha_+\invers df(-a+ N b) \alpha_+\,.
\]
Since  $df^\Phi H^\Phi=-(dN^\Phi)'$, see again (\ref{eq:Hdf}),  this shows
\[
H^\Phi= \alpha_-\invers(N-\frac ab)\alpha_+\,.
\]
The mean curvature of a surface $f$ in the 3--sphere is given by
$H_{S^3(\mathfrak r)} = \frac 1{\mathfrak r}\Re(fH)$, see
e.g. \cite{tori_revolution}, so that
\[
H^\Phi_{S^3(\mathfrak r)} =
\frac 1{\mathfrak r}\frac a {b}\,,
\]
and $f^\Phi$ has constant mean curvature in $S^3(\mathfrak r)$.  In
particular, if the initial conditions for $\alpha_\pm$ are chosen so
that $\mathfrak r= \frac 1{2\sqrt{r(1-r)}}=\frac 1b$ then
$H^\Phi_{S^3(\mathfrak r)} = a=1-2r$, which gives the usual
correspondence between the constant sectional curvature and the mean
curvature of Calapso transforms.
\end{proof}

\begin{figure}[H]
\includegraphics[height=3cm]{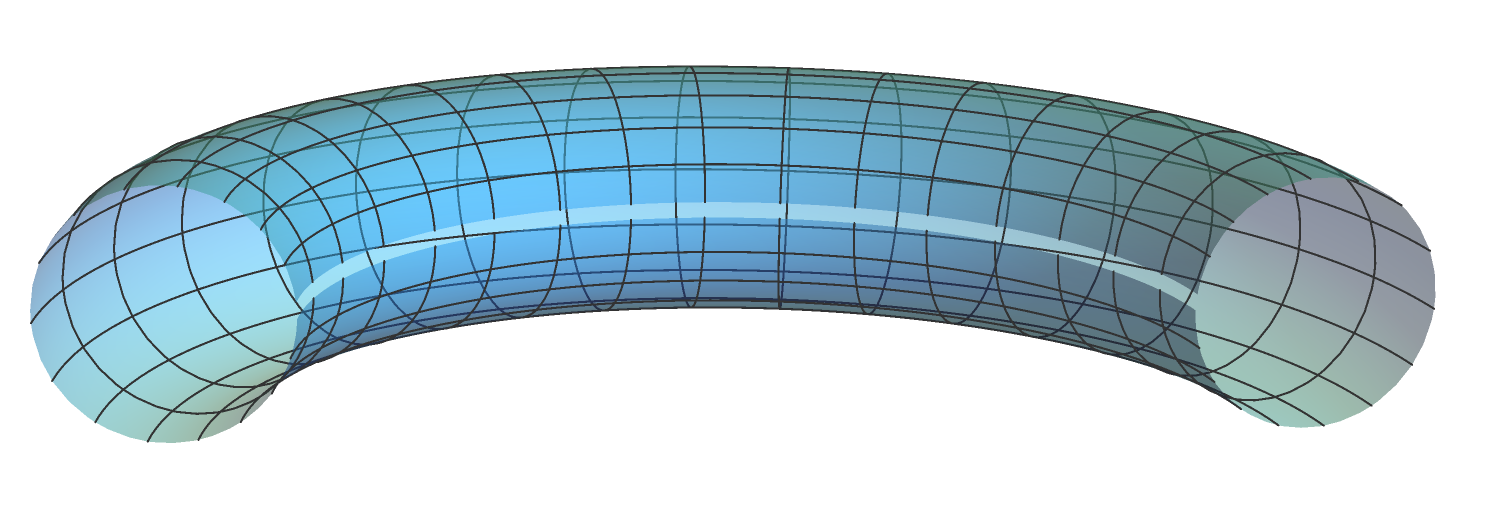}
\includegraphics[height=3cm]{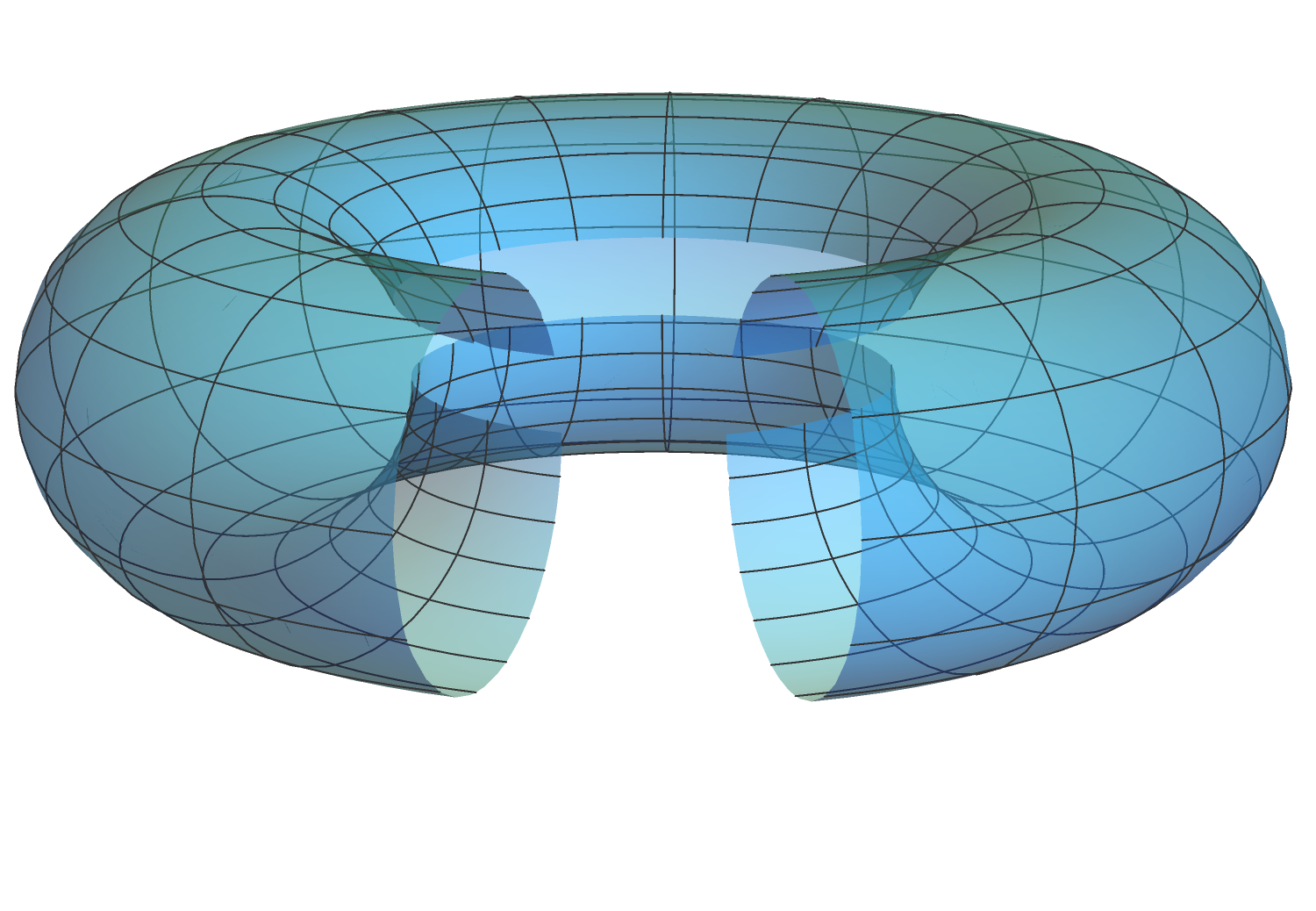}
\includegraphics[width=4.5cm]{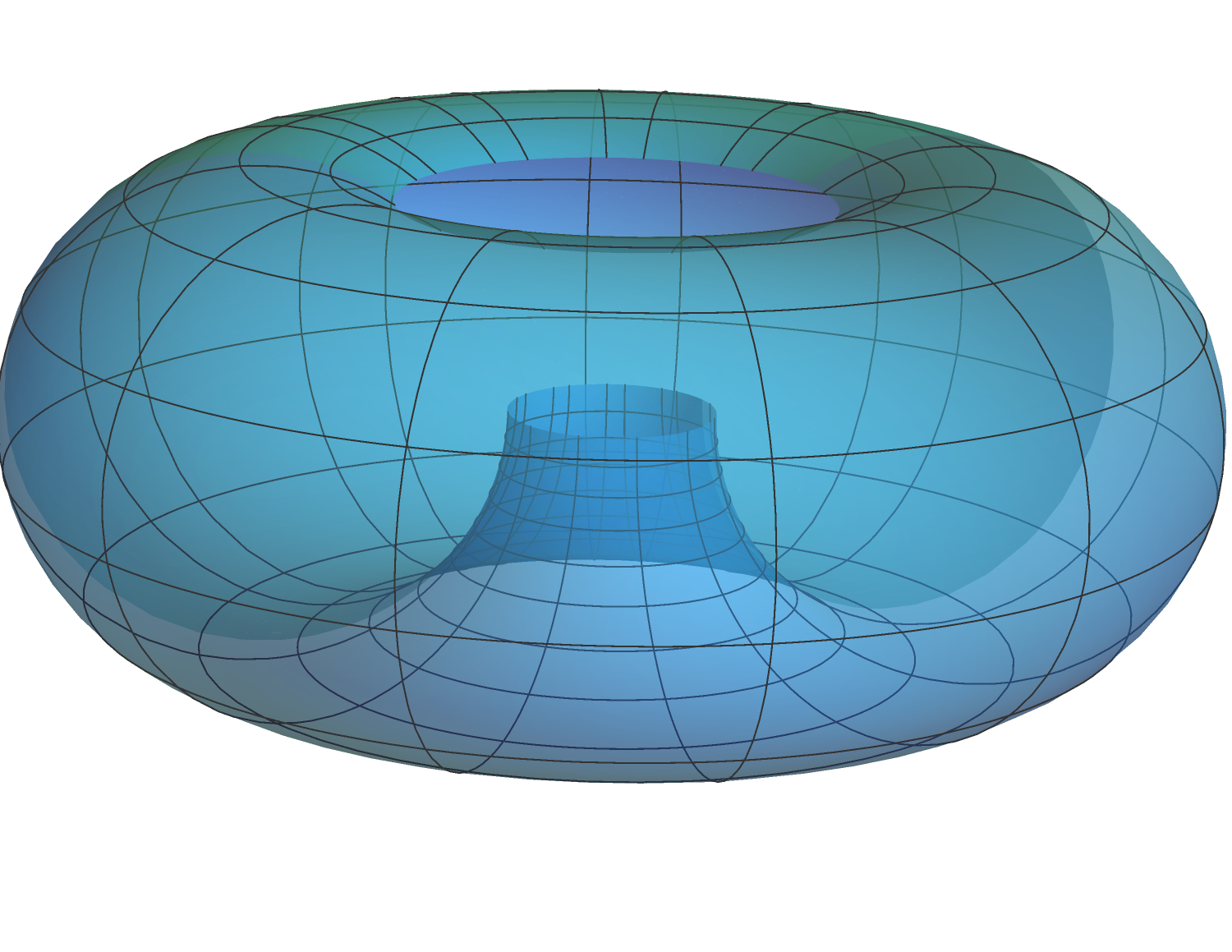}
\caption{Stereographic projections into 3--space of CMC surfaces
  $f^\Phi(x,y) = -\frac{e^{-iy \sqrt{1-r}}}{2 \sqrt{1-r}} + j
  \frac{e^{ix\sqrt r}}{2\sqrt r}$  in the
  isothermic associated family of the cylinder for $r= \frac 1{20}, r=\frac 15$
  and $ r=\frac 12$. Observe that the period of the homogenous tori
  differ from the period of the cylinder and indeed depend on the
  spectral parameter. }\end{figure}

We will now give a geometric interpretation of the Sym--Bobenko
formula for CMC surfaces as a limit of surfaces in the isothermic
associated family.  First we recall:

\begin{theorem}[Sym--Bobenko formula, \cite{sym_soliton_1985, sym_bob,
    simple_factor_dressing}]
  \label{thm:sym-bob}
  For fixed $\mu =e^{is}\in S^1$ and non--trivial $d^N_\mu$--parallel
  section $\alpha\in\Gamma(\ttrivial {})$ the map
  $N^\alpha = \alpha\invers N\alpha: M \to S^2$ is harmonic. Moreover,
  if $\alpha(t)$ is a $d_\lambda$--parallel, smooth extension in
  $\lambda=e^{it}$ of $\alpha = \alpha(s)$ near $\mu$, the CMC surface
  given by
\[
  A^N_{\alpha, s}(f) = -2\alpha\invers \frac{d}{dt}\alpha|_{t=s}
\]
has Gauss map $N^\alpha$. A surface in the family
$A^{N}_{\alpha, s}(f)$ is called an \emph{associated (CMC) surface} of
$f$.  Its family of flat connections $d_{\lambda}^{N^\alpha}$ is given
by
\begin{equation}
\label{eq:assoasso}
\alpha\cdot d_{\lambda}^{N^\alpha} = d^N_{\lambda\mu}\,.
\end{equation}
\end{theorem}

When considering non--trivial parallel sections $\alpha_{\mu_\pm}$ of
$d_{\mu_\pm}^N$ where $\mu_\pm =e^{\pm it}$ are near $\mu= 1$ then the
associated isothermic spectral parameter is $r = \frac{1-\cos t}2$,
and the associated $d_\varrho$--parallel sections
\[
\phi_\pm = \begin{pmatrix} \alpha_\pm \\ \frac 12(N \alpha_\pm(\cos
  t-1) \pm \alpha_\pm \sin t)
\end{pmatrix}
\]
give rise to an invertible matrix $M_t=(\phi_+, \phi_-)$, away from
$t=0$, and thereby to an isothermic surface $A_{\Phi, r}(f)$ in the
(isothermic) associated family of $f$. Although $M_t$ degenerates at
$t=0$ to a singular matrix, we will see that after a suitable M\"obius
transform the limit of $A_{\Phi,r}(f)$ for $r$ at zero still exists
and gives the CMC surface $f$:

\begin{theorem}
 \label{thm: limit1}
 Let $f: M \to \R^3$ be a CMC surface.  Let $\mu=e^{is}\in S^1$ be
 fixed, $\alpha$ a $d_\mu^N$--parallel section and
 $f^\alpha =A_{\alpha, s}^N(f)$ the corresponding CMC surface with
 Gauss map $N_\alpha=\alpha\invers N \alpha$. Denote by
 $A_{\Phi,r}(f)$ the associated family of isothermic surfaces for
 $r\in (0,1)$ of $f$ and by $r(\mu) = -\frac{(\mu-1)^2}{4\mu}$ the
 isothermic spectral parameter given by $\mu$.  Then up to M\"obius
 transformation
\[
\lim_{r\to r(\mu)}  A_{\Phi,r}(f)  =f^\alpha\,.
\]
\end{theorem}
\begin{figure}[H]
    \includegraphics[width=2cm]{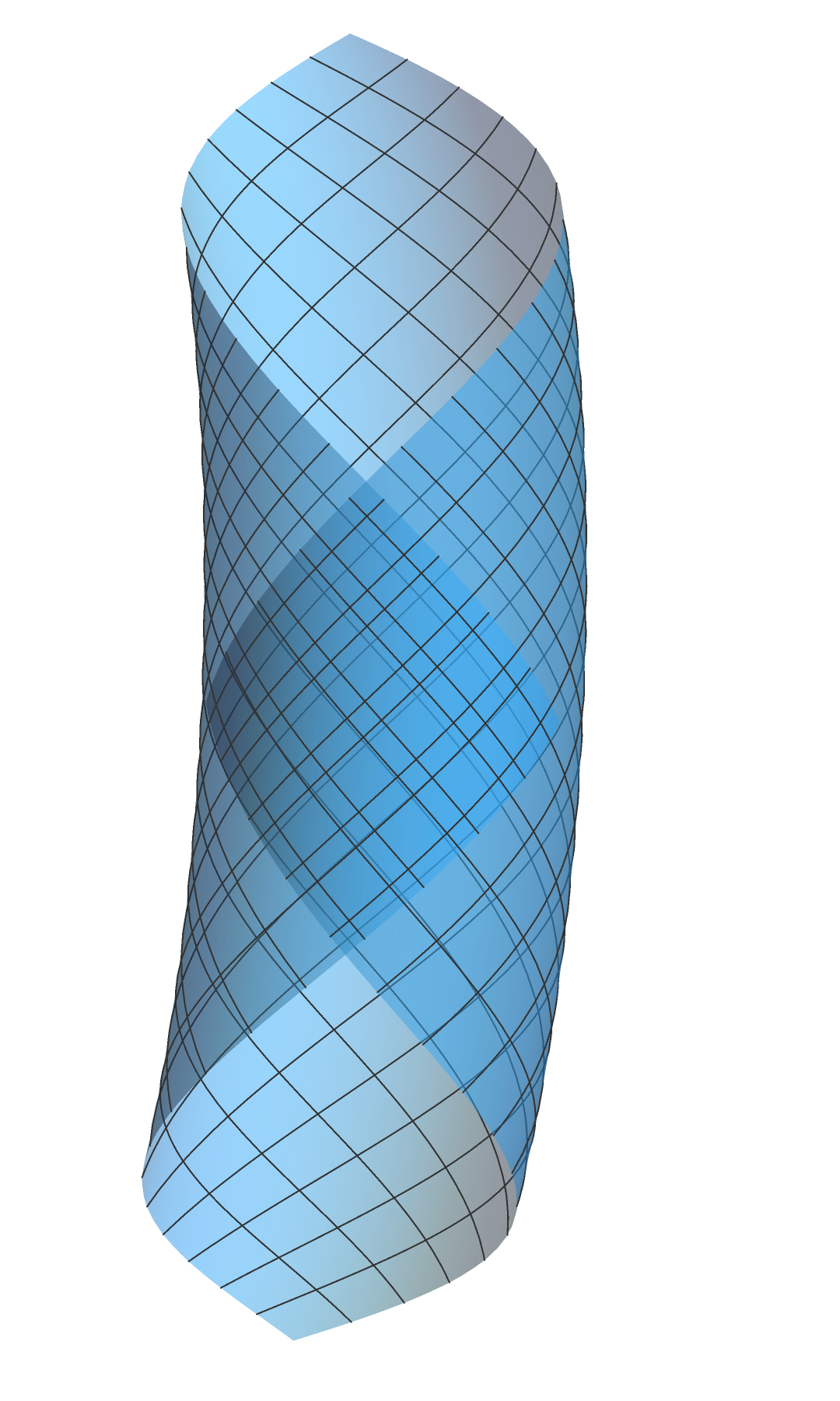}
    \includegraphics[width=2cm]{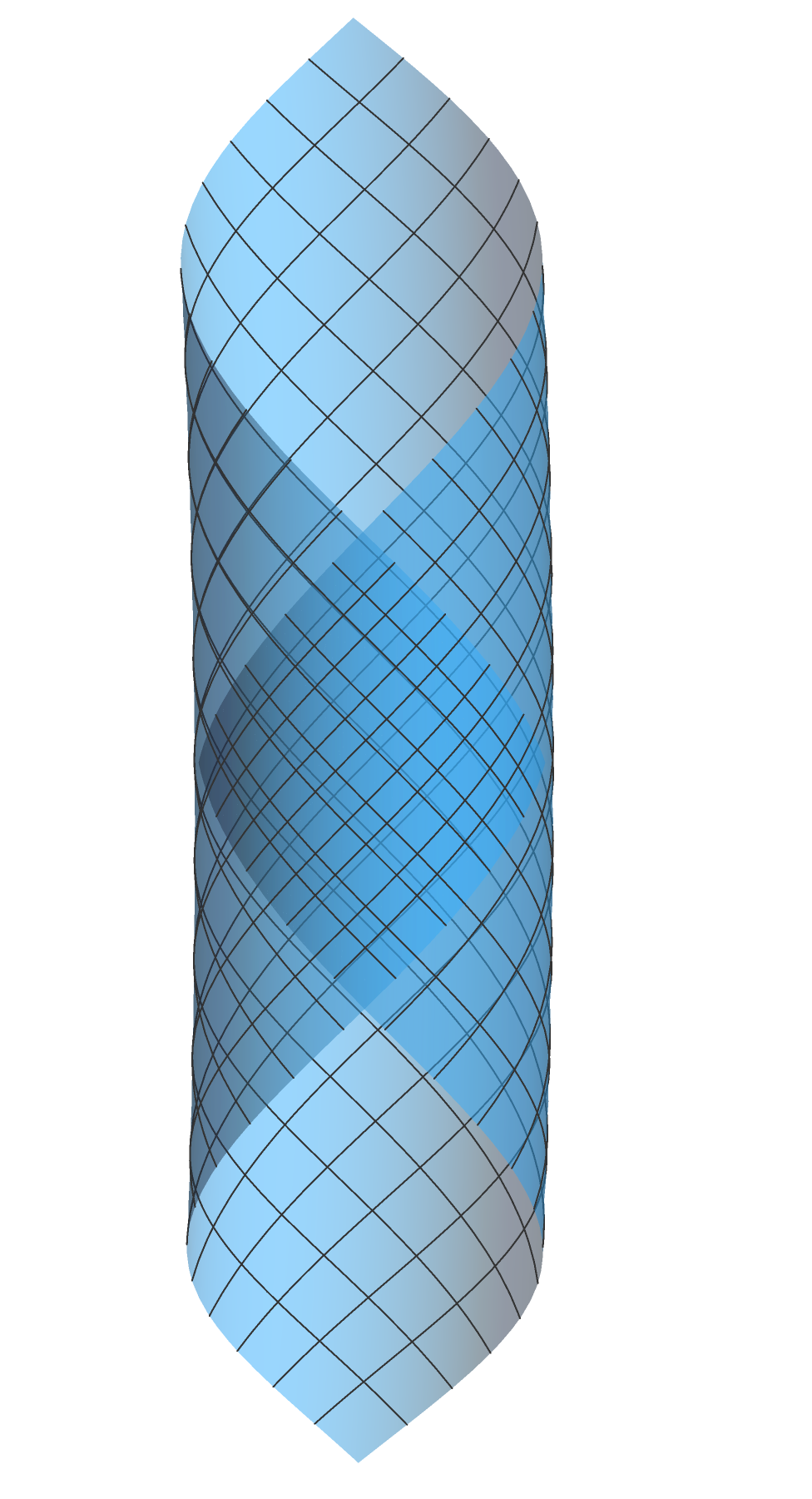}
    \includegraphics[width=2cm]{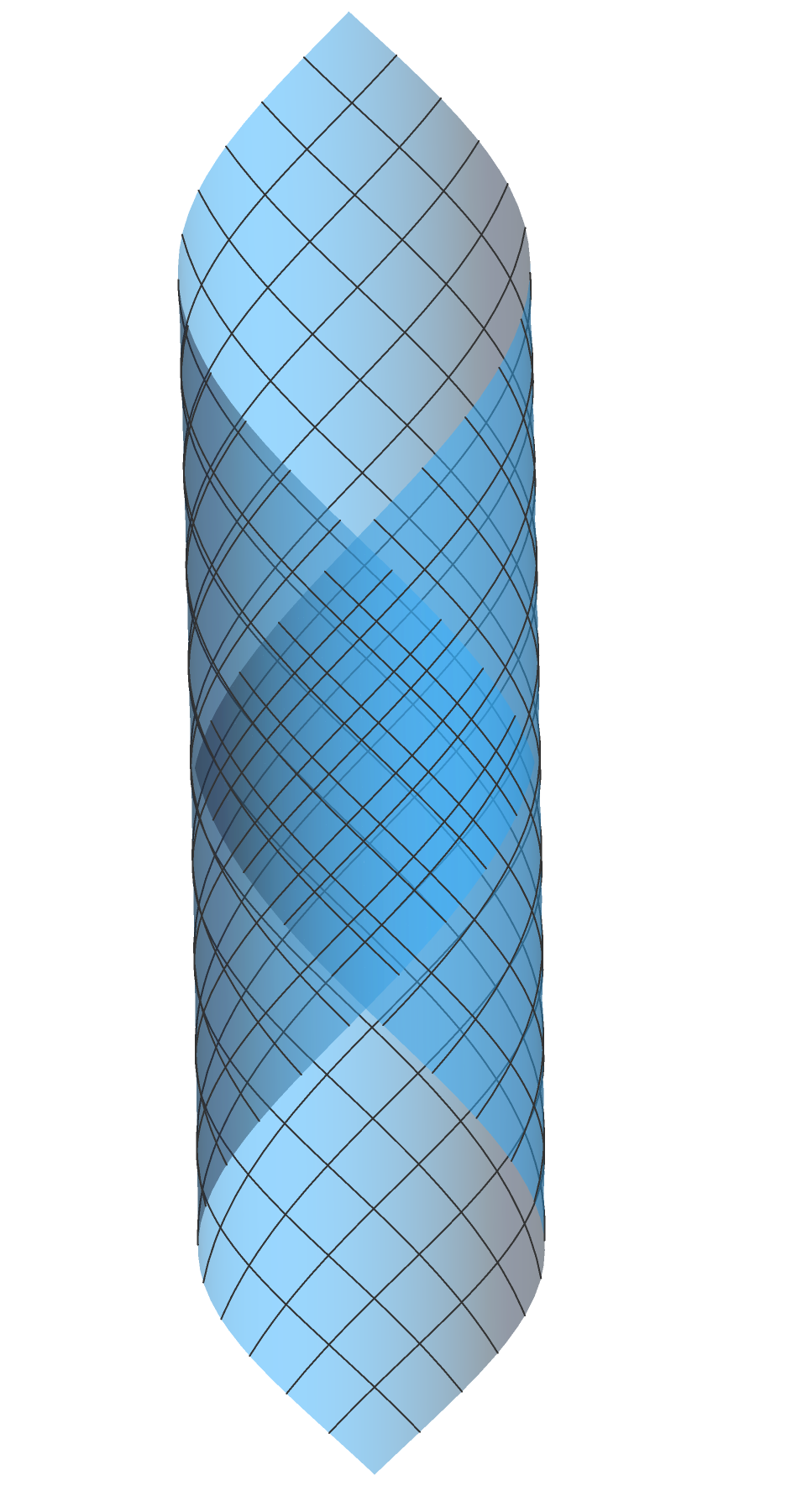}
    \includegraphics[width=2cm]{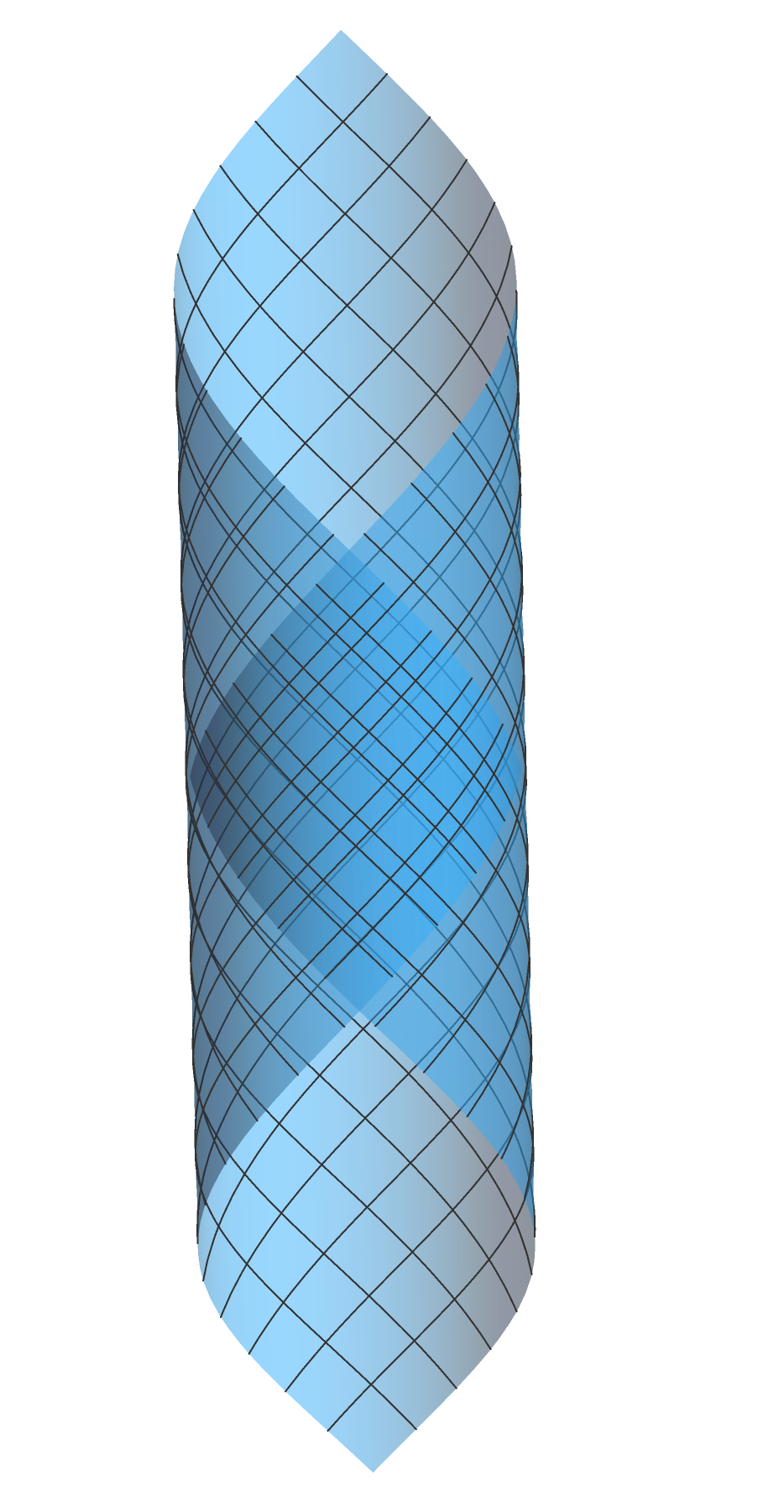}
      \includegraphics[width=2cm]{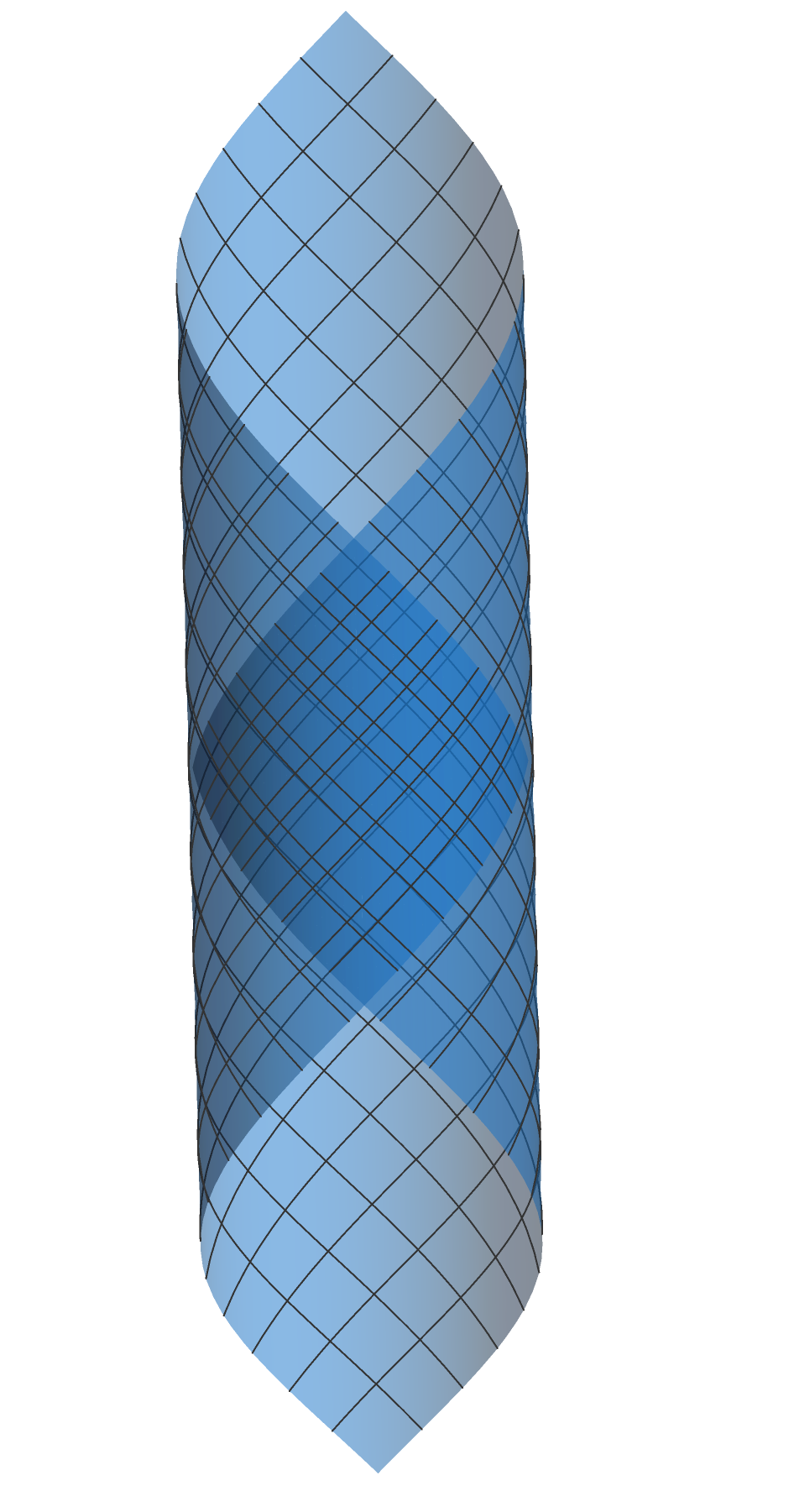}
\caption{Members of the isothermic associated family $A_{\Phi,
    r}(f^\alpha)$ of the CMC
  surface  $f^\alpha=A^N_{\alpha, \frac\pi
    2}(f)$  and their limit $f^\alpha$ for  $\alpha=\alpha_+(\frac \pi 2) $ and
  $r=\frac{1-\cos(t)}2$,  $t= \frac 12,
  \frac 14, \frac 1{8}, \frac 1{100}$, orthogonally projected into
  3--space. }\end{figure}
\begin{proof}
  We first consider the case when $\mu=1$ and $\alpha=1$ is the
  constant $d_{\mu=1}$--parallel section. Then the associated
  isothermic spectral parameter is $r(\mu) = 0$, and $f^\alpha=f$. For
  $r\in(0,1)$ write the associated CMC spectral parameter as
  $\mu_\pm = 1-2r \pm 2i \sqrt{r(1-r)} = e^{\pm it}\in S^1$. Let
  $\alpha_\pm$ be non--trivial $d_{\mu_\pm}^N$--parallel sections in
  the neighbourhood of $\mu=1$ which depend smoothly on $\mu_\pm$,
  with $\alpha_\pm(\mu_\pm =1) =1$, and
  $\beta_\pm =-N\alpha_\pm r \pm \alpha_\pm \sqrt{r(1-r)}$. Then for
  $r\in(0,1)$ the sections
\[
\phi_1 = \begin{pmatrix} \alpha_+ \\ \beta_+
\end{pmatrix}, \phi_2 =\frac 1t \begin{pmatrix} \alpha_+-\alpha_- \\
  \beta_+ -\beta_-
\end{pmatrix}
\]
are $d_r$--parallel and linearly independent over $\H$ since
$\alpha_\pm$ are linearly independent over $\H$, see Remark
\ref{rem:quaternionicindependence}.  Put $\Phi = FM$ where
\[
M= (\phi_1, \phi_2) = \begin{pmatrix} \alpha_1 & \alpha_2\\ \beta_1 &
  \beta_2
\end{pmatrix}, \quad F = \begin{pmatrix} 1 & f \\ 0&1
\end{pmatrix}\,,
\] 
and consider the associated family of isothermic surfaces
$A_{\Phi, r}(f)$ given by Theorem \ref{thm:iso_ass_fam} in the affine
coordinate as
\[
A_{\Phi,r}(f) = -\alpha_1\invers\alpha_2\,.
\]

Now, $\mu_\pm = e^{\pm it}\in S^1$ gives
$1-2r = \Re(\mu_\pm) = \cos t$ and thus
\begin{align*}
\lim_{r\to 0} A_{\Phi,r}(f)&= \lim_{t\to 0} A_{\Phi,\frac{1-\cos t}2}(f)\,.
\end{align*}
 
By assumption $\lim_{t\to 0} \alpha_1 =1$ and it remains to show that,
up to translation, $\lim_{t\to 0}\alpha_2=- f$.

Since $\lim_{t\to 0}(\alpha_+ -\alpha_-) =0$ we can apply
L'H\^opital's rule to obtain
\[
\lim_{t\to 0} \alpha_2 = \lim_{t\to 0} \frac 1t(\alpha_+ -\alpha_-) =
\frac{d}{dt} (\alpha_+ - \alpha_-)|_{t=0}\,.
\]
We conclude by applying the Sym--Bobenko formula  
\begin{equation}
\label{eq:geom sym}
\lim_{r\to 0} A_{\Phi,r}(f) = -
\frac{d}{dt} (\alpha_+ - \alpha_-)|_{t=0} = f\,.
\end{equation} 

The general case now follows by applying the same arguments to the CMC
surface $f^\alpha =A^N_{\alpha, s}(f) $ and its associated family of
isothermic surfaces $ A_{\Phi^\alpha, r}(f^\alpha)$ to obtain
\[
\lim_{r\to 0} A_{\Phi^\alpha, r}(f^\alpha)= f^\alpha\,.
\]

Since for $\lambda\in S^1$ parallel sections of $d_\lambda^N$ and
$d_{\lambda}^{N_\alpha}$ correspond via the gauge relation
(\ref{eq:assoasso}) we can write a $d_{\lambda}^{N_\alpha}$--parallel
$\Phi^\alpha$ in terms of
 \[
\alpha\invers\alpha_i^{\mu\lambda}\,, \quad i=1,2\,,
\]
where $\alpha_i^{\mu\lambda}$ are $d_{\mu\lambda}^N$--parallel
sections and $\mu\lambda\in S^1$. Therefore using the isothermic
spectral parameter
$r(\lambda\mu) = -\frac{(\lambda\mu-1)^2}{4\lambda\mu}$ we see
\[
A_{\Phi^\alpha, r(\lambda)}(f^\alpha)= - (\alpha\invers\alpha_1^{\mu\lambda})\invers
(\alpha\invers\alpha_2^{\mu\lambda})
= A_{\Phi, r(\lambda\mu)}(f)
\]
 so that
\[
f^\alpha =\lim_{r\to 0} A_{\Phi^\alpha, r} (f^\alpha)= \lim_{r\to
  r(\mu)}A_{\Phi, r}(f)\,.
\]

\end{proof}

\begin{rem}
We note that 
\[
M=M_t \begin{pmatrix} 1 & \frac 1t\\ 0 & -\frac 1t
\end{pmatrix} 
\]
for $M_t=(\phi_+, \phi_-)$ and $M =(\phi_1, \phi_2)$ so that the
surface in the associated family given by $(FM_t)\invers L$ is
M\"obius equivalent to the one given by $(F M)\invers L$. Moreover,
the equation (\ref{eq:geom sym}) shows that the Sym--Bobenko
derivative is indeed geometrically given as the limit of the
isothermic associated surfaces of $f$.
\end{rem}
`
\subsection{The $\mu$--Darboux transformation of CMC surfaces}

In \cite{cmc, simple_factor_dressing} it is shown that the harmonic
Gauss map of a CMC surface gives rise to a Darboux transformation
which preserves the harmonicity of the Gauss map, and thus the CMC
property (up to translation in $\R^4$).  We will now discuss how this
so--called \emph{$\mu$--Darboux transformation} is related to the
isothermic $\varrho$--Darboux transformation.

We first investigate which $\varrho$--Darboux transforms of a CMC
surface $f: M\to\R^3$ have constant mean curvature.  Recall that a
classical Darboux transform $\hat f = f+T: \tilde M \to\R^3$ with
$dT= -df + Tdg r T$, $r\in\R$, of a CMC surface $f: M\to\R^3$ has
constant mean curvature if and only if $|T-N|^2=1-\frac 1r$, see
\cite{darboux_isothermic}.  We generalise this to give a condition for
a general $\varrho$--Darboux transform to have constant mean
curvature.

\begin{theorem}
  Let $f: M \to\R^3$ be a CMC surface, $H=1$, with dual surface given
  by the parallel CMC surface $g=f+N$. Let
  $\phi=\begin{pmatrix}\alpha\\ \beta
\end{pmatrix}
$ be a non--trivial $d_\varrho$--parallel section,
$\varrho\in\C\setminus\{0,1\}$.  Then
$\hat f = f + \alpha\beta\invers: \tilde M \to \R^4$ is a
$\varrho$--Darboux transform of $f$.

Moreover, $\hat f$ has constant real part and its imaginary part is a
CMC surface in $\R^3$ with mean curvature $\hat H =1$ if and only if
\begin{equation}
\label{eq:cmc condition}
(T^d+N)^2=\hat
\varrho\invers-1\,, 
\end{equation}
where $T^d = \beta \varrho\invers \alpha\invers$ and  $\hat
\varrho=\alpha\varrho\alpha\invers$. 
\end{theorem}
 
\begin{rem} Note that in contrast to the case of a real parameter this
  is not a condition on distance since in general the difference
  $T^d = \hat g - g$ between the parallel surface and its Darboux
  transform can take values in 4--space. Moreover, the right hand side
  $\hat\varrho\invers-1$ is not constant for
  $\varrho\in\C\setminus\R$.
\end{rem}

\begin{proof}
We first observe that by Theorem \ref{thm: both parallel} if for a
$d_\varrho$--parallel section $\phi = \begin{pmatrix} \alpha \\ \beta
\end{pmatrix}
$ one of the parallel sections $\alpha$ or $\beta$ vanishes at $p$
then it vanishes identically, and by \eqref{eq:iso diff} both parallel
sections $\alpha$ and $\beta$ are identical zero, contradicting the
assumption that $\phi$ is non--trivial. Hence
$T=\alpha\beta\invers: \tilde M \to \H_*$ and $\hat f = f+T$ is a
$\varrho$--Darboux transform. With $T=\alpha\beta\invers$ the
differential of $\hat f$ is given by (\ref{eq:dhatf}) as
\[
d\hat f  =  Tdg \hat\varrho T\,,
\]
and the left normal of $\hat f$ is $\hat N =-TNT\invers$.  By
(\ref{eq:Hdf}) the Darboux transform $\hat f$ has constant real part
and mean curvature $\hat H =1$ if and only $(d\hat N)' = -d\hat
f$. Therefore, we now compute
$(d\hat N)' = \frac 12(d\hat N - \hat N*d\hat N)$ for
$\hat N =- TNT\invers$.  First we have
\[
d\hat N = [dT T\invers, \hat N] - TdNT\invers\,.
\]
From the Riccati type equation (\ref{eq:gen riccati}) we obtain
\[
dTN+*dT = Tdg(\hat \varrho TN +N\hat \varrho T)\,.
\]
Therefore, with $dg =(dN)''$, we see
\begin{align*}
(d\hat N)' &= -\frac 12Tdg(\hat \varrho TN +N\hat \varrho T - N (\hat\varrho TN +
N\hat \varrho T)N +2) T\invers\\
& =- Tdg(\hat \varrho TN + N\hat \varrho T
+1)T\invers
\end{align*}
so that $(d\hat N)'= -d\hat f = -Tdg\hat \varrho T $ if and only if
\begin{equation}
\label{eq: cmc T}
1 + N \hat \varrho T + \hat \varrho T N = \hat \varrho T^2\,.
\end{equation} 
Since $(\hat\varrho T)\invers = T^d$, the latter is equivalent to
(\ref{eq:cmc condition}).
\end{proof}

Recall \cite{cmc} that for fixed $\mu\in\C_*$ a $\mu$--Darboux
transform $D^N_{\alpha, \mu}(f)$ is given by a $d^N_\mu$--parallel
section $\alpha\in\Gamma(\widetilde{\trivial{}})$: if
\[ d\alpha = -\frac 12df(N\alpha(a-1)+\alpha b)
\]
where $a=\frac{\mu+\mu\invers}2, b=i\frac{\mu\invers-\mu}2$ then
$e\alpha$ is a holomorphic section and $\varphi =e\alpha+\psi\beta$
with $\beta=\frac 12(N\alpha(a-1) + \alpha b)$ is the prolongation of
$e\alpha$. Thus,
\[
D^N_{\alpha,\mu}(f) = f + T: \tilde M \to \H
\]
is a Darboux transform of $f$, a so--called \emph{$\mu$--Darboux
  transform}, where
\[
T\invers = \frac 12(N(\hat a-1) + \hat b)
\]
and $\hat a = \alpha a\alpha\invers, \hat b = \alpha b\alpha\invers$.
Note that indeed $T\invers(p) \not=0$ for all $p$ since otherwise
$N= \hat b(1-\hat a)\invers$ at some $p$ which leads to the
contradiction $-1=N(p)^2= \frac{a+1}{1-a}$.

Since $\varphi= e\alpha + \psi \beta$ is $d_\varrho$--parallel for
$\varrho=\frac{1-a}2$ by Theorem \ref{thm: drho parallel} the
$\mu$--Darboux transformation is a special case of the
$\varrho$--Darboux transformation \cite{cmc}:

\begin{theorem}
\label{thm:rhoequalsmu}
Let $f: M \to \R^3$ be a constant mean curvature surface in $\R^3$
with mean curvature $H=1$ and dual surface $g=f+N$ where $N$ is the
Gauss map of $f$. Then any $\varrho$--Darboux transform
$\hat f: \tilde M \to\R^4$, $\varrho\in \C\setminus\{0,1\}$, is given
by
\[
\hat f = f+(\alpha_++\alpha_-)(\beta_++\beta_-)\invers
\]
where $d_{\mu_\pm}^N\alpha_\pm =0$ and
$\beta_\pm = \frac 12(N\alpha_\pm(a-1)\pm \alpha_\pm b)$ and
$\mu_\pm = a \pm i b = 1-2\varrho \pm 2 i \sqrt{\varrho(1-\varrho)}$.

Moreover, $\mu$--Darboux transforms of $f$ are exactly those
$\varrho$--Darboux transforms of $f$ which have constant mean
curvature in 3--space (possibly after translation).
\end{theorem}
\begin{proof}
  The first statement follows directly from Theorem \ref{thm: drho
    parallel}. Now assume that $\hat f: \tilde M \to\R^4$ is a
  $\varrho$--Darboux transform with constant real part and
  $\Im(\hat f)$ has constant mean curvature. Then
  $\hat f = f+\alpha\beta\invers$ with
  $\alpha=\alpha_+ +\alpha_-, \beta=\beta_++\beta_-$ and $\alpha$ is
  nowhere vanishing.

  Using $a^2+b^2=(1-2\varrho)^2 + 4\varrho(1-\varrho) =1$ we have
  $\varrho\invers-1 = \frac{b^2}{(a-1)^2}$ so that
\[
\hat \varrho\invers -1 = (\alpha_++\alpha_-)\frac{b^2}{(a-1)^2}
(\alpha_++\alpha_-)\invers\,.
\]
Therefore, the condition $(T^d+N)^2= \hat\varrho\invers-1$ for
$\hat f$ to have constant mean curvature (\ref{eq:cmc condition})
gives
\[
(\alpha_++\alpha_-)\invers(T^d+N)^2 (\alpha_++\alpha_-)=
\frac{b^2}{(a-1)^2}\,.
\]
Since $\frac{b^2}{(a-1)^2} =\varrho\invers-1 \not=0$ and $(\alpha_++\alpha_-)\invers(T^d+N)^2
(\alpha_++\alpha_-)$ is smooth, we obtain
\begin{equation}
\label{eq:alphas}
(\alpha_++\alpha_-)\invers(T^d+N) (\alpha_++\alpha_-)=
\pm\frac{b}{a-1}\,.
\end{equation}
On the other hand,  we have 
\[
\beta = \beta_++\beta_-= \frac 12(N\alpha(a-1)+(\alpha_+-\alpha_-)b)
\]
so that
\[
T^d +N=\beta\varrho\invers\alpha\invers + N =
 (\alpha_--\alpha_+)\frac{b}{a-1}(\alpha_++\alpha_-)\invers\,.
\]
Comparing to the previous equation (\ref{eq:alphas}) we see that
$\alpha_+=0$ or $\alpha_-=0$.  In this case,
$\hat f = f+ \alpha_\pm \beta_\pm\invers$ is a $\mu$--Darboux
transform.

Conversely, see \cite{cmc}, let $f^\mu=f+\alpha\beta\invers$ be a
$\mu$--Darboux transform where $d_\mu^N\alpha=0$ and
$\beta=\frac 12(N\alpha(a-1)+\alpha b)$ as before,
$a=\frac{\mu+\mu\invers}2, b=i\frac{\mu\invers-\mu}2$.  Let
$\varrho=\frac{1-a}2$ then
\begin{equation}
\label{eq:td}
T^d = \beta \varrho\invers \alpha\invers = - (N + \alpha \frac{b}{a-1}
\alpha\invers)
\end{equation}
and the CMC condition (\ref{eq:cmc condition}) follows with
$a^2+b^2=1$. Thus, every $\mu$--Darboux transform has constant mean
curvature.
\end{proof}

\begin{figure}[H]
  \includegraphics[height=4cm]{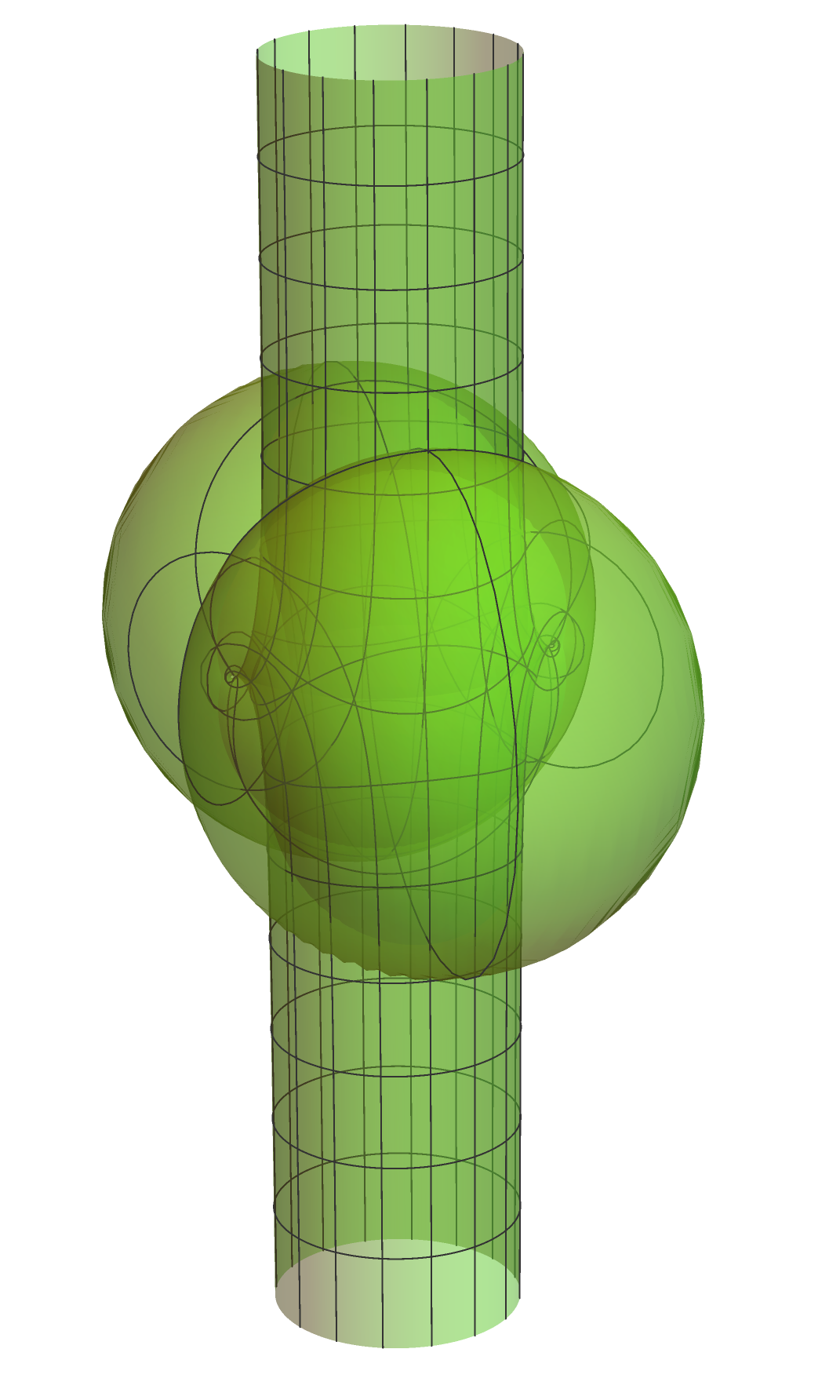}
\includegraphics[height=4cm]{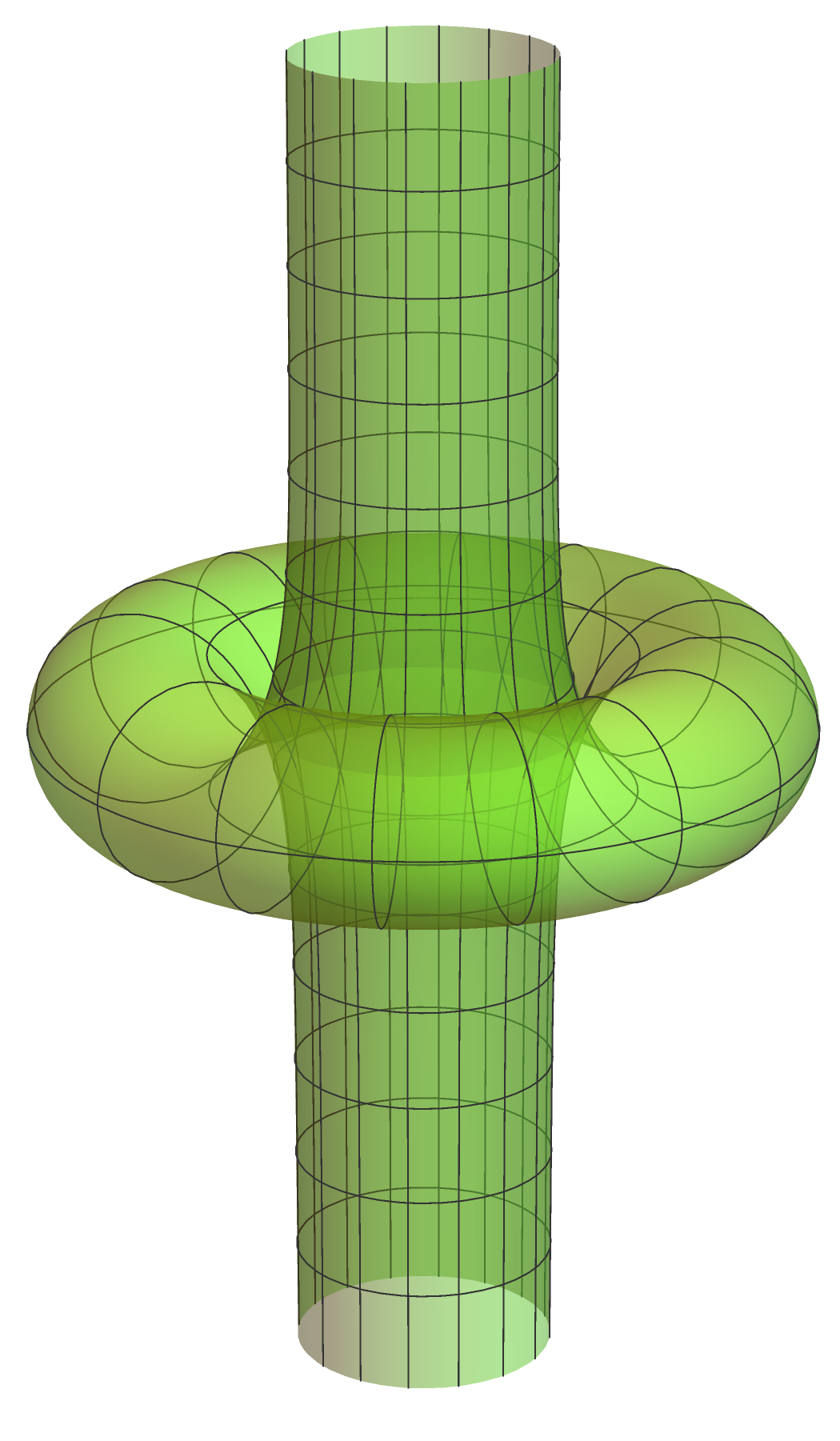}
\includegraphics[height=4cm]{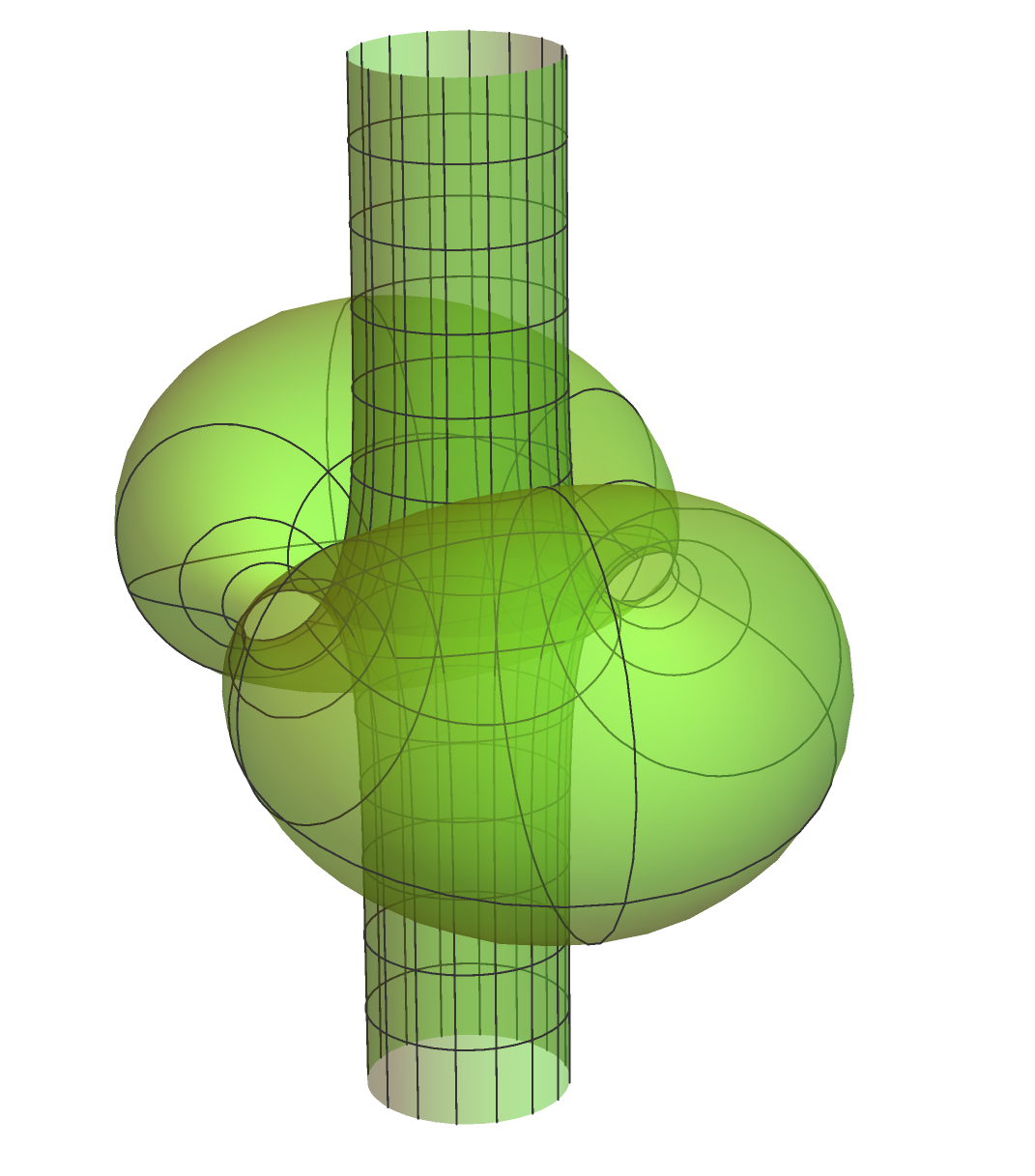}
\includegraphics[height=4cm]{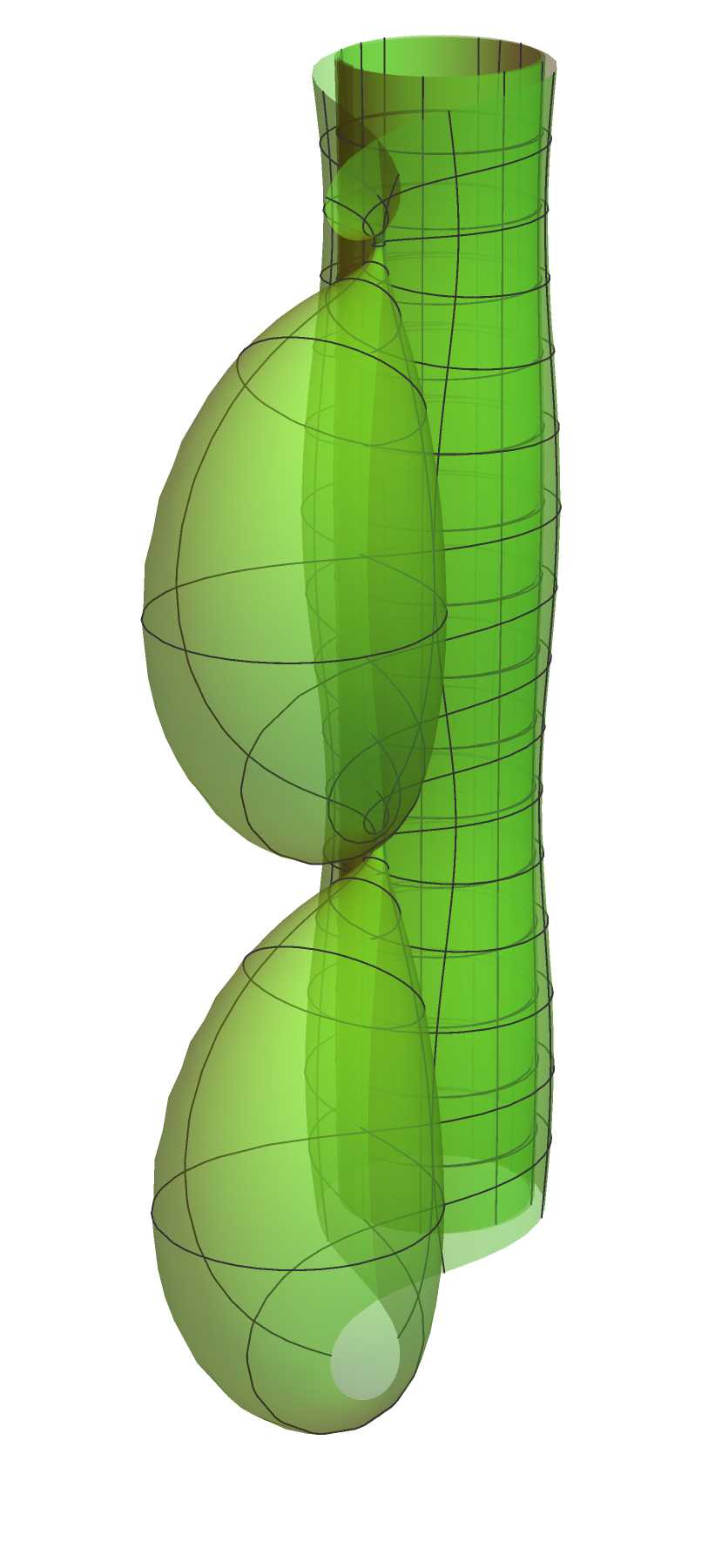}
\includegraphics[height=4cm]{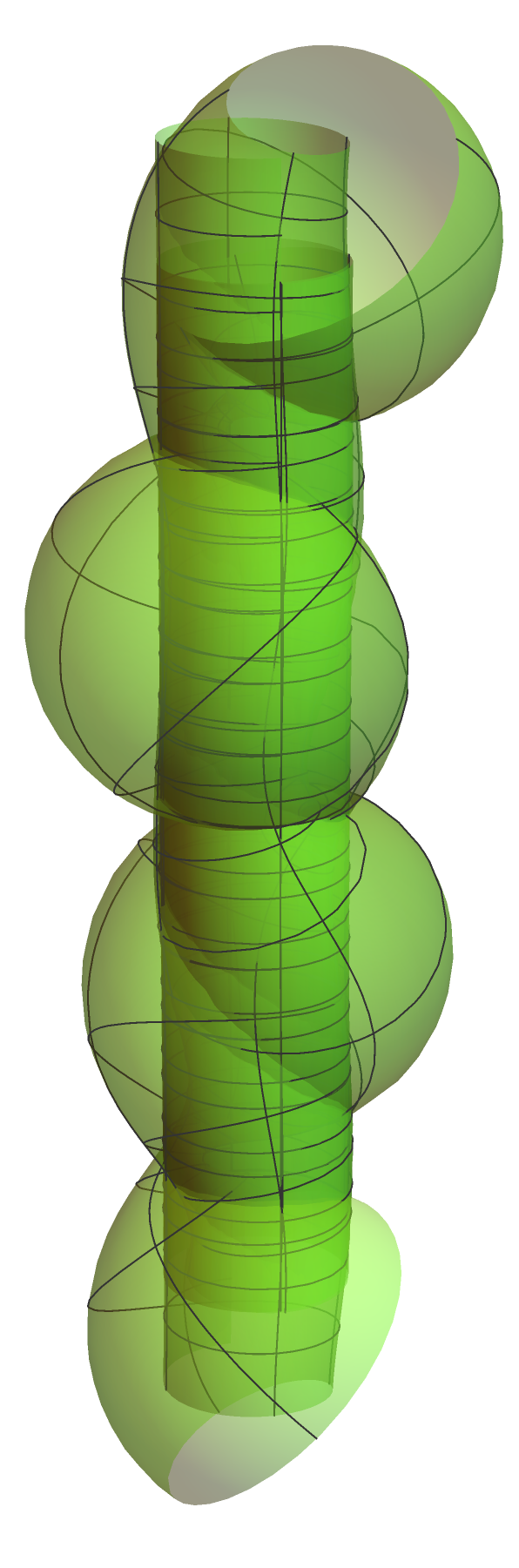}
\includegraphics[height=4cm]{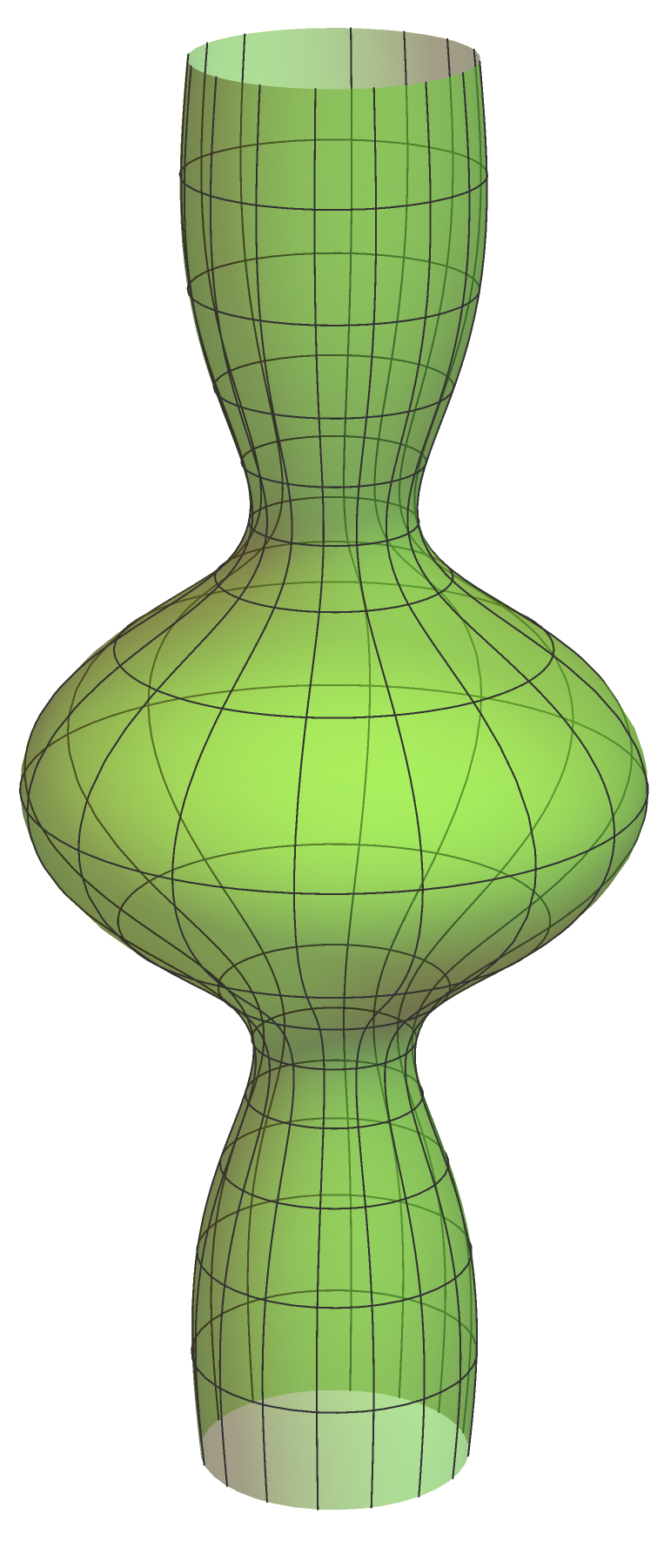}

\caption{The first CMC bubbleton is a $\mu$--Darboux transform of a
  cylinder at the resonance point $\mu= 7-4\sqrt 3$ (which gives
  $\varrho=\frac{1-a}2=-3$). The first surface of revolution is a
  $\varrho$--Darboux transform at $\varrho=-3$ which is not CMC. The
  next surface is an isothermic bubbleton, a $\varrho$--Darboux
  transform at the resonance point $\varrho=-3$ whereas the next one
  is a $\varrho$--Darboux transform at $\varrho=1$, which is given by
  $d_{\varrho=1}$--parallel sections which are not obtained from
  $d_{\mu_\pm}$--parallel sections only (here $\mu_\pm = -1$). The next surface
  is a $\mu$--Darboux transform at $\mu = -(1+2i) + 2 \sqrt{i-1}$
  (that is $\varrho=1+i$), that is a CMC surface.  The last surface of
  revolution is a $\varrho$--Darboux transform at $\varrho =1+i$ with
  non--constant real part, orthogonally projected into 3--space.}
\end{figure}

\begin{rem}
  \label{rem: parallel is Darboux}
  We note that for $\mu =-1$ for any $d_{\mu}^N$--parallel section
  $\alpha$ we have $\beta =-N\alpha$ so that the corresponding
  $\mu$--Darboux transforms are the parallel CMC surface
  $f^\mu = f + \alpha\beta\invers = f+N =g$, see
  \cite{darboux_isothermic}.
\end{rem}

We can now express the $\mu$--Darboux transform of the parallel
surface in terms of the CMC surface $f$, the spectral parameter and
parallel sections of the associated family of $f$:

\begin{cor} 
\label{cor:parallel mu dt}
Let $f: M \to\R^3$ be a CMC surface with dual surface given by the
parallel CMC surface $g= f+N$.

  If $d_\mu^N\alpha=0$, $\mu\in\C\setminus\{0,1\}$, such that
  $f^\mu= f+\alpha\beta\invers$ is a $\mu$--Darboux transform of $f$
  given by $\alpha$ and
  \[
    \beta = \frac 12(N\alpha(a-1)+ \alpha b)\,,
  \]
  that is $f^\mu=D^N_{\alpha,\mu}(f)$, then
  \[
g^\mu= g + \beta\varrho\invers\alpha\invers = f -\alpha\frac{b}{a-1} \alpha\invers
\] is   a $\mu$--Darboux
  transform of $g$ with  parameter
  $\mu$ and $d_\mu^{-N}$--parallel section $\beta$, that is,
  \[
    g^\mu = D^{-N}_{\beta,\mu}(g)\,.
  \]
  Here, $a= \frac{\mu+\mu\invers}2, b= i\frac{\mu\invers-\mu}2$ and
  $\varrho=\frac{1-a}2$ is the isothermic spectral parameter given by
  $\mu$.

  Moreover, $g^\mu$ is the parallel CMC surface of the $\mu$--Darboux
  transform $f^\mu$ of $f$, that is,
  \[
    g^\mu = f^\mu + N^\mu
\]
where $N^\mu$ is the Gauss map of $f^\mu$.
    
\end{cor}
\begin{proof}
  The Gauss map of $g$ is $N_g=-N$.  Since $\alpha$ is
  $d_\mu^N$--parallel if and only if $\alpha_g=\beta$ is
  $d^{-N}_\mu$--parallel and
  $d^{-N}_\mu \beta = d\beta+ dg \alpha \varrho$ by Theorem \ref{thm:
    both parallel}, we see that
  \[ \beta_g = \alpha\varrho
  \]
  and hence
\[
g^\mu
= g+ \alpha_g\beta_g\invers = g+\beta\varrho\invers \alpha\invers
\]
is a $\mu$--Darboux transform of $g$.  Moreover, (\ref{eq:td}) now
allows to simplify
\[
g^\mu = f - \alpha\frac{b}{a-1}\alpha\invers\,.
\]
Since $f^\mu= f+ T$ is CMC by \eqref{eq: cmc T} we see that
$T=\alpha\beta\invers$ satisfies
\[
  \hat \varrho T(1-NT\invers) = T\invers + N\hat \varrho = \frac 12
  \alpha b\alpha\invers 
\]
so that the parallel CMC surface of $f^\mu$ is given by
\[
  f^\mu + N^\mu = f + T(1-NT\invers) = f+\alpha\frac{b}{1-a}\alpha\invers =g^\mu
\]
where we used that $N^\mu =- TN T\invers$ by \eqref{eq:dhatf}.

\end{proof}

We also obtain that Bianchi permutability preserves the CMC condition:
\begin{theorem}
  \label{thm: bianchi cmc}
  Let $f: M\to\R^3$ be a CMC surface and
  $\alpha_l\in\Gamma(\ttrivial {})$ parallel sections with respect to
  the flat connections $d_{\mu_l}^N$ where
  $\mu_l\in\C\setminus\{0, 1\}$, $l=1,2$. Let
  $f_l = f+ \alpha_l\beta_l\invers$ be the CMC Darboux transforms
  given by $\alpha_l$ and
  \begin{equation}
    \label{eq:betai}
    \beta_l=\frac 12(N\alpha_l(a_l-1) + \alpha_l b_l)
  \end{equation}
  where
  $a_l=\frac{\mu_l +\mu_l\invers}2,
  b_l=i\frac{\mu_l\invers-\mu_l}2$. Denote by
  $\varrho_l = \frac{1-a_l}2$ the associated isothermic spectral
  parameter.

  Then the common $\varrho$--Darboux transform $\hat f$ of $f_1$ and
  $f_2$ which is given by the two corresponding
  $d_{\varrho_l}$--parallel sections
  $\phi_l=\begin{pmatrix} \alpha_l \\ \beta_l
  \end{pmatrix}
  $ via Bianchi permutability is a CMC surface, and hence a common
  $\mu$--Darboux transform of $f_1$ and $f_2$.
\end{theorem}
\begin{proof}
  By Bianchi permutability, Theorem \ref{thm: gen bianchi}, we have
  that $\hat f = f_1 + \alpha\beta\invers$ is given by
 \[
   \alpha=\alpha_2-\alpha_1\beta_1\invers \beta_2, \qquad
   \beta=\beta_2-\beta_1\varrho_1\invers\alpha_1\invers
   \alpha_2\varrho_2\,.
 \]
 In particular, this gives
 \begin{equation}
   \label{eq:drhopar}
   d\alpha = -df_1\beta\,.
 \end{equation}
 Since $f_1=f+ T_1$, $T_1= \alpha_1\beta_1\invers$, is CMC (possibly
 after translation) its Gauss normal $N_1=- T_1 N T_1\invers$ is
 harmonic and the connections in its associated family are flat.  We
 show that $\alpha$ is $d_{\mu_2}^{N_1}$--parallel, that is,
 \[ d\alpha = -df_1 \frac 12(N_1\alpha(a_2-1) + \alpha b_2)\,.
   \]
   Put differently, we need to show by (\ref{eq:drhopar}) that
   \[
     \beta =\frac 12(N_1\alpha(a_2-1)+\alpha b_2)\,.
   \]
   We easily compute with (\ref{eq:betai})  that
   \[
     \beta = \beta_2-\beta_1\varrho_1\invers\alpha_1\invers
   \alpha_2\varrho_2 = 
   \frac 12\left(\alpha_2b_2 - \alpha_1\frac{b_1}{a_1-1}\alpha_1\invers
   \alpha_2(a_2-1)\right)\,.
 \]
 On the other hand, using
 $N\beta_2(a_2-1)-\beta_2b_2= \alpha_2(a_2-1)$ and
 $N_1 =-T_1NT_1\invers$ we see
 \[
   \frac 12(N_1\alpha(a_2-1)+\alpha b_2) = \frac 12\left(T_1(1- NT_1\invers)
   \alpha_2(a_2-1) + \alpha_2b_2\right)\
\]
and as before
\[
  T_1(1- NT_1\invers)
  = \alpha_1 \frac{b_1}{1-a_1}\alpha_1\invers\,.
\]
Combining the previous equations, we obtain the claim since
\begin{align*}
  \beta & =
   \frac 12\left(\alpha_2b_2 - \alpha_1\frac{b_1}{a_1-1}\alpha_1\invers
   \alpha_2(a_2-1)\right) = \frac 12\left(T_1(1- NT_1\invers)
   \alpha_2(a_2-1) + \alpha_2b_2\right)\\ & = \frac
 12(N_1\alpha(a_2-1)+\alpha b_2)\,.
\end{align*}
\end{proof}

\begin{figure}[H]
  \includegraphics[height=5cm]{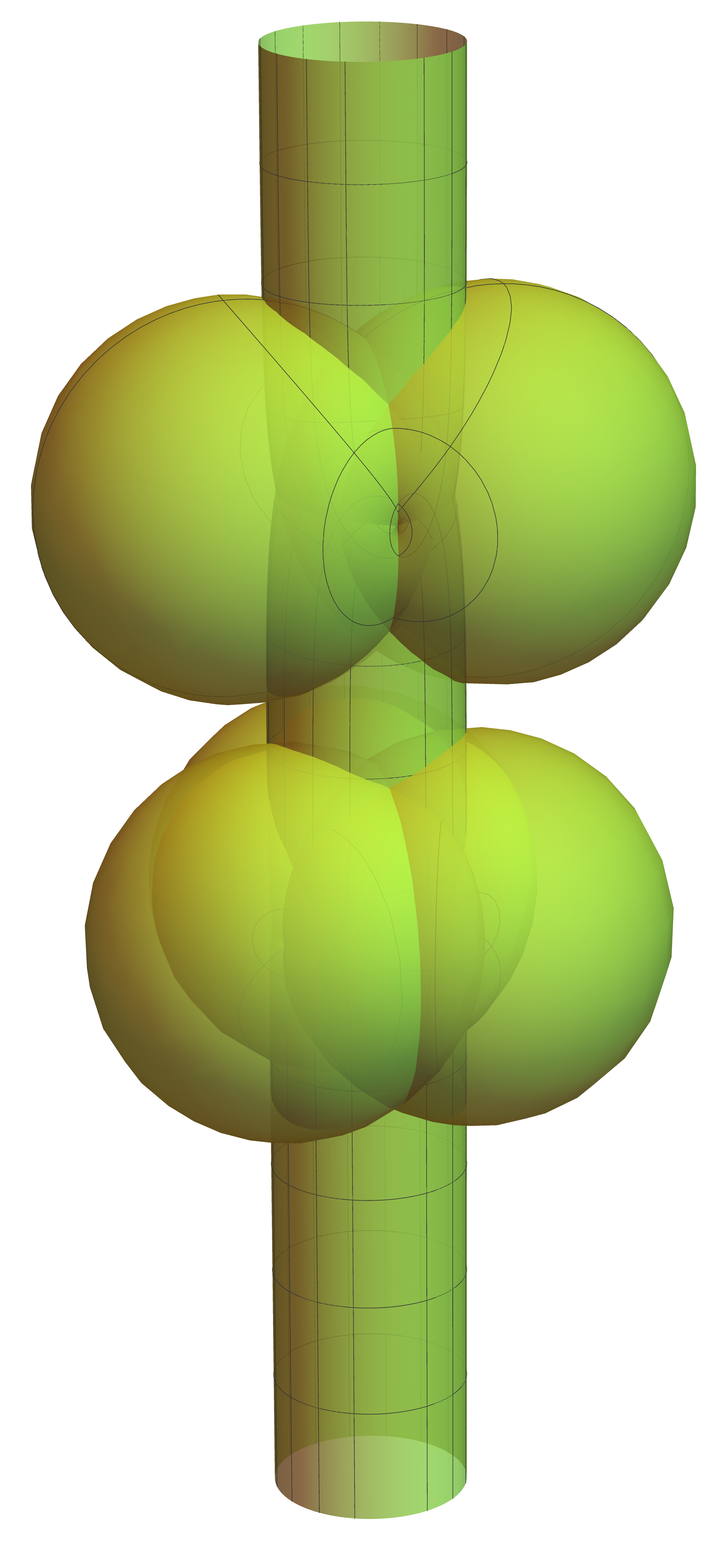}\qquad
  \includegraphics[height=5cm]{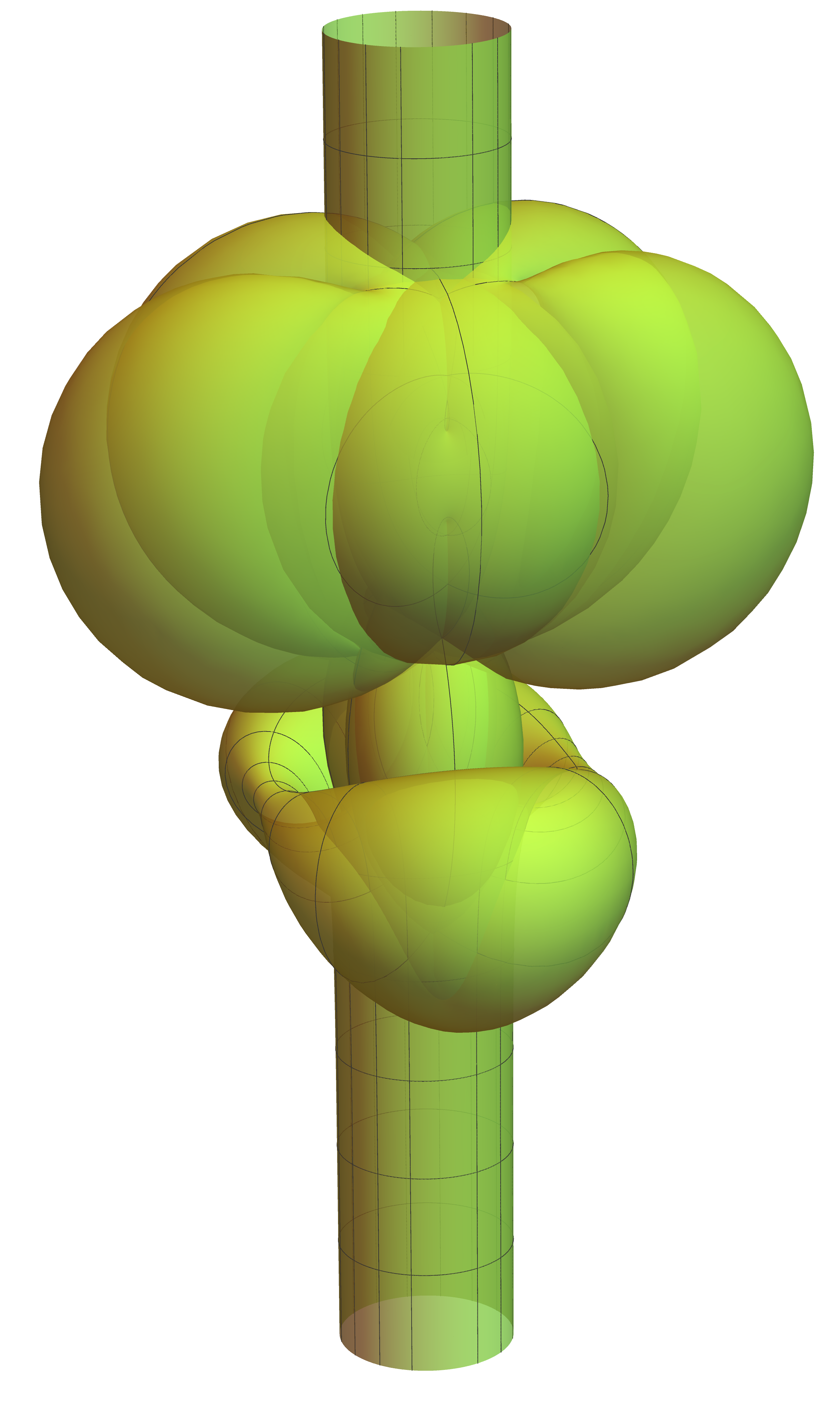}
  \caption{A CMC surface obtained by Bianchi permutability as common
    Darboux transform of two CMC bubbletons, which are $\mu$--Darboux
    transforms of the cylinder at the resonance points
    $\mu_1= 7-4 \sqrt{3}$ and $\mu_2 = 17-12 \sqrt{2}$ respectively,
    and an isothermic surface obtained by Bianchi permutability as a
    common Darboux transform of a CMC and an isothermic bubbleton,
    which are $\varrho$--Darboux transforms at the corresponding
    isothermic resonance points $\varrho_1=-3$ and $\varrho_2=-8$. }
\end{figure}

\subsection{Simple factor dressing of CMC surfaces}

In \cite{simple_factor_dressing} the simple factor dressing for CMC
surface was defined by its family of connections
$\check d_\lambda^N = r_\lambda\cdot d_\lambda^N$ where the gauge
matrix is given by
\[
  r_\lambda = r_\lambda^E = \pi_E \gamma_\lambda + \pi_{Ej}
\]
and $E$ is a $d_\mu$--stable bundle, $\pi_E, \pi_{Ej}$ the projections
along the splitting $\underline{\tilde \C}^2 = E\oplus Ej$ and
$\gamma_\lambda=
\frac{1-\bar\mu\invers}{1-\mu}\frac{\lambda-\mu}{\lambda
  -\bar\mu\invers}$.  Moreover, it was shown that if $d_\mu^N\alpha=0$
and $E=\alpha\C$, then the simple factor dressing
$r^N_{E,\mu}(f) = \check f$ of a CMC surface is a CMC surface in
3--space with harmonic Gauss map
\[
  \check N= TNT\invers
\]
where $T=\alpha \beta\invers$,
$\beta = \frac 12(N\alpha(a-1)+\alpha b)$ and
$a=\frac{\mu+\mu\invers}2, b= i\frac{\mu\invers-\mu}2$.

\begin{theorem}
  \label{thm:cmc simple factor}
  The simple factor dressing $\check f = r^N_{E,\mu}(f) $ of a CMC
  surface $f$ given by $d_\mu^N$--parallel $\alpha$, $E =\alpha\C$, is
  a 2--step Darboux transform of $f$ up to possible translation.  More
  precisely, the simple factor dressing is given by
  \[
    \check f = f - \alpha \frac{b}{a-1} \alpha\invers
  \]
  and $\check f$ is a common Darboux transform of the $\mu$--Darboux
  transform $ f^\mu = f + \alpha\beta\invers$ of $f$ given by $\alpha$
  and the parallel CMC surface $g=f+N$ which is a $\mu$--Darboux
  transform at $\mu_0=-1$.
 \end{theorem}
 \begin{proof} Recall that by Remark \ref{rem: parallel is Darboux}
   any non--trivial parallel section $\alpha_0$ of $d_{\mu_0}^N$ gives
   with $\beta_0 =-N\alpha_0$ the $\mu$--Darboux transform
   $f^{\mu_0} = f+ N =g$.  By Bianchi permutability Theorem \ref{thm:
     gen bianchi} and Theorem \ref{thm: bianchi cmc} the common
   $\varrho$--Darboux transform of $\hat f$ and $g$ is the CMC surface
  \[
   \hat g =  g +
    \tilde \alpha\tilde\beta\invers
  \]
  where
  \[\tilde \alpha = \alpha - \alpha_0\beta_0\invers \beta = \alpha - N
    \beta \,, \qquad
    \tilde \beta = \beta - \beta_0  \varrho_0\invers\alpha_0
    \invers\alpha\frac{1-a}2 = \beta + N\alpha \frac{1-a}2
  \]
  and $\varrho_0 = \frac{1-a_0}2= 1$.  
  Using $\beta = \frac 12(N\alpha(a-1)+ \alpha b)$ and $a^2+ b^2=1$, we obtain
  \[
    \tilde \alpha = \frac 12(\alpha(a+1)- N\alpha b)\,, \qquad \tilde
    \beta = \frac 12 \alpha b
  \]
  so that, again with $a^2+b^2=1$, 
  \[
  \hat g = f - \alpha \frac{b}{a-1} \alpha\invers\,.
  \]
  To show that $\check f$ and the 2--step Darboux transform $\hat g$
  coincide up to translation it is enough to show that $\hat g$ has
  the same Gauss map $\check N = TNT\invers$ as $\check f$.  We
  compute
  \[
    d\hat g = df + [df\beta\alpha\invers, \alpha\frac{b}{a-1}
    \alpha\invers] = T\invers
    \hat \varrho\invers df
    T\invers
  \]
  where we used
  $T\invers \hat \varrho\invers = - N - \alpha \frac{b}{a-1}
  \alpha\invers$.  This shows that the right normal of $\hat g$ is
  $\hat R = T N T\invers$, and since the $\mu$--Darboux transform
  $\hat g$ of $g$ is a CMC surface in 3--space up to translation,
  i.e., $\hat R = \hat N$, this shows the claim.
   
\end{proof}

\section{Constrained Willmore surfaces}
\label{sec:cw}

In \cite{bohle, willmore_harmonic} a family of flat connections
$d_\lambda^S$ was defined for constrained Willmore surfaces, that is,
critical points of the Willmore energy with preserved conformal
structure. In particular, parallel sections give again $\mu$--Darboux
transforms which are constrained Willmore. On the other hand,
\cite{burstall_quintino} defined a simple factor dressing for
constrained Willmore surfaces which is given by the choice of a
spectral parameter and a complex 2--dimensional bundle which is
$d_\mu^S$--stable. In general, it is unknown whether these transforms
are linked. Here we will show that in case of a CMC surface, a special
choice of $d_\mu^S$--stable bundle corresponds to the simple factor
dressing of the CMC surface.

\subsection{The associated family of a constrained Willmore surface}
Recall \cite[Section 7.2]{coimbra} that for a conformal immersion
$f: M \to\R^3$ with constant mean curvature $H =1$ the conformal Gauss
map of $f$ is the complex structure
\[
S  =F\begin{pmatrix} N &0\\ 1 &-N 
\end{pmatrix}
F\invers, \quad F= \begin{pmatrix} 1 &f \\ 0&1
\end{pmatrix}
\]
with Hopf field
\[
*A=\frac 14(S*dS- dS) = \frac 14F \begin{pmatrix} 0 &0\\ *dg& 2dg
\end{pmatrix}F\invers\,,
\]
where $g= f+ N$ and thus $dN+ N*dN = 2 (dN)'' = 2dg$. 
Putting
\[
\eta= -\frac 14 F\begin{pmatrix} 0 &0\\ dg &0
\end{pmatrix}F\invers\,,
\]
we have $*\eta = S\eta = \eta S$ and $\eta$ takes values in the line
bundle $L$ of $f$.  In particular, $f$ is constrained Willmore,
\cite{constrainedWillmore}: the 1--form $\hat A = A + \eta$ satisfies
$d*\hat A = 0$ since
\[
*\hat A = \frac 12 \begin{pmatrix} 0 & f\\ 0&1
\end{pmatrix}
dg
\]
and $df\wedge dg=0$.  As before, this allows to introduce a spectral
parameter and a $\C_*$--family of flat connections on $\C^4$ where the
complex endomorphism $I$, which acts by right multiplication by $i$,
gives the complex structure on $(\H^2, I)= \C^4$:

\begin{prop}[see \cite{willmore_harmonic, bohle}]
  Let $f: M \to\R^3$ be CMC with Gauss map $N$ and parallel CMC
  surface $g= f+N$, and let $N_g=-N$ be the Gauss map of $g$.  Then
  the family
\[
d_\lambda^{S} = d + (\lambda-1)\hat A\oz + (\lambda\invers-1)\hat A\zo = d+ \frac 12\begin{pmatrix} 0 & f \\ 0 &1
\end{pmatrix} dg(N_g(a-1)+ b)
\]
is flat for all $\lambda\in\C_*$.
\end{prop}

In the case when $\mu=1$ or $dg=0$ we have $d_\mu^S = d$ and all constant
sections are parallel.

In the case when $\mu\not=1$ or $f$ is not a round sphere, i.e., $dg$
is not vanishing identically, we see from the explicit expression of
$d^{S}_\mu$ that a section
$\varphi\in\Gamma( \underline{\tilde\C}^4)$,
$\varphi = e\nu+\psi \beta$, is $d^{S}_\mu$--parallel if and only if
\begin{equation}
\label{eq: ds paralle}
d\nu = -df \beta\,, \quad d\beta = -\frac
12dg(N_g\beta(a-1)+\beta b)\,.  
\end{equation}

We now see that all parallel sections of $d_\mu^S$ are given
algebraically by parallel sections of the associated family $d_\mu^N$
of the harmonic Gauss map $N$ of $f$:

\begin{theorem}
\label{thm: CW DT}
Let $f: M \to\R^3$ be CMC with Gauss map $N$ and not the round
sphere. Denote by $d^{S}_\lambda$ the family of flat connections given
by the conformal Gauss map $S$ of $f$. Then a section $\varphi $ is
$d_\mu^S$--parallel, $\mu\in\C\setminus\{0, 1\}$, if and only if
\[
  \varphi = e(\alpha + n) + \psi \beta
\]
with $n\in\H$, $d_\mu^N$--parallel $\alpha$ and
$\beta = \frac 12(N\alpha (a-1) + \alpha b)$ where
$a=\frac{\mu+\mu\invers}2$ and $ b= i\frac{\mu\invers-\mu}2$.

Moreover, all $d_\mu^N$--parallel sections are obtained algebraically
from $d^S_\mu$--parallel sections $\varphi = e\nu + \psi \alpha$ by
$\alpha = N\beta - \beta \frac{b}{a-1}$.
\end{theorem}

\begin{proof}
  Let first $\alpha$ be $d_\mu^N$--parallel, then
  $d\alpha = -df \beta$ with
  $\beta = \frac 12(N\alpha (a-1) + \alpha b)$.  By Theorem \ref{thm:
    both parallel} we have
\[
  d\beta = -dg(-N\beta(a-1) + \beta b)\,.
\]
This shows by (\ref{eq: ds paralle}) that
$\varphi = e\alpha + \psi \beta$ is $d_\mu^S$--parallel. Moreover, any
constant section $en$ is clearly $d_\mu^S$--parallel.

Conversely, assume that $\varphi = e \nu + \psi \beta$ is
$d_\mu^S$--parallel. Putting $\alpha = N\beta - \beta \frac{b}{a-1}$
we see with $a^2+b^2=1$ and (\ref{eq: ds paralle}) that
\[
  d \alpha  = -df \beta\,,
\]
that is, $\nu = \alpha + n$ with constant $n\in\H$. Since
\[
  \frac 12(N\alpha(a-1)+\alpha b) = \beta
\]
we see that $\alpha$ is $d_\mu^N$--parallel, and compairing the
complex dimensions of the corresponding space of parallel sections, we
obtain the claim.
\end{proof}

We note that the parallel section of the CMC integrable system which
appears naturally here by (\ref{eq: ds paralle}) is the
$d_\mu^{-N}$--parallel section $\beta$: the $d_\mu^N$--parallel
section is then given by the correspondence of parallel sections of
$f$ and its parallel CMC surface $g=f+N$. This observation might be
relevant to understand the link between the $\mu$--Darboux
transformation and simple factor dressing in the case of general
constrained Wilmore surfaces.

In \cite{burstall_quintino} the associated family of constrained
Willmore surfaces with spectral parameter $\mu\in S^1$ has been
discussed. In the case when $f: M \to\R^3$ is a CMC surface, for
$\mu\in S^1$ the $d^S_\mu$--parallel sections
\[
  e, \varphi = e\alpha + \psi \beta\,,
\]
where $\alpha$ is $d_\mu$--parallel and
$\beta = \frac 12(N\alpha (a-1) + \alpha b)$, form a quaternionic
basis of $d_\mu^S$--parallel sections since in this case $a, b\in\R$
and $d_\mu^S$ is quaternionic.  By \cite{burstall_quintino} the
associated family $A^S_{\Phi, \mu}(f)$ of constrained Willmore
surfaces is given, up to M\"obius transforms, by its associated line
bundle
\[
  L^\Phi =\Phi \invers L
\]
where $\Phi = (e, \varphi)\in \Gamma(\End_\H(\ttrivial 2))$.
Therefore, as in the case of isothermic surfaces, we obtain for
$\mu\not=0, 1$, with
\[
  M = \begin{pmatrix} 1 & \alpha \\ 0 & \beta
  \end{pmatrix}
\]
and
\[
  L^\Phi = M\invers \begin{pmatrix} 0 \\ 1
  \end{pmatrix}
  \H = \begin{pmatrix} - \alpha \beta \invers \\ \beta \invers 
  \end{pmatrix} \H
\]
the affine coordinate of the associated family by
\[
  A^S_{\Phi, \mu}(f)  = -\alpha\,,
\]
where $\alpha$ is $d_\mu^N$--parallel and $\mu\in S^1$.

 \begin{figure}[H]
  \includegraphics[height=4cm]{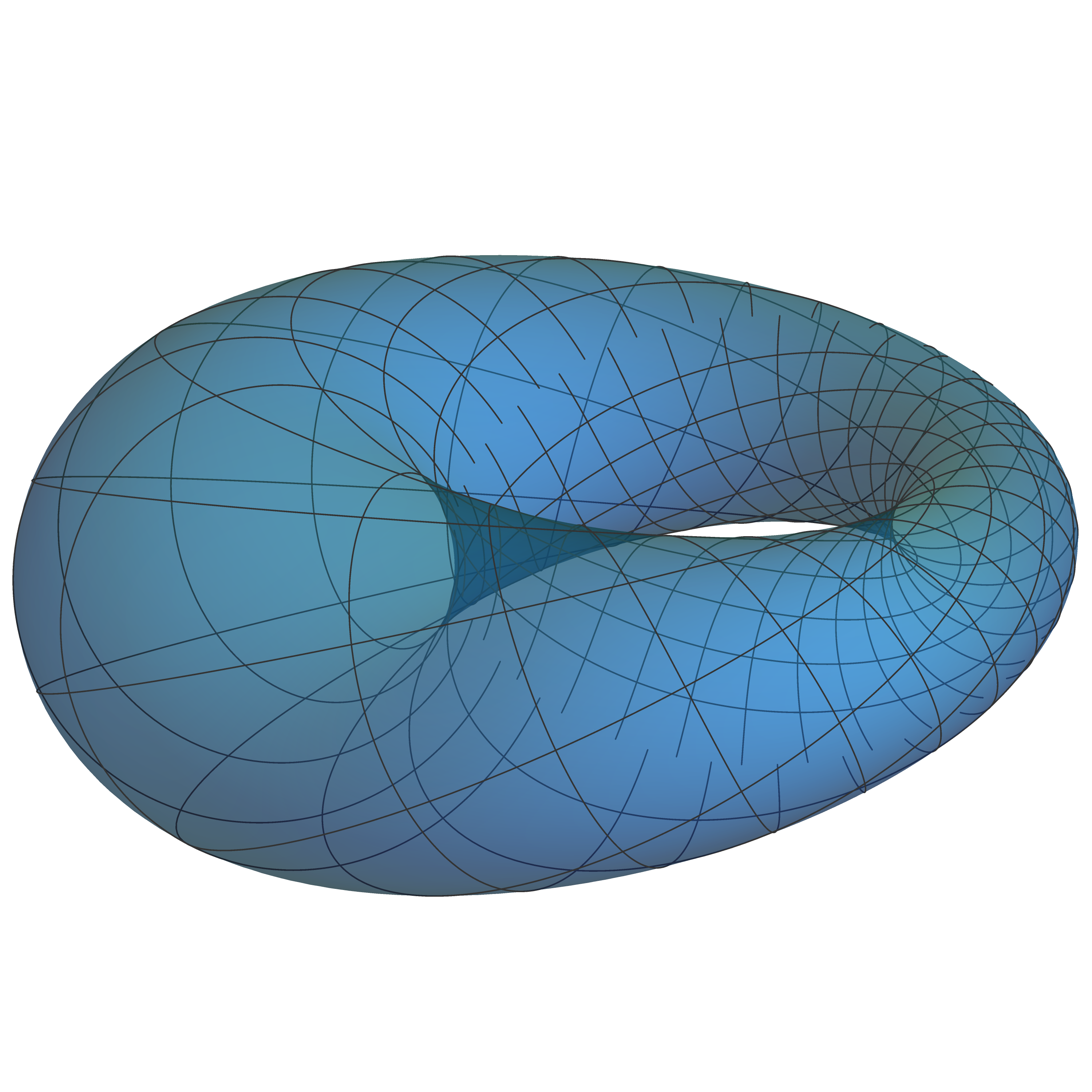}
\includegraphics[height=4cm]{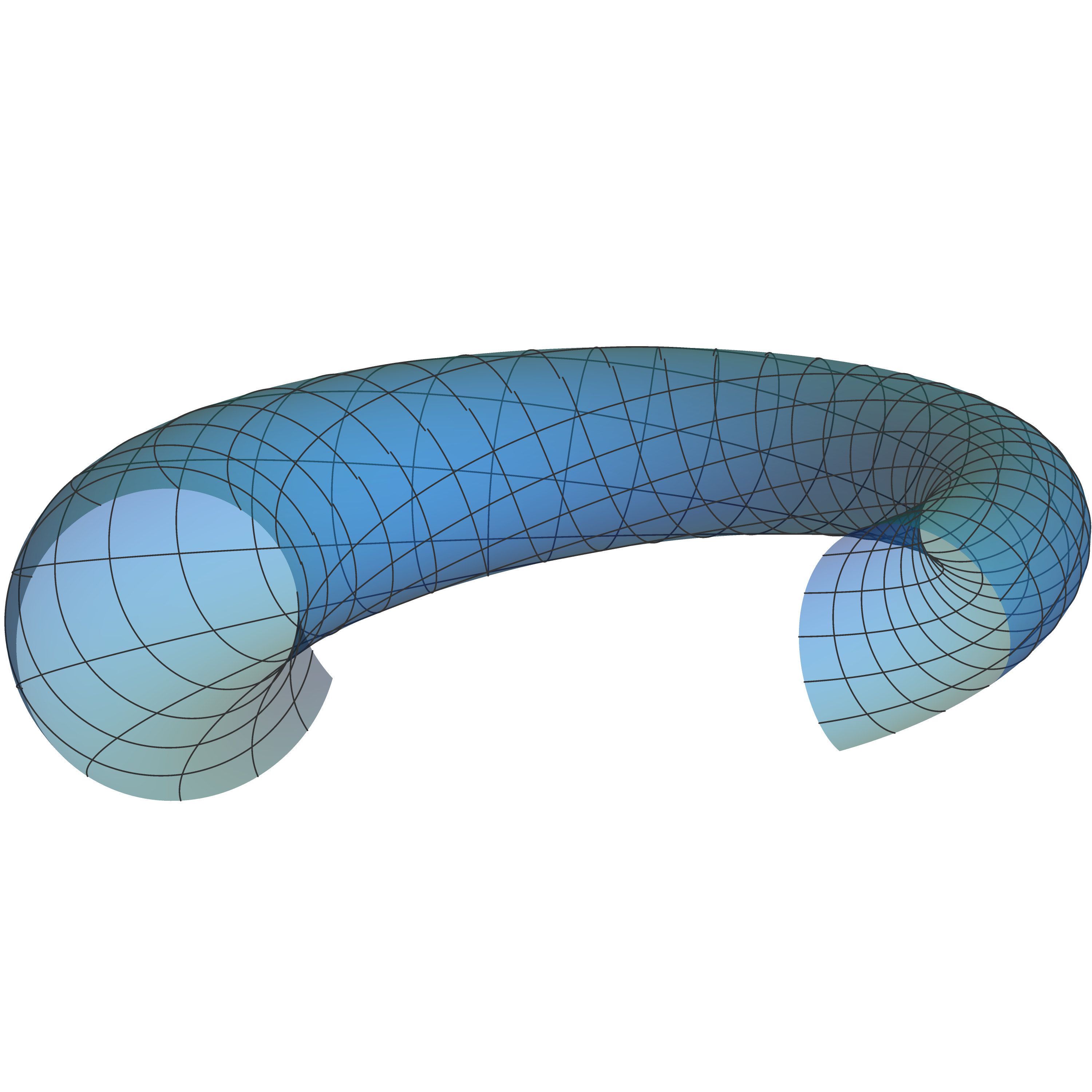}
\includegraphics[height=4cm]{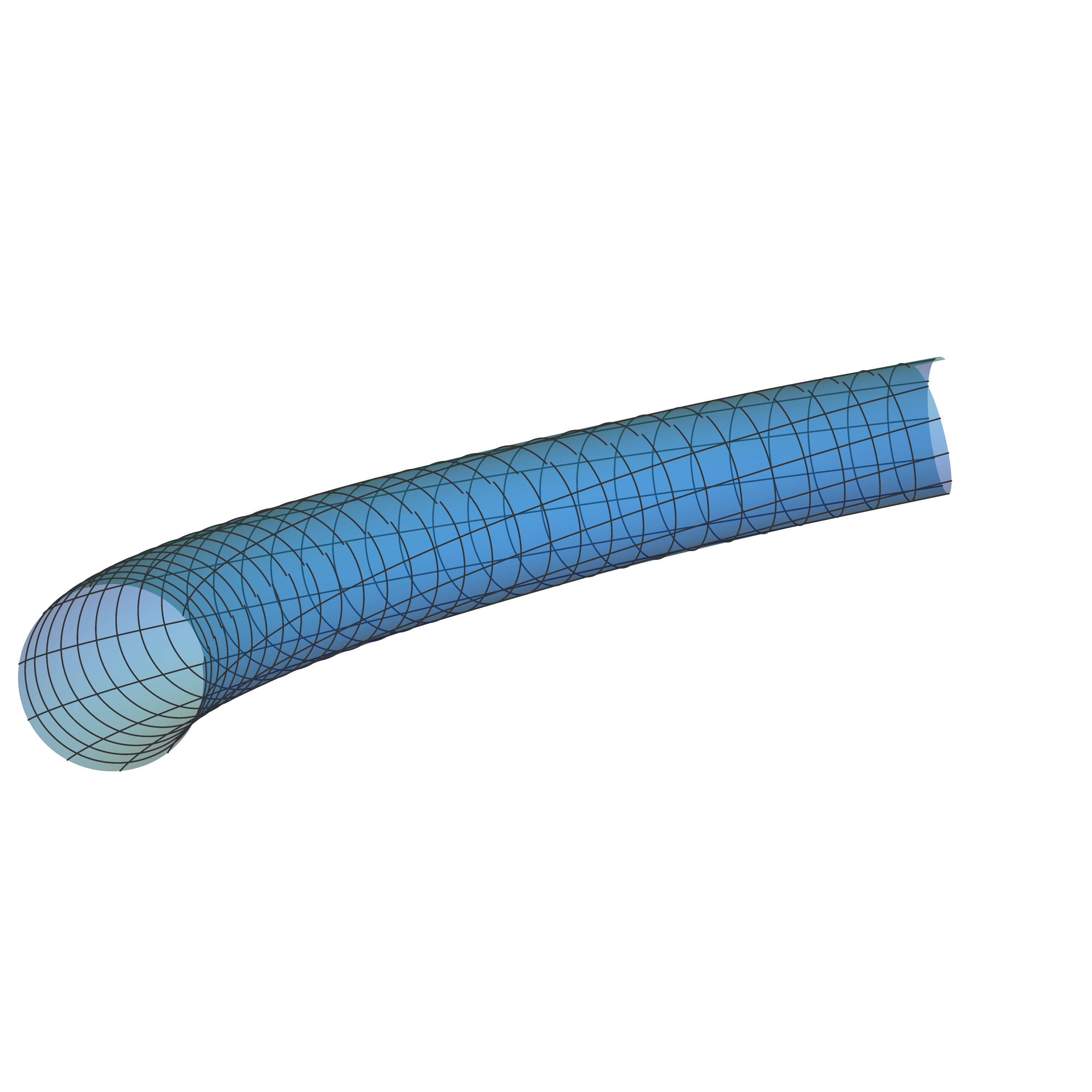}
\caption{Constrained Willmore surfaces in $S^3$ in the associated family of the
  cylinder. The parametrisation is only periodic in the
  $x$--direction, with the period depending on the spectral parameter
  $\mu=e^{it}\in S^1$, $t= \frac\pi 2, \frac\pi 3,\frac \pi {10}$.
}
\end{figure}

Since $d\alpha = -\frac 12df(N(a-1)+b)\alpha$ we see that the left
normal $N^\Phi=N$ of $A^S_{\Phi,\mu}(f)$ is the Gauss map $N$ of $f$
and the right normal is $R^\Phi= \alpha\invers N\alpha$:

\begin{theorem} The constrained Willmore associated family
  $f^\Phi=A^S_{\Phi, \mu}(f)$ of a CMC surface $f: M \to \R^3$ with
  $\mu\in S^1$ and Gauss map $N$ is, up to M\"obius transform, a
  family of CMC surfaces in $S^3(\mathfrak r)$ with mean curvature
  \[
    H^\Phi_{S^3(\mathfrak r)} = \frac 1{\mathfrak r} \left(\frac{b}{a-1}\right)
  \]
  where $\mathfrak r = |\alpha|$ and $a= \frac{\mu+\mu\invers}2, b =
  i\frac{\mu\invers-\mu}2$ for $\mu\in S^1$.
\end{theorem}
\begin{proof} As in the isothermic case we have
  $N^\Phi f^\Phi = f^\Phi R^\Phi$ so that $f^\Phi$ takes values in the
  3--sphere of radius $\mathfrak r = |\alpha|$.  Now
  $df^\Phi H^\Phi =-(dN^\Phi)'= -(dN)' = df$ gives
  \[
    H^\Phi = \alpha\invers(N + \frac{b}{1-a})
  \]
  and thus
  \[
      H^\Phi_{S^3(\mathfrak r)} = \frac 1{\mathfrak r}  \Re(f^\Phi
      H^\Phi) = \frac 1{\mathfrak r} \left(\frac{b}{a-1}\right)\,.
    \]
  \end{proof}

Similar to the isothermic case, we can link the constrained Willmore
associated family to the CMC associated family:
\begin{theorem}
  \label{thm: limit2}
  Let $f: M \to \R^3$ be a CMC surface and $d_\lambda^S$ its
  associated family of flat connections. Denote by
  $f^\alpha = A^N_{\alpha, s}(f)$ the CMC surface in the associated
  family given by a $d_\mu^N$--parallel section $\alpha$ with
  $\mu=e^{is}\in S^1$ and by $ A^{S^\alpha}_{\Phi, \lambda}(f^\alpha)$
  the associated family of constrained Willmore surfaces of $f^\alpha$
  for $\lambda\in S^1$.  Then up to M\"obius transformation
\[
  \lim_{\lambda\to 1}  A^{S^\alpha}_{\Phi, \lambda}(f^\alpha) = A^N_{\alpha, s}(f)\,.
  \]
\end{theorem}
\begin{figure}[H]
    \includegraphics[width=2cm]{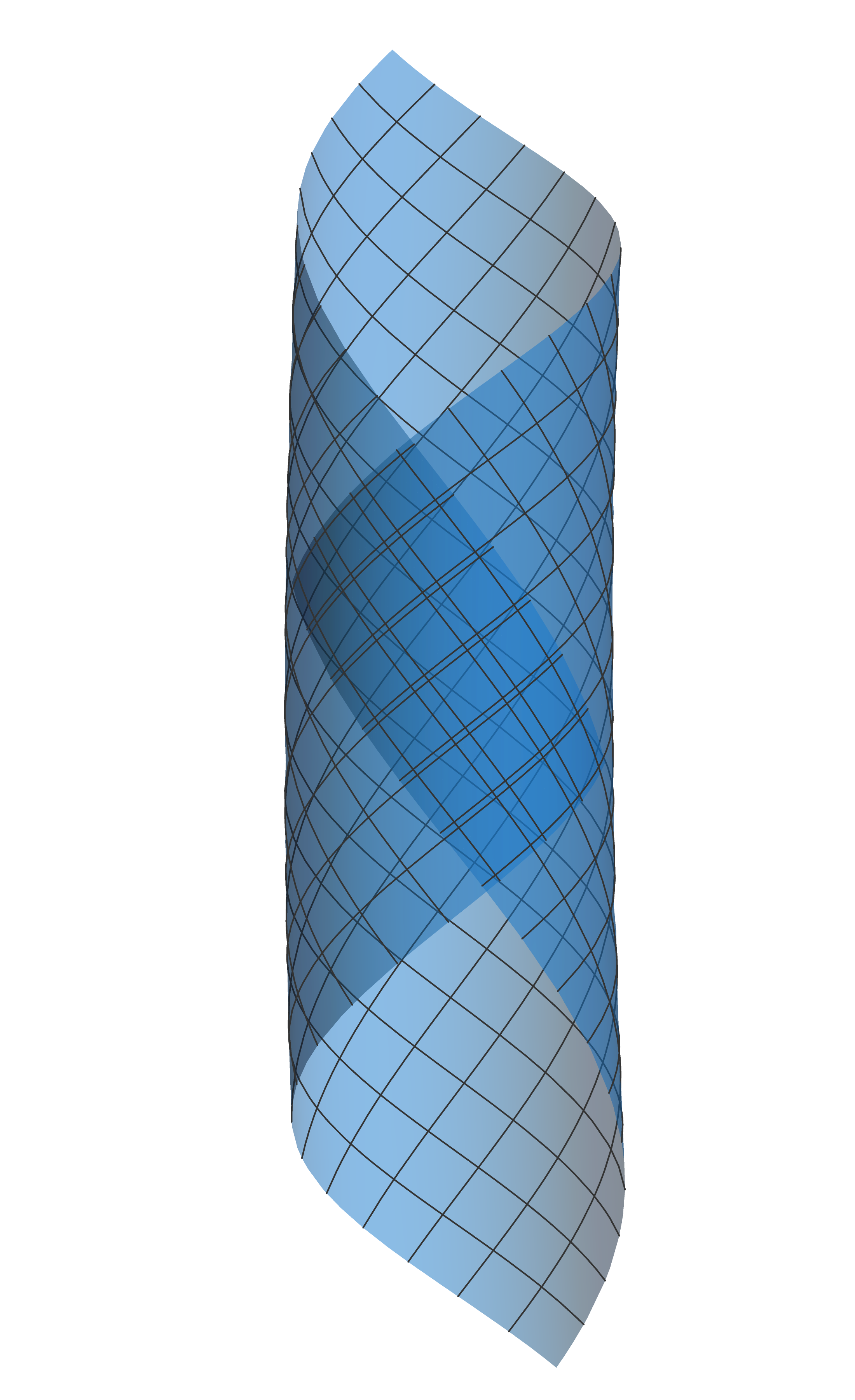}
    \includegraphics[width=2cm]{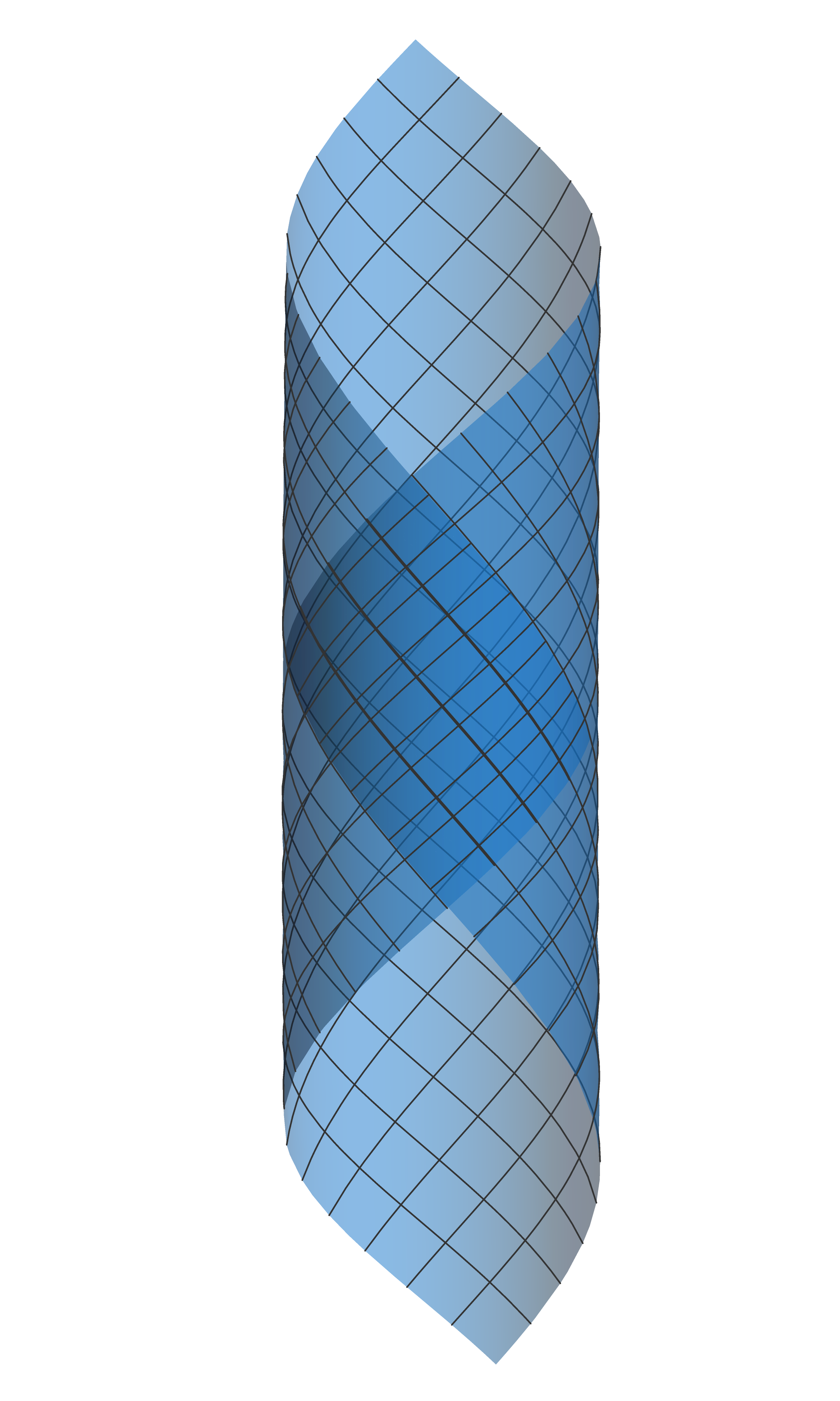}
    \includegraphics[width=2cm]{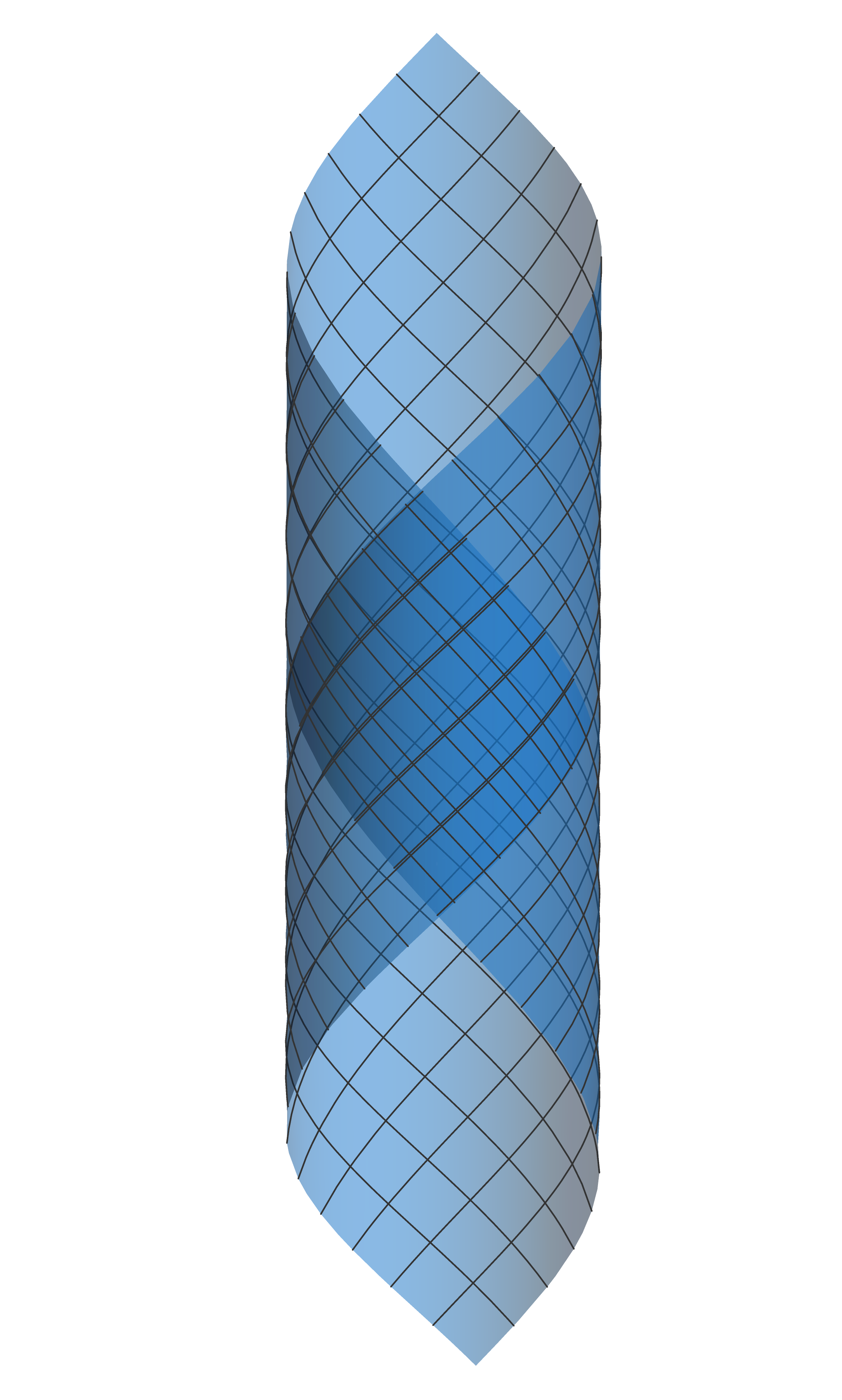}
    \includegraphics[width=2cm]{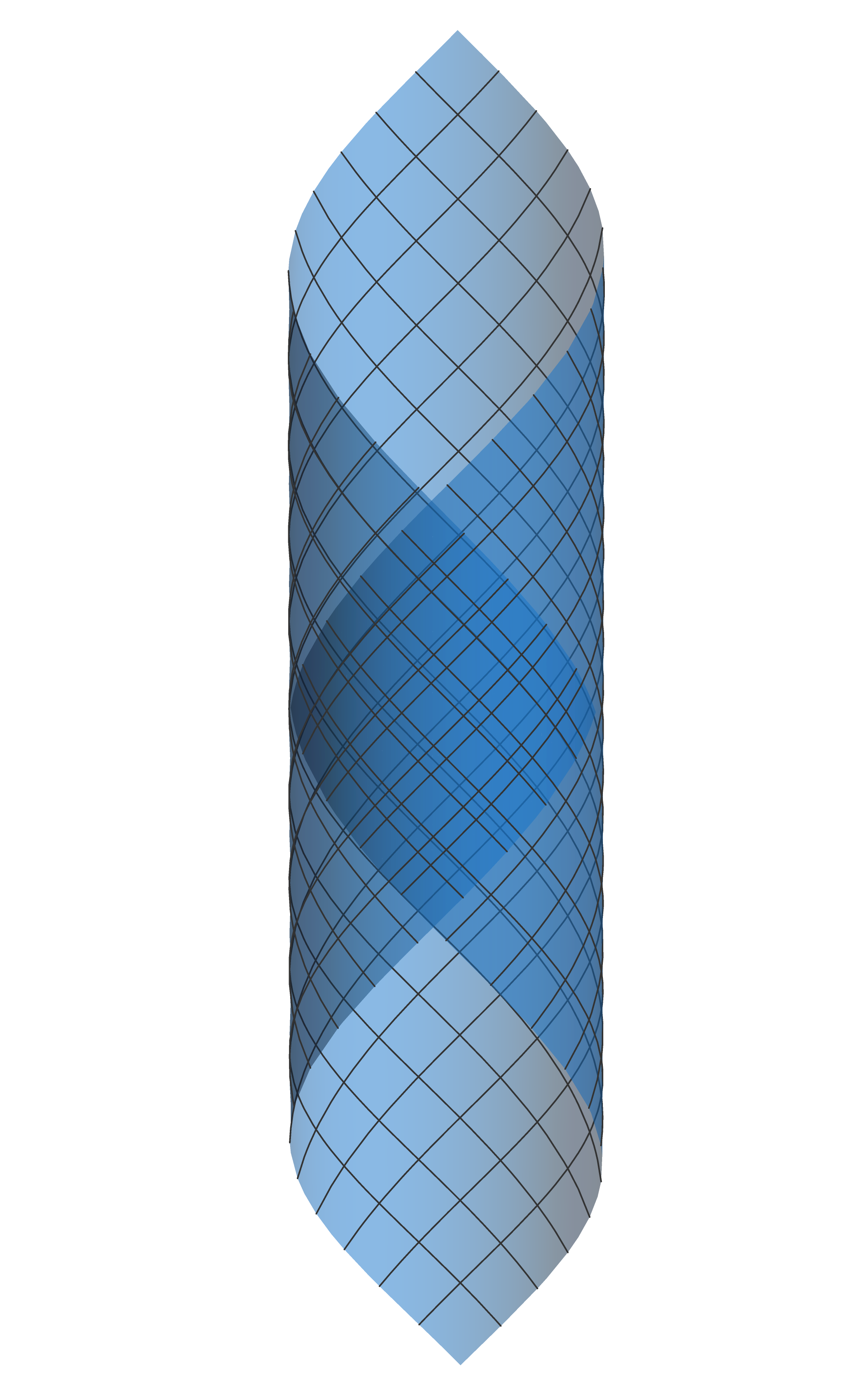}
      \includegraphics[width=2cm]{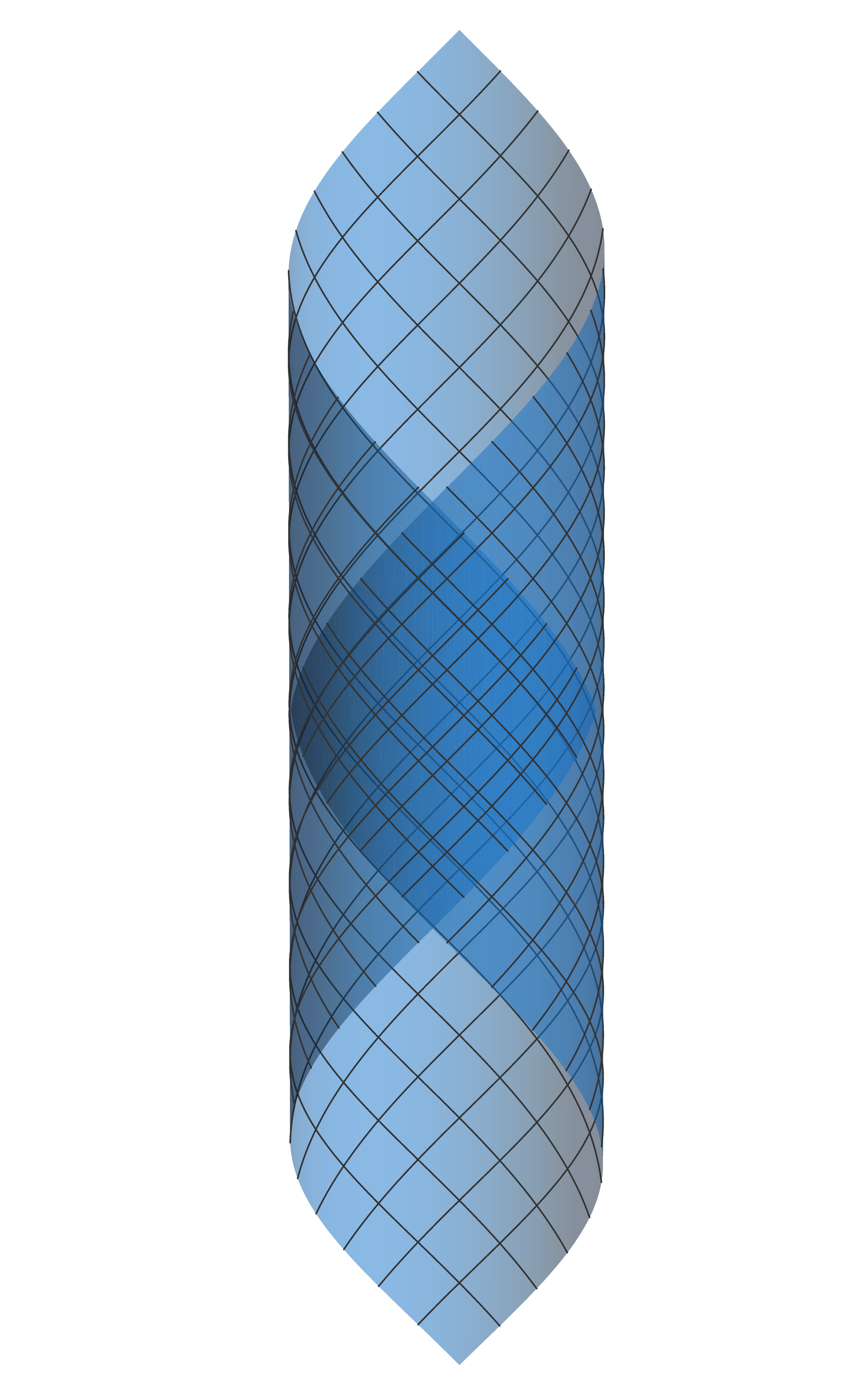}
\caption{Members of the constrained Willmore associated family $A_{\Phi,
    \lambda}^{S^\alpha}(f^\alpha)$ of the CMC surface $f^\alpha=A^N_{\alpha, \frac\pi
    2}(f)$ in
  the 
 associate family of the cylinder   and their limit $f^\alpha$ for  $\alpha=\alpha_+(\frac \pi 2) $ and
  $\lambda=e^{it}$,  $t= \frac 12,
  \frac 14, \frac 1{8}, \frac 1{100}$, orthogonally projected into
  3--space. }\end{figure}

\begin{proof}
  Let $\mu=e^{is}\in S^1$ and $\alpha$ a non--trivial
  $d_\mu^N$--parallel section, and consider the associated family
  $d_\lambda^{N^\alpha}$ of the harmonic Gauss map
  $N^\alpha=\alpha\invers N \alpha$ of the CMC surface
  $f^\alpha = A^N_{\alpha, \mu}(f) $ given by $\alpha$ via the
  Sym--Bobenko formula, see Theorem \ref{thm:sym-bob}, as
  $f^\alpha =-2 \alpha\invers \frac{d}{dt} \alpha|_{t=s}$.

  Let $\alpha_{\lambda\mu}$ be $d_{\lambda\mu}^N$--parallel, smooth in
  $\lambda$, with $\alpha_{\lambda\mu = \mu}=\alpha$. By
  (\ref{eq:assoasso}) we know that
   \[
     \alpha^{N_\alpha}_\lambda=\alpha\invers
     \alpha_{\lambda\mu}
   \]
   is a parallel section of $d_\lambda^{N^\alpha}$.
   
   Putting
   $\beta^{N^\alpha}_\lambda =\frac
   12(N^\alpha\alpha^{N^\alpha}_\lambda(a-1) +
   \alpha^{N^\alpha}_\lambda b)$ and
   $\psi^\alpha= \begin{pmatrix} f^\alpha \\ 1
  \end{pmatrix}
  $ we see that
  $$\varphi_\lambda= e\alpha^{N^\alpha}_\lambda +
  \psi^\alpha\beta^{N^\alpha}_\lambda$$ is
  $d_\lambda^{S^\alpha}$--parallel, where $S^\alpha$ is the conformal
  Gauss map of the CMC surface $f^\alpha$. In particular, up to
  M\"obius transformation, the associated family
  $f^{\Phi}= A^{S^\alpha}_{\Phi, \lambda}(f^\alpha)$ of the
  constrained Willmore surface $f^\alpha$ is given by
  $\Phi\invers L^\alpha$ where $L^\alpha=\psi^\alpha \H$ is the line
  bundle of $f^\alpha$ and
 \[
   \Phi = (e,\frac 2t(\varphi_\lambda - e))\,,
 \]
 with $\lambda=e^{i t}$.  Then as before we can compute the affine
 coordinate to be
 \[ f^{\Phi} = \frac 2t(1-\alpha^{N^\alpha}_{e^{it}})= \frac 2 t
   (1-\alpha\invers \alpha_{e^{i(t+s)}})\,.
  \]
  Taking the limit for $\lambda\mu$ to $\mu$, that is, $t$ to $0$, by
  the Sym--Bobenko formula \ref{thm:sym-bob} we obtain the result.
\end{proof}

\subsection{$\mu$--Darboux transforms of a constrained Willmore surface}

In \cite{bohle} it was shown that Darboux transforms of constrained
Willmore surfaces are constrained Willmore. We will here use the
explicit parallel sections given by the harmonic Gauss map of a CMC
surface to give $\mu$--Darboux transforms explicitly.

For this, let $\varphi = e\nu + \psi \beta$ be
$d_\mu^S$--parallel. Then $d\nu = -df \beta$ shows as before that
$e\nu$ is a holomorphic section, and $\varphi$ is its
prolongation. Thus, the surface $D^S_{\varphi, \mu}(f)$ given by
$\varphi\H$ is a Darboux transform of $f$.

In the case when $\nu = \alpha= N\beta -\beta \frac{b}{a-1}$ we see that
$\varphi = e\alpha + \psi \beta$ is the $\mu$--Darboux transform given
by the $d_\mu^N$--parallel section $\alpha$.  In general, we have:

\begin{theorem}
  \label{thm:muDTsame}
  Let $f: M \to\R^3$ be a CMC surface and
  $D_{\varphi, \mu}^S(f) = f + \nu\beta\invers: \tilde M \to \H$ be
  the Darboux transform given by the $d_\lambda^S$--parallel section
  $\varphi =e \nu + \psi \beta$ where $d_\lambda^S$ is the associated
  family of flat connections given by the conformal Gauss map $S$ of
  $f$.

  Then $\hat f =D_{\varphi, \mu}^S(f)$ is constrained Willmore with
  harmonic right normal. Moreover, $\hat f$ is a CMC surface in
  3--space (up to translation) if and only if
  $\hat f = D^N_{\alpha, \mu}(f)$ is, up to translation, a
  $\mu$--Darboux transform of $f$.
  \end{theorem}
  \begin{proof}
       
    Since $\varphi = e\nu + \psi \beta$ is $d_\mu^S$--parallel we can
    write $\nu = \alpha + n$ with $n\in\H$ and
    $\alpha =N\beta-\beta\frac{b}{a-1}$ where
    $a= \frac{\mu+\mu\invers}2$, $b= i \frac{\mu\invers-\mu}2$ and
    $\beta$ is a $d_\mu^{-N}$--parallel section.  We compute
    \[
      d\hat f = df + d\nu\beta\invers - \nu\beta\invers d\beta
      \beta\invers = \nu\beta\invers dg \alpha \varrho \beta\invers
    \]
    with $\varrho = \frac{1-a}2$. By Corollary \ref{cor:parallel mu
      dt} the $\mu$--Darboux transform $g^\mu = g + T_g$,
    $T_g = \beta\varrho\invers \alpha\invers$, of the parallel CMC
    surface $g=f+N$ has Gauss map
    \[
      N^\mu_g = -T_g N_g T_g\invers= - T_gNT_g\invers\,.
    \]
    Hence the right normal
    \[
      \hat R = - T_g N T_g\invers
    \]
    of $\hat f$ is the negative of the Gauss map $N^\mu_g$ of $g^\mu$.
   
    Put differently, $\hat R$ is the Gauss map of the parallel CMC
    surface $f^\mu= f + \alpha\beta\invers$ of $g^\mu$. Therefore,
    $\hat R$ is harmonic, and $\hat f$ is constrained Willmore, see
    \cite{hsl, bohle}.

    For $\hat f$ to be CMC in 3--space (up to translation) the left
    and right normal of $\hat f$ have to coincide. In this case,
    $\hat f$ and $f^\mu$ have the same Gauss map, and $\hat f$ is a
    $\mu$--Darboux transform of $f$ up to translation.
  \end{proof}

  \begin{figure}[H]
    \includegraphics[height=5cm]{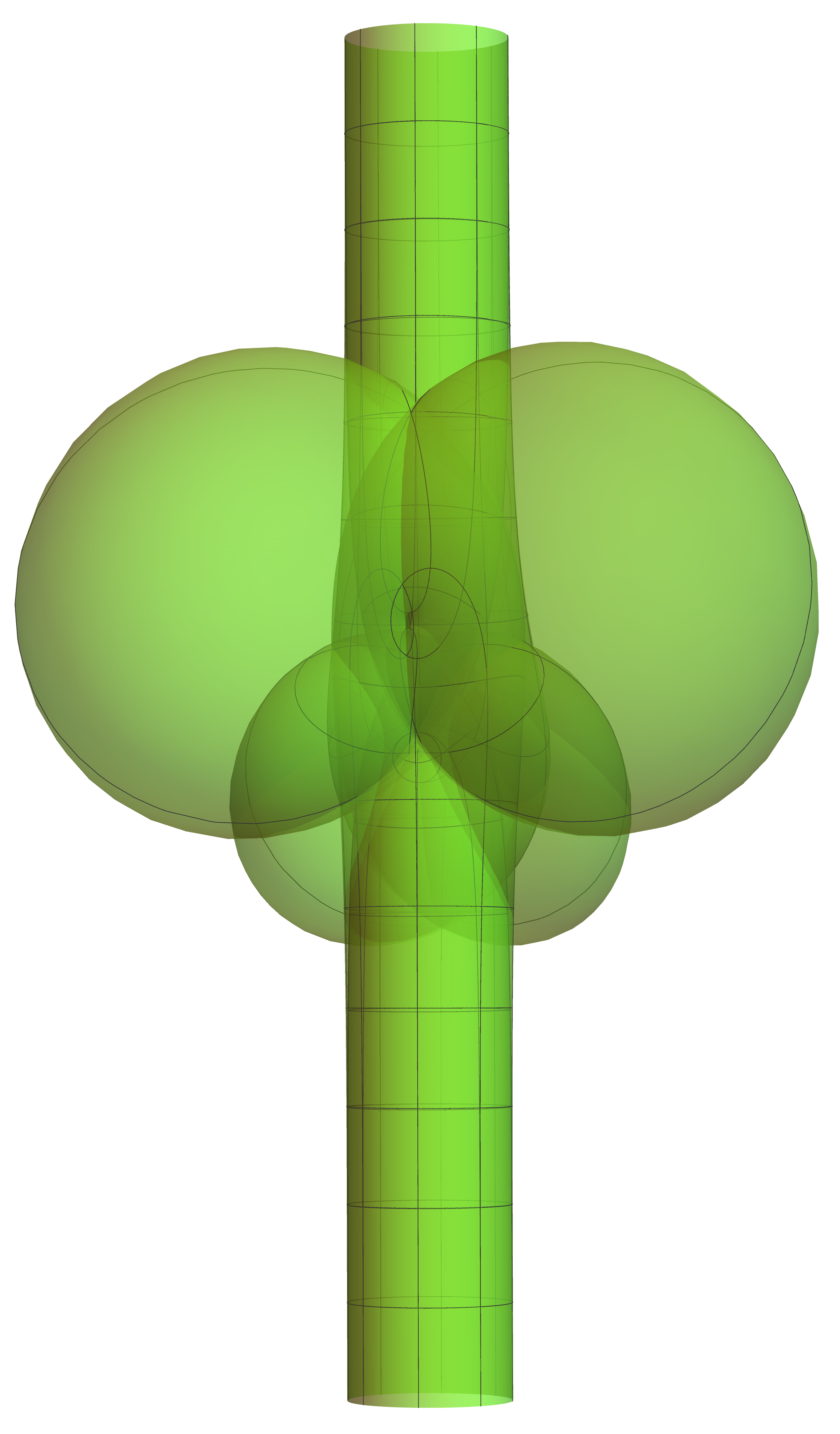} \qquad
     \includegraphics[height=5cm]{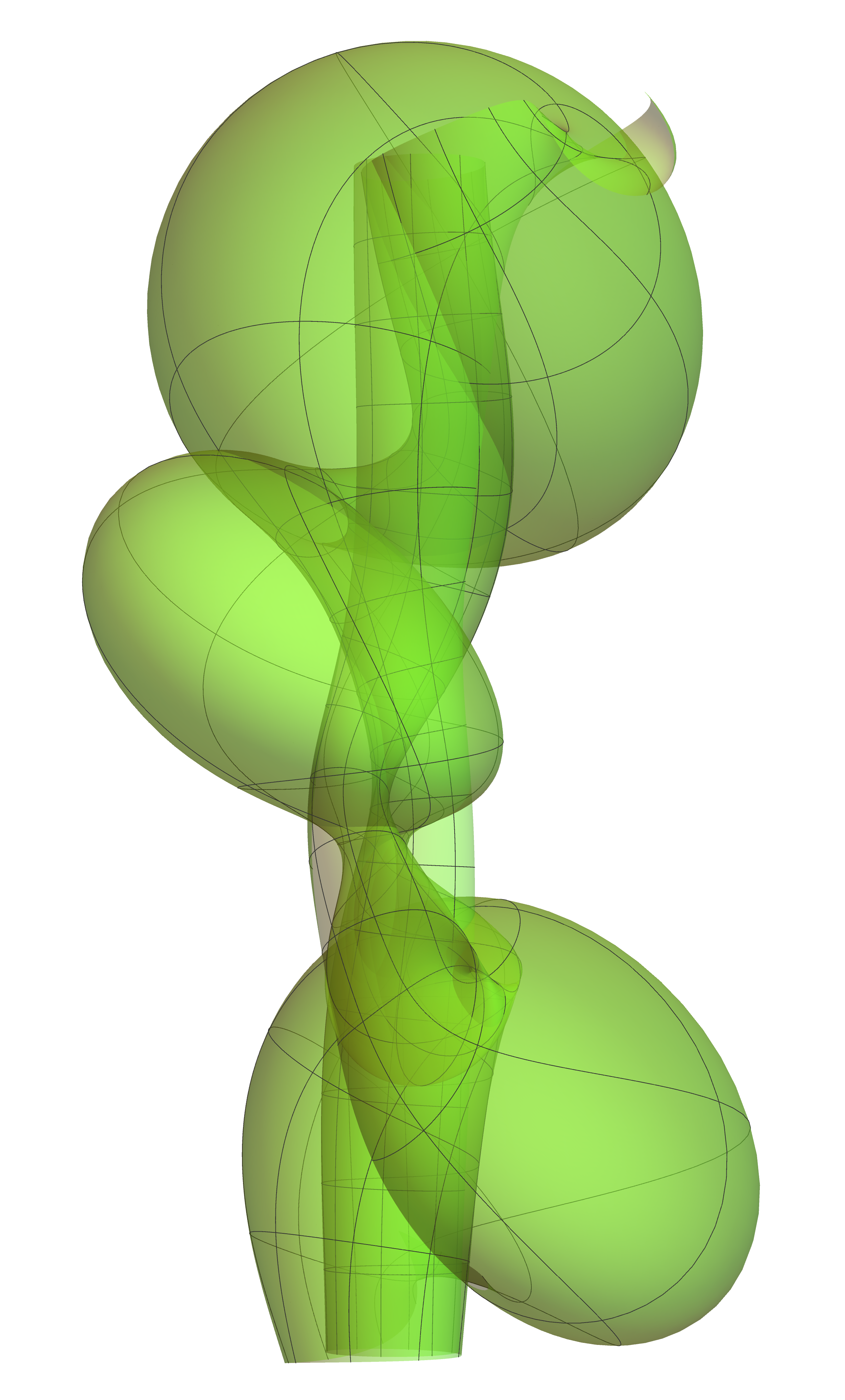}
     \caption{Constrained Willmore surfaces in $S^4$ obtained as a
       Darboux transform of the cylinder, periodic on the
       double--cover for $\mu =7-4 \sqrt{3}$ and $n=1- 4k$, and
       non--periodic for $\mu= (-1-2 i)+2 \sqrt{-1+i}$ and $n=1-4k$,
       orthogonally projected to $\R^3$.  }
\end{figure}

\subsection{Simple factor dressing of the conformal Gauss map} 

In \cite{burstall_quintino} a simple factor dressing of constrained
Willmore surfaces is introduced: in the case when the constrained
Willmore surface is in the 4--sphere, similar to the isothermic
2--step simple factor dressing, a $d_\mu^S$--stable bundle $W$ gives a
new constrained Willmore surface. It is also shown there that the
simple factor dressing is a transformation defined in
\cite{willmore_harmonic}: the simple factor dressing of the conformal
Gauss map is, \cite{willmore_harmonic, burstall_quintino}, the
conformal immersion $\check f = r_{W,\mu}^S(f)$ given by the line
bundle
$\check L = \frac12(S + \Phi\frac{b}{a-1}\Phi\invers) L \subset
\ttrivial 2$ over the universal cover $\tilde M$ where
$\Phi = (\varphi_1, \varphi_2)$ and $\{\varphi_1, \varphi_2\}$ is a
$d_\mu^S$--parallel basis of the complex bundle $W$. In the case when
$f: M \to\R^3$ is a CMC surface and
$\varphi_1= en, \varphi_2 = e\alpha + \psi \beta$ with
$n\in\H_*, \beta=\frac 12(N\alpha(a-1) + \alpha b)$ and
$d_\mu^N\alpha=0$ we can link this simple factor dressing to the
simple factor dressing of the harmonic Gauss map $N$. In this case, we
have
\[
\Phi = F \begin{pmatrix} n & \alpha\\ 0 &\beta
\end{pmatrix}\,, 
\]
and since $S \psi = -\psi N$ and
$(\varphi_2-e\alpha) =\psi\beta\in\Gamma(\tilde L)$ we see as, in the
minimal case, \cite{sfd_lopez_ros}, that
\[
\check\varphi := (S+\Phi 
\frac{b}{a-1}\Phi\invers)(\psi \beta) = -\psi\alpha + e(\alpha\frac{b}{a-1} - n \frac{b}{a-1}
n\invers\alpha) \in\Gamma(\check L)
\]
where we use that by assumption $\alpha = N\beta-\beta\frac
b{a-1}$. Thus
\[
\check f = f -\alpha\frac{b}{a-1}\alpha\invers
\]
up to constant. Therefore, by Theorem \ref{thm:cmc simple factor} the
simple factor dressing $\check f =r^S_{W,\mu}(f)$ is the CMC surface
given by the simple factor dressing $\check f =r_{E,\mu}^N(f)$ of the
harmonic Gauss map $N$ of $f$ with $d_\mu^N$--stable bundle
$E = \alpha\C$. Moreover, the theorem also shows that
$\check f = r_{W,\mu}^S(f)$ is a $\mu$--Darboux transform of the
parallel CMC surface $g=f+N$ of $f$.

We summarise:

\begin{theorem} 
\label{thm:sfdsame}
Let $f: M \to\R^3$ be a CMC surface and $d_\lambda^{S}$ its associated
family given by the conformal Gauss map $S$. Let
$\check f = r^S_{W, \mu}(f)$ be the simple factor dressing given by
the complex rank 2 bundle $W\subset \ttrivial 2$ which is spanned by
the $d_\mu^S$--parallel sections
$en, \varphi = e\alpha+\psi\beta\in\Gamma(\ttrivial 2)$,
$n\in\H_*$. Then $\check f$ is, up to translation, the simple factor
dressing
  \[
    r^S_{W, \mu}(f)=r^N_{E, \mu}(f)
  \]
  where $E$ is the $d_\mu^N$--stable bundle given by the
  $d_\mu^N$--parallel section $\alpha$.

  Moreover, the simple factor
  dressing  $\check f $ is a
  $\mu$--Darboux transform of the parallel surface $g=f+N$,
  \[
    \check f =
    D^{-N}_{\beta, \mu}(g)\,,
  \]
  
  or, equivalently, the 2--fold Darboux transform given by Bianchi
  permutability from the Darboux transforms
  $f^\mu = D^N_{\alpha, \mu}(f)= f + \alpha\beta\invers $ and
  $g= f+N$.
  
\end{theorem}

\section{Conclusion}

In this paper we investigated the integrable systems of a CMC surface
$f: M \to\R^3$ which are given by the harmonic Gauss map $N$, by the
isothermicity of $f$ and by the conformal Gauss map $S$ in terms of
their associated families of flat connections $d_\lambda^N$,
$d_\varrho$ (and $\d_\varrho= F\cdot d_\varrho$), and $d_\lambda^S$
respectively. We showed that the parallel sections of one family of
connections give, together with the Gauss map $N$ of $f$, all parallel
sections of the other two via algebraic operations. Here there is a
2:1 correspondence of the CMC and Constrained Willmore spectral
parameter $\mu$ and the isothermic spectral parameter $\rho$ via
$\mu_\pm = 1-2\rho \pm 2 i \sqrt{\rho(1-\rho)}$ and
$\rho = -\frac{(\mu-1)^2}{4\mu}$ respectively.  This provides a
unified view on all three integrable systems of a CMC surface.

In particular, we demonstrated the correspondence between the various
associated families, Darboux transforms and simple factor dressings.

We consider CMC surfaces with $H=1$, and exclude the case of a round
sphere. Recalling that we denote by $g=f+N$ the parallel CMC surface
with Gauss map $N_g=-N$, and for $\mu\in\C_*, \rho\in \C, r\in\R_*$,
\begin{enumerate}
  \item[] a $d_\mu^N$--parallel section by $\alpha$, \\
  \item[] a $d_{\mu}^{N_g}$--parallel section by $\beta$, \\
  \item[] a $d_{\mu}^S$--parallel section by $\varphi$, \\
  \item[] a $d_\rho$--parallel section by $\phi$,\\
  \item[] a $\d_r$--parallel endomorphism by $\Phi$,\\
  \item[] a $\d_\rho$--stable  or $d_\mu$--stable complex line bundle by $E$,\\
  \item[] a $d_\mu^S$--stable complex rank 2 bundle by $W$\,,
  \end{enumerate}
  
  we summarise our results:

\begin{theorem}
  Let $f: M \to \R^3$ be a CMC surface with $H=1$ and $f$ not a round
  sphere. Then a surface $f^\alpha=A_{\alpha, s}^N(f)$,
  $\mu=e^{is}\in S^1$, in the CMC associated family of $f$ is given,
  up to M\"obius transformation, as the limit for $r$ to
  $r(\mu) = \frac 12(1-\cos s)$ of surfaces $A_{\Phi,r}(f)$ in the
  isothermic associated family of $f$, $r\in(0,1)$.

  Furthermore, $f^\alpha$ is also the limit for $\lambda$ to $\mu$ of
  the associated family $A^{S^\alpha}_{\Phi, \lambda}(f^\alpha)$ of
  $f^\alpha$, viewed as constrained Willmore surface, given by a
  $d_\lambda^{S^\alpha}$--parallel section $\Phi$:
\[
  A_{\alpha, s}^N(f) = \lim_{r\to r(\mu)} A_{\Phi,r}(f) = \lim_{\lambda\to
    \mu} A^{S^\alpha}_{\Phi, \lambda}(f^\alpha)\,.
  \]
   
  The $\mu$--Darboux transforms $D_{\alpha,\mu}^N(f)$ of $f$,
  $\mu\in\C\setminus\{0,\pm 1\}$, are exactly the isothermic
  $\varrho$--Darboux transforms $D_{\phi,\varrho}(f)$ of $f$,
  $\varrho\in\C\setminus\{0,1\}$, which have constant mean curvature
  (and hence also the isothermic simple factor dressings
  $r_{E,\varrho}(f)$ which have constant mean curvature). They also
  are given by the constrained Willmore Darboux transforms
  $D^S_{\varphi,\mu}(f)$ which have constant mean curvature, that is,
\begin{align*}
  \{ D_{\alpha,\mu}^N(f)\mid \mu\in\C\setminus\{0,\pm
  1\}\}&=\{D_{\phi,\varrho}(f) \mid D_{\phi,\varrho}(f)  \text{ has
          CMC}, \varrho\in\C\setminus\{0,1\}\} \\
&=  \{r_{E, \varrho}(f)\mid r_{E,\varrho}(f)  \text{ has
                                                  CMC},  \varrho\in\C\setminus\{0,1\}\}\\
       &=  \{D^S_{\varphi,\mu}(f) \mid D^S_{\varphi,\mu}(f)  \text{ has
          CMC}, \mu \in\C\setminus\{0,\pm 1\}\}\,.
\end{align*}

In contrast, the simple factor dressing $r^N_{E,\mu}(f)$,
$\mu\in \C\setminus\{0, \pm 1\}$, is a special case of the constrained
Willmore simple factor dressing $r^S_{W, \mu}(f)$ with
$\infty\oplus E= W$. Here $E=\varphi\C$ and
$\varphi = e\alpha+\psi\beta$ is given by a non--trivial
$d_\mu^N$--parallel section $\alpha$. These simple factor dressings
coincide with the 2--step Darboux transform of $f$ given by the
parallel CMC surface $g=D^N_{\alpha_0, \mu =-1}(f)$ and the
$\mu$--Darboux transform $D^N_{\alpha, \mu}(f)$. Finally, the simple
factor dressing $r^N_{E,\mu}(f)$ also is given by the $\mu$--Darboux
transform $D^{-N}_{\beta,\mu}(g)$ of the parallel surface:
\begin{align*}
  \{ r_{E,\mu}^N(f) \mid \mu\in \C\setminus\{0, \pm 1\}\}&=
                                                           \{D^{-N}_{\beta,\mu}(g)  \mid \mu \in \C\setminus\{0, \pm 1\} \} \\
                                                         &
                                                           = \{ r^S_{W, \mu}(f)
  \mid  W= \infty \oplus E, \mu\in \C\setminus\{0, \pm 1\}\}\,.
\end{align*}

\end{theorem}

\bibliographystyle{amsplain}
\bibliography{doc}

\end{document}